\documentclass[11pt]{svjour3}                     

 \usepackage{geometry}
\usepackage{tikz}
\usetikzlibrary{calc}
\usepackage{graphicx}%
\usepackage{multirow}%
\usepackage{amsmath,amssymb,amsfonts}%
\usepackage{mathrsfs}%
\usepackage[title]{appendix}%
\usepackage{xcolor}%
\usepackage{textcomp}%
\usepackage{manyfoot}%
\usepackage{comment}
\usepackage{booktabs}%
\usepackage{algorithm}%
\usepackage{algorithmicx}%
\usepackage{algpseudocode}%
\usepackage{listings}%
\usepackage{subcaption}
\newtheorem{assumption}{Assumption}%

\DeclareMathOperator{\conv}{conv}
\DeclareMathOperator{\epi}{epi}

\DeclareMathOperator{\cone}{cone}

\DeclareMathOperator{\cl}{cl}

\DeclareMathOperator{\R}{\mathbb{R}}
\DeclareMathOperator{\Q}{\mathbb{Q}}

\DeclareMathOperator{\for}{for}

\DeclareMathOperator{\sconv}{sconv}
\DeclareMathOperator{\Lf}{\mathcal{L}}

\newcommand{\EP}{EP}

\definecolor{links}{RGB}{204,36,29}
\usepackage[colorlinks=true,breaklinks=true,bookmarks=true,urlcolor=links,citecolor=links,linkcolor=links,bookmarksopen=false,draft=false]{hyperref}
\def\EMAIL#1{\href{mailto:#1}{#1}}
\def\URL#1{\href{#1}{#1}}         

\def\makeheadbox{}
\date{}

\begin{document}

\title{Disjunctive Submodular Functions: Envelopes and Applications to Inventory and 0-1 Quadratic Optimization}

    \titlerunning{Disjunctive Submodular Functions}


\author{ Zhongqi Wu \and  Taotao He   \and Mohit Tawarmalani
}

\authorrunning{Wu, He and Tawarmalani} 


\institute{ Zhongqi Wu \at
              Antai College of Economics and Management, Shanghai Jiao Tong University \\
              \EMAIL{wuzhongqi@sjtu.edu.cn}          
              \and
Taotao He \at
              Antai College of Economics and Management, Shanghai Jiao Tong University \\
              \EMAIL{hetaotao@sjtu.edu.cn}, \URL{https://taotaoohe.github.io/}           
           \and
           Mohit Tawarmalani \at
              Mitch Daniels School of Business, Purdue University\\
							\EMAIL{mtawarma@purdue.edu}, \URL{https://mohit.prof/}
}


\maketitle

\begin{abstract}
This paper considers convex envelopes of disjunctive submodular functions---functions that are lattice family submodular over faces of a hypercube---and constructs the first strongly polynomial algorithm for their separation when there are two facial disjunctions. Submodular functions, whose convex envelopes are characterized by the Lov\'asz extension, have occupied a fundamental role in constructing relaxations for combinatorial and nonlinear optimization problems. However, disjunctive submodular function envelopes have not been explored besides the use of ellipsoid algorithm, which remains practically intractable. Our algorithm is derived in three steps by expressing the disjunctive function as a minimum of two extended submodular functions, introducing a variable lifting technique, and constructing the sublinear envelope in the lifted space. The paper also makes several other contributions. First, we provide a disjunctive formulation for the case where each submodular function admits a linear programming formulation. Second, we derive the closed-form sublinear envelope characterization for intersecting submodular functions, yielding new structural insights into a multi-product inventory sales maximization problem. Third, we fully characterize the convex envelope of a bilinear function defined over a cycle graph in the original variable space. Finally, we show computationally that the cycle inequalities close approximately 60\% of the gap for complete and Hadamard graphs, over 30\% of the gap for complete bipartite graphs, and over 80\% of the gap for sparse graphs such as cactus and Halin graphs. The resulting relaxations are also more efficient to solve than previous extended space formulations.
			
				\keywords{  Submodular functions, Convexification, Weighted Polymatroid Intersection, Bilinear Programming, Multi-Product Newsvendor Problem }
\end{abstract}
\def \revise {\color{black}}
\def \DS {\texttt{DS }}
\def \PS {\textsf{PS }}

\section{Introduction}\label{sec1}

Strong convex relaxations are critical for the efficiency of branch-and-bound solvers in 0-1 nonlinear programming. While submodularity and disjunctive programming have independently proven to be powerful in this context, their simultaneous exploitation remains underexplored. We introduce \textit{disjunctive submodularity} as a new unifying abstraction. This framework leverages the lattice-family submodular structure of functions restricted to subdomains (faces of the 0-1 hypercube) to derive tractable, tight convex relaxations.

The convex envelope of a standard submodular function is characterized by the Lov\'asz extension~\cite{lovasz1983submodular}, with its polyhedral structure explicitly described by interpolating the function over the Kuhn triangulation of the 0-1 hypercube~\cite{tawarmalani2013explicit}. The associated separation problem reduces to linear optimization over an extended polymatroid, solvable efficiently via a greedy algorithm~\cite{edmonds1970submodular}. This tractability has established submodularity as a fundamental tool for strengthening relaxations in mixed-binary convex conic programming~\cite{atamturk2019lifted,atamturk2020submodularity,atamturk2023supermodularity,kilincc2025conic}, stochastic programming~\cite{xie2021distributionally,kilincc2022joint}, and structured nonconvex sets~\cite{yu2023generalized,xu2025submodular}. Beyond relaxation, submodular functions have found diverse applications in dynamic programming~\cite{iancu2013supermodularity}, multi-armed bandit problems~\cite{bertsimas1996conservation}, and neural network verification~\cite{tjandraatmadja2020convex}. Recent advances in submodularity recognition further broaden its potential applicability~\cite{deza2026sum}.

Parallel to these developments, disjunctive programming techniques describe the convex hull of sets representable as a disjunctive union of compact convex sets in higher-dimensional space~\cite{rockafellar1997convex,balas1979disjunctive,ceria1999convex}. This approach has yielded tight relaxations for various nonlinear functions and constraints~\cite{tawarmalani2002convex,kronqvist202550,tawarmalani2026new}. For structured sets, it has enabled the development of compact and even non-extended formulations~\cite{gunluk2010perspective,khajavirad2013convex,de2024explicit,xu2025gaining,han2025compact}.

Despite the individual strengths of these paradigms, \textit{disjunctive submodular functions}---defined as functions that are lattice-family submodular over faces of the hypercube---remain underutilized in practice. A primary obstacle is that Kuhn's triangulation no longer characterizes the convex envelope, and standard disjunctive programming formulations are exponential in size. While polynomial-time separation is theoretically possible~\cite[Corollary 2.9]{tawarmalani2013explicit}, existing methods rely on the ellipsoid algorithm, which does not scale to large instances. Specifically, this paper bridges the gap by developing a \textbf{strongly polynomial algorithm} to derive the convex envelope of $f(x_0,x)$, where $x_0\in\{0,1\}$, $x\in \{0,1\}^n$ and both $f(0,x)$ and $f(1,x)$ are submodular.

Our main contributions are summarized below:
\begin{enumerate}
    \item \textbf{Tractability Result:} We prove that the convex envelope of a disjunctive submodular function with two disjunctions admits a separation oracle solvable in strongly polynomial time (Theorem~\ref{convex-sep_of_lat}). This result rests on three key theoretical developments in Section~\ref{section:preliminaries}:
    \begin{itemize}
        \item An \textit{exact representation theorem} expressing disjunctive submodular functions as the minimum of submodular functions with equivalent separation oracle complexities (Theorem~\ref{them:extension-lattice}).
        \item A \textit{lifting result} transforming the convex envelope construction into a sublinear envelope construction problem over a truncated cone (Section~\ref{sec:from-sublinear-to-convex}).
        \item An \textit{extension result} expanding the domain of the function from the truncated cone to the full cube (Theorem~\ref{theo:extension-recovery}).
    \end{itemize}
    By combining these steps with the weighted polymatroid intersection algorithm, we establish the desired tractability.

    \item \textbf{Sublinear Envelopes:} We explore sublinear envelopes for disjunctive submodularity. Specifically, we develop the sublinear envelope for the minimum of two intersecting family submodular functions (Proposition~\ref{prop:sep_intersecting}). When one function in the minimum is linear, we describe the coefficients of the facet defining inequalities in terms of the function's Dilworth truncation. Instead of evaluating the truncation to compute each coefficient, we improve the complexity to solving $n$ submodular minimization problems (Theorem~\ref{theo:facet-1}). 

    \item \textbf{Value Function of Inventory Optimization:} We relate the value function of a multi-product newsvendor problem to the sublinear envelope of a disjunctive submodular function. This approach generalizes recent characterizations for sales maximization models~\cite{punt2025multi} by reframing the problem over intersecting families. We provide a closed-form sublinear envelope characterization in the space of the original problem variables (Corollary~\ref{coro:regular}).

    \item \textbf{0-1 Bilinear Programming:} We consider 0-1 bilinear functions whose interaction graphs (where variables $x_i$ and $x_j$ are connected if they appear in a bilinear term together) form cycles. We provide the first closed-form convex envelope characterization for this class of functions in the space of the original problem variables (Theorem~\ref{theo:cycle-envelope}). This construction is based on a general LP formulation for separating $K$ submodular functions, assuming each separation oracle is a polynomial-sized LP (Theorem~\ref{theorem:LP}). This yields several advantages:
    \begin{itemize}
        \item \textit{Compactness:} Our formulations require variables that scale linearly with the original variables. In contrast, existing literature relies on lifted-space formulations based on odd-cycle inequalities and $0\text{-}\frac{1}{2}$ Chv\'{a}tal-Gomory cuts~\cite{padberg1989boolean,bonami2018globally,gupte2020extended}, which grow quadratically for dense graphs.
        \item \textit{Effectiveness:}  For dense graphs, relative to McCormick relaxation, our inequalities close approximately 60\% of the gap for complete and Hadamard instances, and over 30\% for bipartite instances. For sparse graphs (cactus and Halin), the improvement is more pronounced, closing over 80\% of the gap when longer cycle lengths are used for separation.
        \item \textit{Efficiency:} Our method is computationally more efficient than the baseline McCormick relaxation because it avoids introducing auxiliary variables for bilinear terms, leading to faster solution times despite additional cutting-plane iterations.
    \end{itemize}
\end{enumerate}

The paper is organized as follows. Section~\ref{section:preliminaries} presents the theoretical foundations of our framework, reducing the convex envelope of a disjunctive submodular function to the convex and sublinear envelopes of the minimum of submodular functions. Section~\ref{section:env_minimum} studies the convex and sublinear envelopes of pointwise minima of $K$ submodular functions under various settings. Applications are discussed in Section~\ref{section:app}, covering multi-product inventory (Section~\ref{sec:inventory}) and bilinear functions over cycle graphs (Section~\ref{sec:bilinear}). Finally, Section~\ref{sec:computation} reports computational results highlighting the efficacy of the derived inequalities for 0-1 bilinear programming.

\textbf{Notation.} We use $\R$ to denote the set of real numbers and $\R_+$ the set of non-negative real numbers. For a set $S \subseteq \R^n$, $\conv(S)$ denotes the convex hull of $S$, and $\cone(S)$ denotes the convex conic hull of $S$. For a function $f:S\subseteq\R^n\to\R$, $\epi(f)$ denotes the epigraph of $f$. For any positive integer $n$, let $[n]:=\{1,\ldots,n\}$. We use $\mathbf{0}$ and $\mathbf{1}$ to denote the all-zero and all-one vectors, respectively. Moreover, $e_i$ denotes the $i^{\text{th}}$ standard basis vector of $\R^n$. We omit the dimension $n$ when it is apparent from the context.

\section{Preliminaries}
\label{section:preliminaries}
{\revise
This section presents the foundations of our framework. Section~\ref{sec:disjunctive-submodular} establishes an equivalence between disjunctive submodular functions and pointwise minimum of finitely many submodular functions. Section~\ref{sec:from-sublinear-to-convex} introduces sublinear and convex envelopes and develops a lifting technique for recovering the convex envelope from the sublinear envelope of a lifted function. Finally, Section~\ref{section:lifting_dsf} combines these results to derive the convex envelope of a disjunctive submodular function from the sublinear envelope of its pointwise minimum representation.
}

\subsection{Disjunctive Submodularity and Pointwise Minimum Representations}
\label{sec:disjunctive-submodular}

We will show that  a function defined on a union of lattice families, where it is submodular on each constituent family, can be represented as the pointwise minimum of globally submodular functions. A set \(X\subseteq\{0,1\}^n\) is a \textit{lattice family} if it is closed under componentwise maximum (\(x\vee y\)) and minimum (\(x\wedge y\)). A function \(g:X\to\mathbb{R}\) is \textit{submodular} on \(X\) if 
\[
g(x)+g(y) \geq g(x\vee y)+g(x\wedge y) \quad \for\: x, y\in X.
\]
A function is supermodular if its negation is submodular. The following examples illustrate how non-submodular functions may admit disjunctive formulations via lattice families, and how these relate to the pointwise minimum of submodular functions.

\begin{example}[Cycle Bilinear Function]
\label{ex:cycle-disjunction}
Consider \(f(x) = \sum_{i=1}^{n-1}a_i x_ix_{i+1} + x_1x_n\) for \(x\in\{0,1\}^n\), with \(a_i<0\). The term \(x_1x_n\) is supermodular, so \(f\) is not globally submodular. However, consider the faces
\[
X^0 := \bigl\{ x\in\{0,1\}^n \bigm| x_1=0 \bigr\} \quad \text{ and } \quad
X^1 := \bigl\{ x\in\{0,1\}^n \bigm| x_1=1 \bigr\}.
\]
Let $q(x) :=\sum_{i=1}^{n-1}a_i x_ix_{i+1}$, which is globally submodular. On \(X^0\), \(f(x)=q(x)\), and on \(X^1\), \(f(x)=q(x)+x_n\). Thus, \(f\) is submodular on each face. Since \(x_1x_n=\min\{x_1,x_n\}\) for binary variables, we have
\begin{equation*}
\label{eq:cycle-min-representation}
f(x) = q(x)+\min\{x_1,x_n\} =\min\bigl\{q(x)+x_1,\, q(x)+x_n\bigr\}.
\end{equation*}
Both terms in the minimum are globally submodular. \qed
\end{example}

\begin{example}[Signed Cut Function]
\label{ex:signed-cut-disjunction}
Let \(G=([n],E)\) have edge weights \(w_{ij}\in\mathbb{R}\), and consider \(f(x) = \sum_{(i,j)\in E}w_{ij}|x_i-x_j|\). For \(w_{ij}\geq 0\), the term $w_{ij}|x_i-x_j|$ is submodular and for \(w_{ij}<0\), it is supermodular. Let \(E^- = \bigl\{ (i,j)\in E \bigm| w_{ij}<0 \bigr\}\). For any partition \(E^-_1\cup E^-_2 = E^-\) with \(E^-_1\cap E^-_2 = \emptyset\), define
\begin{equation*}
\label{eq:signed-cut-lattice}
X(E^-_1,E^-_2) := \left\{ x\in\{0,1\}^n \;\middle|\;
\begin{aligned}
x_i&\geq x_j\quad \forall (i,j)\in E^-_1\\
x_i&\leq x_j\quad \forall (i,j)\in E^-_2
\end{aligned}
\right\}.
\end{equation*}
Each \(X(E^-_1,E^-_2)\) is a lattice family because order constraints are preserved under componentwise min/max. On this set, negative-weight terms become linear (e.g., \(|x_i-x_j|=x_i-x_j\) if \(x_i\geq x_j\)), while nonnegative terms remain submodular. Thus, \(f\) is submodular on each lattice family. As the partition ranges over all orientations of \(E^-\), these families cover \(\{0,1\}^n\). Theorem~\ref{them:extension-lattice} below shows that \(f\) can be represented as the minimum of submodular extensions, each agreeing with \(f\) on its respective lattice family. \qed
\end{example}

Consider a function $f: X \to \R$ defined on a lattice family $X \subseteq \{0,1\}^n$. Suppose that $X = \bigcup_{k=1}^K X_k$, where each $X_k \subseteq X $ is a lattice family and $f$ is submodular on every $X_k$. We refer to this class of functions as \textit{disjunctive submodular functions}. In Theorem~\ref{them:extension-lattice}, we show that any disjunctive submodular function can be represented as the pointwise minimum of submodular functions. The key idea is to construct, for each $k$, a submodular extension of the restriction $f|_{X_k}$ that agrees with $f$ on $X_k$ and dominates  $f$ elsewhere.  To this end, we introduce a retraction operator $\rho_k: X \to X_k$ satisfying:
\[
\rho_k(x)\in X_k \quad \text{ and } \quad \rho_k(x)=x \quad \for  x\in X_k. 
\]
With this, for each $k$, we define an extension function $f_k: X \to \R$ as follows
\begin{equation}\label{eq:extension-lattice} 
f_k(x) = f\bigl(\rho_k(x)\bigr) + C \bigl\|x-\rho_k(x) \bigr\|_1 \quad \for  x\in X,
\end{equation} 
where \(C \geq f^U - f^L\) with \(f^L \leq f(x) \leq f^U\) on \(X\), which can be computed efficiently by deriving lower and upper bounds for $f$ over the lattice families $X_k$, respectively~\cite{fujishige2005submodular}. The retraction operator ensures that $f_k$ coincides with $f$ on $X_k$ while the penalty term guarantees that $f_k$ dominates $f$ on $X \setminus X_k$. 

To ensure that each extension function is submodular, we construct a particular retraction operator from a given representation of the lattice family $X_k$. Recall that each lattice family admits a representation in terms of fixed coordinates and precedence constraints~\cite[Prop. 10.3.3]{groetschel1988geometric}. Specifically, for each $k$, there exist index sets $I^0_k, I^1_k$ and a reflexive and transitive relation $D_k \subseteq [n] \times [n]$ such that 
\begin{equation}\label{eq:lattice-rep}
X_k = \bigl\{ x \in \{0,1\}^n \bigm| x_i=0 \; \for i \in I_k^0,\ x_i =1 \; \for  i\in I_k^1,\ x_i \geq x_j \; \for \; (i,j) \in D_k  \bigr\},
\end{equation}
which will be referred to as a pre-order representation of a lattice family. Without loss of generality, we assume that $(i,j) \in D_k$ and $i \in I^0_k$ imply $j \in I^0_k$, and $(i,j) \in D_k$ and $j \in I^1_k$ imply $i \in I^1_k$.  Given this representation, define the greatest and least elements of $X_k$ as 
\[
\top_k := \sum_{i \notin I^0_k} e_i \quad  \text{ and } \quad \bot_k :=  \sum_{i \in I_k^1} e_i,
\]
respectively. In addition, associated with $D_k$  we introduce a closure operator $c_k$ which maps a vector  $ x \in  \{0,1\}^n$ to a vector $ y \in \{0,1\}^n$  so that:
\[
y_i =  \max_{j:(i,j) \in D_k} x_j \qquad \for i \in [n]. 
\]
The retraction operator is then defined as
\begin{equation}\label{eq:retraction}
\rho_k(x) := c_k(x \wedge \top_k) \vee \bot_k. 
\end{equation}
The construction of $\rho_k$ can be interpreted as a sequence of corrections that map an arbitrary point $x \in X$ into the lattice family. First, the operation $x \wedge \top_k$ removes all coordinates that violate the fixed-zero constraints. Next, the closure operator $c_k$ ensures that the resulting vector satisfies the precedence constraints induced by $D_k$. Last, the join with $\bot_k$ enforces the fixed-one constraints. {\revise This particular choice of retraction, together with a sufficiently large constant $C\geq f^U-f^L$,  preserves the submodularity of each extended function, leaving the minimum unchanged as we show next.}

\begin{theorem}\label{them:extension-lattice}
Let \(f\) be defined on a lattice family \(X=\bigcup_{k=1}^K X_k\) and be submodular on each \(X_k\) defined as in~\eqref{eq:lattice-rep}. Let $\rho_k$ be the retraction defined as in~\eqref{eq:retraction}, and $f_k$ be the extension function given by~\eqref{eq:extension-lattice}. Then, each $f_k$ is submodular on $X$ and 
\[ 
f(x) = \min_{k\in [K]} f_k(x)  \quad \for x\in X. 
\] 
\end{theorem} 
\begin{proof} 
See Appendix~\ref{proof:theo-extension}. \qed
\end{proof} 

A key special case occurs when each \(X_k\) is a face of the hypercube. More precisely, for every \(k\in[K]\), there exist disjoint index sets \(I_k^0,I_k^1\subseteq[n]\) such that 
\begin{equation}\label{eq:lattice-face}
    X_k = \left\{ x\in\{0,1\}^n \;\middle|\; x_i=0 \text{ for all } i\in I_k^0,\ x_i=1 \text{ for all } i\in I_k^1 \right\}. 
\end{equation}
Equivalently, \(X_k\) admits the representation in~\eqref{eq:lattice-rep} with index sets $I_k^0,I_k^1$ and relation set $ D_k=\{(i,i):i\in[n]\}$. For this choice of \(D_k\), the associated closure operator satisfies $c_k(x) = x$ for $x\in \{0, 1\}^n$. Therefore, the retraction operator simplifies to \(\rho_k(x)=(x\wedge \top_k)\vee \bot_k\), and~\eqref{eq:extension-lattice} becomes:
\begin{equation}\label{eq:face-proj} 
\begin{aligned}
    f_k(x)&= f\bigl((x\wedge \top_k)\vee \bot_k\bigr) + C\|(x\wedge \top_k)\vee \bot_k - x \|_1\\
    &=f\bigl((x\wedge \top_k)\vee \bot_k\bigr) + C\Bigl( \sum_{i\in I^0_{k}}x_i + \sum_{i\in I^1_{k}}(1-x_i) \Bigr).
\end{aligned}
\end{equation} 
The second equality follows because the retraction fixes the coordinates in $I_k^0$ and $I_k^1$ to 0 and 1, respectively, while leaving all other coordinates unchanged.

\begin{corollary}\label{cor:extension-face}
Let \(f\) be defined on a lattice family \(X=\bigcup_{k=1}^K X_k\) and be submodular on each \(X_k\) defined as in~\eqref{eq:lattice-face}. Let $f_k$ be the extension function given by~\eqref{eq:face-proj}. Then, each $f_k$ is submodular on \(X\) and 
\[
f(x)=\min_{k\in [K]} f_k(x)\quad \for x\in X.
\]
\end{corollary} 
\begin{proof} 
The result follows from Theorem~\ref{them:extension-lattice}, since for each face \(X_k\),
the closure operator \(c_k\) is the identity function and the corresponding extension reduces to~\eqref{eq:face-proj}.\qed
\end{proof}

Theorem~\ref{them:extension-lattice} shows that every disjunctive submodular function can be represented as the pointwise minimum of finitely many submodular functions. The converse also holds, that is, any function that can be expressed as the pointwise minimum of finitely many submodular functions is a disjunctive submodular function.

\begin{remark}

For $k \in [K]$, let $f_k:X\rightarrow\mathbb{R}$  be a submodular function on a lattice family $X\subseteq\{0,1\}^n$, and define $f(x) = \min_{k\in [K]}f_k(x)$. We introduce binary variables $\delta \in \{0,1\}^K$ to indicate $f_k$. More specifically, let 
\[
\Delta:=  \bigl\{ \delta \in \{0,1\}^K \bigm| \delta_1 = 1,\ \delta_1 \geq \delta_2 \geq \cdots \geq \delta_K  \bigr\}. 
\]
The set $\Delta$ is a lattice family whose elements are precisely $v^k:=\sum_{i=1}^k e_i$, $k \in [K]$. Now, consider a function  $g:X\times \Delta \to  \mathbb{R}$ defined as 
\[
g(x,\delta):=
\delta_1f_1(x)+
\sum_{k=2}^K \delta_k\bigl(f_k(x)-f_{k-1}(x)\bigr). 
\]
Clearly, $X \times \Delta$ is a lattice family since the Cartesian product of lattice families is again a lattice family. For each $k \in [K]$,  $g(x,v^k)$ telescopes to $f_k(x)$. Hence, $g$ is submodular on the lattice family $X\times\{v^k\}$ for every $k\in[K]$. Since $X\times\Delta
=
\bigcup_{k=1}^K
\left(X\times\{v^k\}\right)$,  it follows that $g$ is a disjunctive submodular function on $X\times\Delta$. Moreover, minimizing $g$ over the auxiliary indicator variables recovers the original pointwise-minimum function. Equivalently, the epigraph of $f$ is the corresponding projection of the epigraph of $g$. \qed

\end{remark}

\subsection{From sublinear envelopes to convex envelopes}\label{sec:from-sublinear-to-convex}
In this subsection, we present the notion of sublinear and convex envelopes and construct convex envelopes through sublinear envelopes. 
A function $p: X \to \R$, where $X$ is a conic convex subset of $\R^n$, is called sublinear if it is positively homogeneous, \textit{i.e.}, $p(r x) = rp(x)$, for all $x \in X$ and $r \geq 0$, and subadditive, \textit{i.e.}, $p(x+y) \leq p(x) + p(y)$ for $x,y \in X$. Clearly, a sublinear function is a convex function, but the converse does not hold. 
\begin{definition}
    For a function $f: S \subseteq \R^n \to \R$, the function $p: \cone(S) \to \R$ is the sublinear envelope of $f(\cdot)$ over $\cone(S)$ if 
    \begin{enumerate}
        \item $p(\cdot)$ is a sublinear function over $\cone(S)$,
        \item $p(x) \leq f(x)$ for all $x \in S$,
        \item  if $g(\cdot)$ is any sublinear function satisfying $g(x) \leq f(x)$ for all $x \in S$, $g(x)\leq p(x)$ for all $x \in \cone(S)$.
    \end{enumerate}
    We denote the sublinear envelope of $f(\cdot)$  over $\cone(S)$ by $\sconv_{\cone(S)}(f)$. When the set $S$ is clear from the context, we omit the subscript and simply write $\sconv(f)$. 
\end{definition}
\noindent In other words, $\sconv(f)$ is the largest sublinear underestimator of $f$ over $\cone(S)$.  This concept is closely related to the convex envelope, denoted as $\conv(f)$, of $f$ over $\conv(S)$, which is the largest convex underestimator of $f$ over $\conv(S)$. However, $\sconv(f)(x) \leq \conv(f)(x)$ for $x \in \conv(S)$ since $\sconv(f)(\cdot)$ must be sublinear and, therefore, convex.
Even if $f$ is positively homogeneous, $\conv(f)(x) \ne \sconv(f)(x)$; for example, consider $\sqrt{x_1x_2}$ over $[0,1]^2$. To guarantee the properness of $\sconv(f)$, namely, to exclude the value $-\infty$, we impose the following assumption.
\begin{assumption}\label{ass:proper}
There exists a linear function $l:\cone(S)\to \R$ such that $l(x)\leq f(x)$, for all $x\in S$.
\end{assumption}

We begin by presenting two  representations of the sublinear envelope that will be used throughout the paper. The dual representation expresses the sublinear envelope through conic combinations of points in the original domain and is analogous to the classical convex combination representation of the convex envelope~\cite[Proposition 2.31]{rockafellar1998variational}:
\begin{equation}\label{eq:envdual}
\conv(f)(x) = \inf_{\lambda} \biggl\{\sum_{v \in V}\lambda_vf(v) \biggm| V \subseteq S,\  x = \sum_{v \in V}v \lambda_v,\ \sum_{v \in V}\lambda_v =1,\  \lambda_v \geq 0 \for v \in V \biggr\}.     
\end{equation}
The primal representation characterizes its closure as the pointwise supremum of all linear underestimators of the original function.

\begin{lemma}\label{lemma:dualsconv}
    Let $f: S \subseteq \R^n \to \R$ be a function satisfying Assumption~\ref{ass:proper}. Then,  for $x \in \cone(S)$, we have 
    \[
    \begin{aligned}
        \sconv(f)(x) &= \inf_{\lambda} \Bigl\{\sum_{v \in V}\lambda_vf(v) \Bigm| x = \sum_{v \in V}v \lambda_v,\ V \subseteq S,\  \lambda_v \geq 0 \for v \in V \Bigr\} \\
    \end{aligned}
    \]
    Moreover, for $x \in \cone(S)$
    \begin{equation*}
        \cl\sconv(f)(x) = \sup \bigl\{ \langle \alpha,x \rangle  \bigm| \langle \alpha, v \rangle \leq f(v) \, \for v \in S \bigr\}.
\end{equation*}
\end{lemma}
\noindent We will primarily consider the case in which $S\subseteq\{0,1\}^n$. In this setting, Assumption~\ref{ass:proper} is satisfied whenever $f(\mathbf{0})\geq 0$. Moreover, both optimization problems in Lemma~\ref{lemma:dualsconv} can be formulated as linear programs and are feasible. Therefore, by strong duality of linear programming, the infimum in the first formulation coincides with the supremum in the second. Hence, for every $x\in\cone(S)$, $\sconv(f)(x)=\cl\sconv(f)(x)$. In the following remark, we discuss that a pre-order representation of a lattice family can be simplified in the context of envelope construction, which will be used throughout the paper.
\begin{remark}\label{rmk:lattice-family}
Consider a function $f: S \subseteq \{0,1\}^n \to \R$, where $S$ is a lattice family admitting  a pre-order representation given by index sets $I^0,I^1 \subseteq [n]$ and a reflexive and transitive relation $D\subseteq[n]\times[n]$. We simplify the pre-order representation for both convex and sublinear envelopes.

\medskip
\noindent
\textbf{Convex envelope.}
For the purpose of computing the convex envelope of $f$, we may, without loss of generality, assume that $I^0=I^1=\emptyset$ and that $D$ is a partial order. Indeed, coordinates fixed to either $0$ or $1$ can be eliminated by restricting $f$ to the corresponding face of the hypercube. If $i\in I^0 \cup I^1 $, we can restrict $f$ to the face $x_i =0$ or $x_i=1$ and delete this coordinate. It follows from the dual representation in~\eqref{eq:envdual} that the convex envelope of the original function can be recovered from the convex envelope of the resulting lower-dimensional function. 
 We may also assume that $D$ is antisymmetric. Suppose, otherwise, that there exist distinct $i,j\in[n]$ such that $(i,j)\in D$ and $(j,i)\in D$. The corresponding precedence constraints imply $x_i = x_j$ for every $x \in S$. Thus, the two coordinates can be identified and replaced by a single coordinate. After relabeling, suppose that $i=1$ and $j=2$. We then define  the reduced function
\[
\tilde{f}(x_2,x_3,\ldots,x_n)
:=
f(x_2,x_2,x_3,\ldots,x_n),
\]
on the reduced lattice family $\tilde{S}:=\bigl\{(x_2, x_3, \ldots, x_n ) \in \R^{n-1} \bigm| \exists x_1 \text{ s.t. } (x_1, x_2, \ldots, x_n) \in S  \bigr\}$. 
It follows from the dual representation in~\eqref{eq:envdual} that for $x \in \conv(S)$ we have $\conv(f)(x) = \conv(\tilde{f})(x_2, \ldots, x_n)$. Repeating these reductions eventually yields an equivalent representation in which $I^0=I^1=\emptyset$ and $D$ is a partial order.  Moreover,  these reductions preserve the submodularity of $f$. 

\medskip
\noindent
\textbf{Sublinear envelope.}
For the purpose of computing the sublinear envelope, we can no longer fix coordinates to $1$. We may continue to assume, without loss of generality, that $I^0=\emptyset$ and $D$ is a partial order. However, a coordinate $i\in I^1$ cannot, in general, be deleted. Specifically, a linear underestimator for the function in the higher-dimensional space becomes an affine underestimator once a coordinate is fixed to a non-zero value and is, therefore, not guaranteed to be dominated by the sublinear envelope throughout the domain of the restricted function. Thus, for the sublinear envelope, we retain all coordinates in $I^1$. \qed
\end{remark}

We now establish the connection between sublinear and convex envelopes through a lifting construction. Consider a function $f:X\subseteq\{0,1\}^n\to\mathbb{R}$. We define the lifted domain $\Lf\subseteq\{0,1\}^{n+1}$ as  $\Lf:=\{(0,\mathbf{0})\}\cup \bigl\{(1,x)\bigm| x\in X \bigr\}$, and the corresponding lifted function $F:\Lf\to\mathbb{R}$ by 
\begin{equation} \label{eq:pers-lift} 
F(0,\mathbf{0})=0 \qquad \text{ and } \quad F(1,x)=f(x)\quad \text{for }x\in X. 
\end{equation} 
Thus, the original function $f$ is embedded in the face $t=1$ of the lifted space, while the origin $(0,\mathbf{0})$ is assigned value zero. The following result shows that the convex envelope of $f$ can be recovered by restricting the sublinear envelope of its lift to this face.
\begin{proposition} \label{prop:pers-lift} 
Consider a function $f:X\subseteq\{0,1\}^n\to\mathbb{R}$, and let $F:\Lf\to\mathbb{R}$ be its lift defined by~\eqref{eq:pers-lift}. Then, $\conv(f)(x)=\sconv(F)(1,x)$ for all $x\in \conv(X)$.
\end{proposition}
\begin{proof}
For any $x\in\conv(X)$, we have
\begin{equation*}
    \begin{aligned}
    \conv(f)(x) &= \min_{\lambda}
\biggl\{
\sum_{v\in X}\lambda_v f(v)
\biggm| 
\sum_{v\in X}\lambda_v v = x,\;
\sum_{v\in X}\lambda_v=1,\;
\lambda_v\geq 0 \, \for v\in X
\biggr\}\\
&=\min_{\rho, \lambda}
\biggl\{\rho F(0, \mathbf{0})+
\sum_{v\in X}\lambda_v F(1, v)
\biggm| 
(1,x)=\rho (0, \mathbf{0}) + \sum_{v\in X}\lambda_v (1, v)
, \rho \geq 0, \lambda_v\geq 0\:\for v\in X
\biggr\}\\
&=\sconv(F)(1, x),
    \end{aligned}
\end{equation*}
where the first equality follows from~\eqref{eq:envdual} and $X$ contains a finite number of points, the second equality holds from $F(1,x)=f(x)$ for all $x\in X$, and  the third equality follows from Lemma~\ref{lemma:dualsconv} and the definition of $F$. \qed
\end{proof}

\subsection{Envelopes of Disjunctive Submodular Functions}\label{section:lifting_dsf}
{\revise The focus of this subsection is to derive the convex envelope of a disjunctive submodular function from the sublinear envelope of its pointwise minimum representation. Using extension and lifting, we transform a disjunctive submodular function $f$ to the pointwise minimum of submodular functions $\bar{F}_k$, defined as in~\eqref{eq:extension-sub}. The key step is to establish that the sublinear envelope of $\min_{k}\{\bar{F}_k\}$ indeed produces the convex envelope of $f$ (see Theorem~\ref{theo:extension-recovery}). Although the domain of $f$ may be a lattice family, the domain of $\bar{F}_k$ is chosen to be $\{0,1\}^{n+1}$. This choice enables us to invoke efficient combinatorial algorithms for sublinear envelopes, as discussed in Section~\ref{sec:subm}.
 }

{\revise
Consider a disjunctive submodular function  $f:X \subseteq \{0,1\}^n\to \R$ defined on  a lattice family $X$. To construct its convex envelope, Remark~\ref{rmk:lattice-family} allows us to assume, without loss of generality, that $X$ admits a partial order presentation, that is, there exists a partial order $D$ such that  
\[
X=\bigl\{x \in \{0,1\}^n \bigm| x_i\geq x_j \, \for\ (i,j) \in D \bigr\}. 
\]
This partial order structure is preserved under lifting. In particular,  let $F:\Lf\to\R$ denote the lifted function of $f$ defined as in~\eqref{eq:pers-lift}. Its domain admits a partial order representation given as follows:
\[
\Lf = \bigl\{(t, x) \in \{0,1\}^{n+1} \bigm| t \geq x_i \, \for\ i \in [n] \text{ and } x_i \geq x_j \, \for\ (i,j) \in D \bigr\}.
\]
Note that $\Lf$ is still a strict subset of $\{0,1\}^{n+1}$.

By the disjunctive submodularity of $f$, Theorem~\ref{them:extension-lattice} allows us to represent $f$ as the minimum of $K$ submodular functions $f_k: X \to \R$, $k \in [K]$. Consequently, the lifted function admits an analogous representation, $F(t,x) = \min_{k \in [K]}F_k(t,x)$, where $F_k:\Lf\to\R$ is the lift of $f_k$ given as follows:
\begin{equation*}
F_k(0,\mathbf{0})=0 \qquad \text{ and  } \quad F_k(1,x)=f_k(x)\quad \text{for }x\in X. 
\end{equation*}
Each $F_k$ is submodular on the lattice family $\Lf$. Indeed, for two points on the face $t=1$, the submodularity of $F_k$ reduces to the corresponding property of $f_k$ on $X$, while for any pair involving $(0,\mathbf{0})$, the inequality holds with equality.

It remains to extend the domain $\mathcal{L}$ of each $F_k$ to the full box $\{0,1\}^{n+1}$. We use the retraction operator introduced in Theorem~\ref{them:extension-lattice}. Specifically, let $\rho:\{0,1\}^{n+1}\to\Lf$ be the retraction operator associated with $\Lf$, defined as in~\eqref{eq:retraction}. For each $k \in [K]$, let $\bar{F}_k:\{0,1\}^{n+1}\to\mathbb{R}$, defined as follows:
\begin{equation}
\label{eq:extension-sub}
\bar{F}_k(z) := F_k\bigl(\rho(z)\bigr) + C\|z-\rho(z)\|_1,
\end{equation}
where $C=(n+1)^2(F^U-F^L)$ with $F^L\leq F_k(t, x)\leq F^U$ for $(t, x)\in\Lf$ and $k\in [K]$. Observe that the constant here is different than the one required for Theorem~\ref{them:extension-lattice}.  The choice of $C$ serves two purposes here. First, Theorem~\ref{them:extension-lattice} guarantees that the extension is submodular when $C$ is sufficiently large, in particular when $C\geq F^U-F^L$. Second, the factor $(n+1)^2$ provides the bound needed to ensure that the sublinear envelope over $\mathcal{L}$ is unaffected by the extension to the box $\{0,1\}^{n+1}$. 

The following theorem establishes that the sublinear envelope of the pointwise minimum of the extended functions recovers the convex envelope of $f$.
The detailed proof is provided in Appendix~\ref{proof:theo-extension-recover}.

\begin{theorem}
\label{theo:extension-recovery}
Let $f:X \to \R$  be a disjunctive submodular function defined on a lattice family $X$, with representation $f(x)=\min_{k\in[K]}f_k(x)$, where each $f_k :X \to\mathbb{R}$ is submodular. Define
\[
\bar{F}(t,x):=\min_{k\in[K]}\bar{F}_k(t,x) \qquad \for \, (t,x) \in \{0,1\}^{n+1}
\]
where each $\bar{F}_k$ is given by~\eqref{eq:extension-sub}. Then, each $\bar{F}_k$ is submodular, and
\[
\conv(f)(x)=\sconv(\bar{F})(1,x)
\qquad \for \, x\in\conv(X).
\]
\end{theorem}
\begin{remark}\label{rmk:extension_hypercube}
When the domain $X$ is $\{0,1\}^n$, we  can simplify the expression for the extended functions. In this case, the lifted domain is $\Lf = \{(t,x)\in\{0,1\}^{n+1}\mid t\cdot \mathbf{1}\geq x\}$.  Therefore, we obtain 
\[
\rho(t,x) = \bigl(\max\{t,x_1,\ldots,x_n\},x\bigr) \quad \text{ and } \quad\|(t,x)-\rho(t,x)\|_1=\max\{t,x_1,\ldots,x_n\}-t
\]
Hence, the extended function defined in~\eqref{eq:extension-sub} reduces to
\begin{equation}\label{eq:extension-hypercube}
\bar F_k(t,x) =
\max\{t,x_1,\ldots,x_n\}f_k(x) + C\bigl(\max\{t,x_1,\ldots,x_n\}-t\bigr),
\quad \for k\in[K],
\end{equation}
where $C=(n+1)^2(F^U-F^L)$, with $F^L\leq F_k(t, x)\leq F^U$ for $(t, x)\in\Lf$ and $k\in [K]$.\qed 
\end{remark}
}

\section{Envelopes of the minimum of submodular functions}\label{section:env_minimum}
{\revise
The theoretical developments in Section~\ref{section:preliminaries}, particularly Theorems~\ref{them:extension-lattice} and~\ref{theo:extension-recovery}, reduce the convex envelope of a disjunctive submodular function to the sublinear envelope of a suitably chosen minimum of submodular functions. However, these reductions do not, by themselves, guarantee that the resulting problems are computationally tractable. The geometry of the envelope of this minimum is significantly more complex than that of a single submodular function, as we discuss below.
}

 Specifically, it is well known that the Lov\'asz extension, described via the Kuhn triangulation~\cite{de2010triangulations}, describes the convex envelope of a single submodular function $f$ and this triangulation does not yield the envelope when $f$ does not satisfy this property~\cite{lovasz1983submodular}. Moreover, for a submodular function $f:\{0, 1\}^n\rightarrow \R$ with $f(\mathbf{0}) = 0$, this convex envelope coincides with the sublinear envelope of $f$ because each of the simplices in the Kuhn triangulation contains $\mathbf{0}$ as a vertex. It follows that the minimum of submodular functions, a function that is not itself submodular, is not convexified by restricting attention to the Kuhn's triangulation. We show using a concrete example that the subdivision associated with the envelope has a much richer structure, sometimes matching the sublinear envelope using a different triangulation, while in other cases, departing from the sublinear envelope.

\begin{example}\label{ex:single_point_violation}
Consider the submodular function $ f(x) = -x_1x_3 - 2x_2x_3 $ defined on $ x \in \{0,1\}^3 $. It is straightforward to verify that $ f $ is a submodular function. Its convex envelope is given by the Lovász extension as follows:
\[
\conv(f)(x) = \max\{-x_1 - 2x_3,\; -2x_2 - x_3,\; -x_1 - 2x_2,\; -3x_3\}, \quad \for x \in [0,1]^3,
\]
where each affine function is obtained by interpolating $f$ over vertices of a polytope within the polyhedral subdivision shown in Figure~\ref{fig:origianl}. Note that this subdivision can be further refined into the Kuhn triangulation. 

We now introduce a parametric family of functions $g^k:\{0,1\}^3 \to \R$ by perturbing the value of $f$ at the vertex $(1,1,1)$. Specifically,  for $k \geq 0$, let
\[
g^k(x) := f(x) + k x_1 x_2 x_3. 
\]
In particular, for $0 < k <3$, $g^k$ violates submodularity. To see this, we consider the vectors  $ x' = (1,0,1) $ and $ x'' = (0,1,1) $ and observe that 
\[
g^k(x') + g^k(x'') = -3 < -3 + k = g^k(x' \vee x'') + g^k(x' \wedge x''),
\]
which contradicts the submodularity condition. Nevertheless, for any $0\leq k \leq 2$, $g^k(x)$ is facially submodular function and can be expressed as the minimum of two submodular functions:
 \[
 g^k(x) = \min\{k x_1,\; k x_2x_3\} + f(x) = \min \bigl\{k x_1 - x_1x_3 - 2x_2x_3,\; -x_1x_3 - (2-k)x_2x_3\bigr\}. 
 \]
Using the \texttt{Julia} package \texttt{Polyhedra.jl}~\cite{legat2023polyhedral}, we obtain the convex envelopes of $ g^1 $ and $ g^2 $ as follows:
\[
\begin{aligned}
    &\conv(g^1)(x) = \max\{-x_1 - 2x_2,\; -2x_3,\; -x_2 - x_3\} ,\\
    &\conv(g^2)(x) = \max\{-x_1 - 2x_2,\; -2x_3,\; -x_2 - x_3,\; x_1 - x_3 - 1\}.
\end{aligned}
\]
As with $f$, the convex envelope of $g^1$ is positive homogeneous, and thus coincides with its sublinear envelope. However, the polyhedral subdivision associated with the envelope of $g^1$, depicted in Figure~\ref{fig:increase-1}, cannot be refined into the Kuhn triangulation due to the presence of a simplex that contains both $(0,1,1)$ and $(1,0,1)$ but not their join $(1,1,1)$. For the convex envelope of $g^2$, the corresponding subdivision, given in  Figure~\ref{fig:increase-2}, contains a polytope that does not contain the origin. This structural shift explicitly illustrates the discrepancy between convex and sublinear envelopes. \qed

    \begin{figure}[htbp]
    \centering
    \begin{tikzpicture}
        \node (left) at (0,0) {
            \begin{subfigure}[t]{0.28\textwidth}
                \centering
                \includegraphics[width=\textwidth]{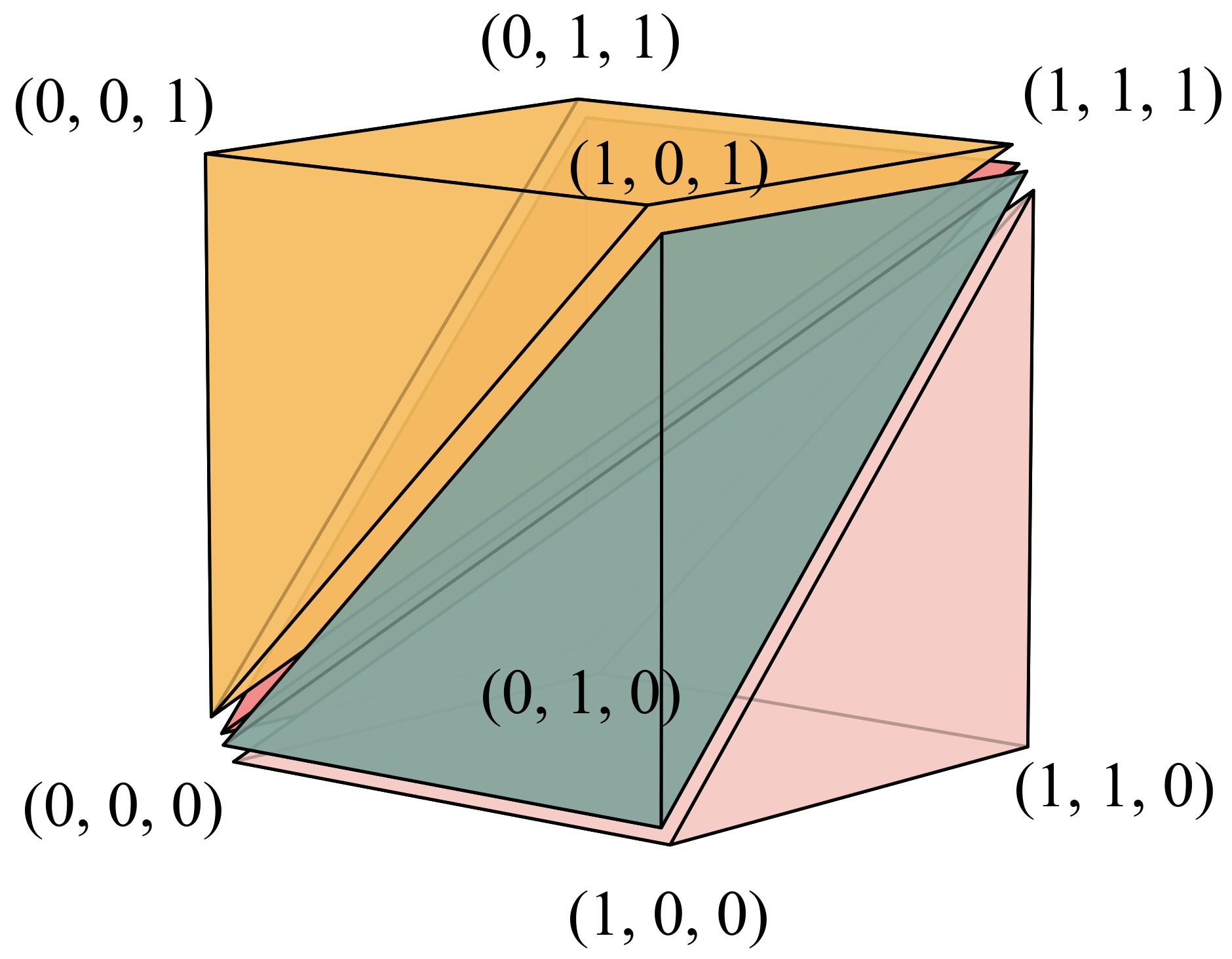}
                \caption{Subdivision of $f$.}
                \label{fig:origianl}
            \end{subfigure}
        };

        \node (middle) at (5.4,0) {
            \begin{subfigure}[t]{0.28\textwidth}
                \centering
                \includegraphics[width=\textwidth]{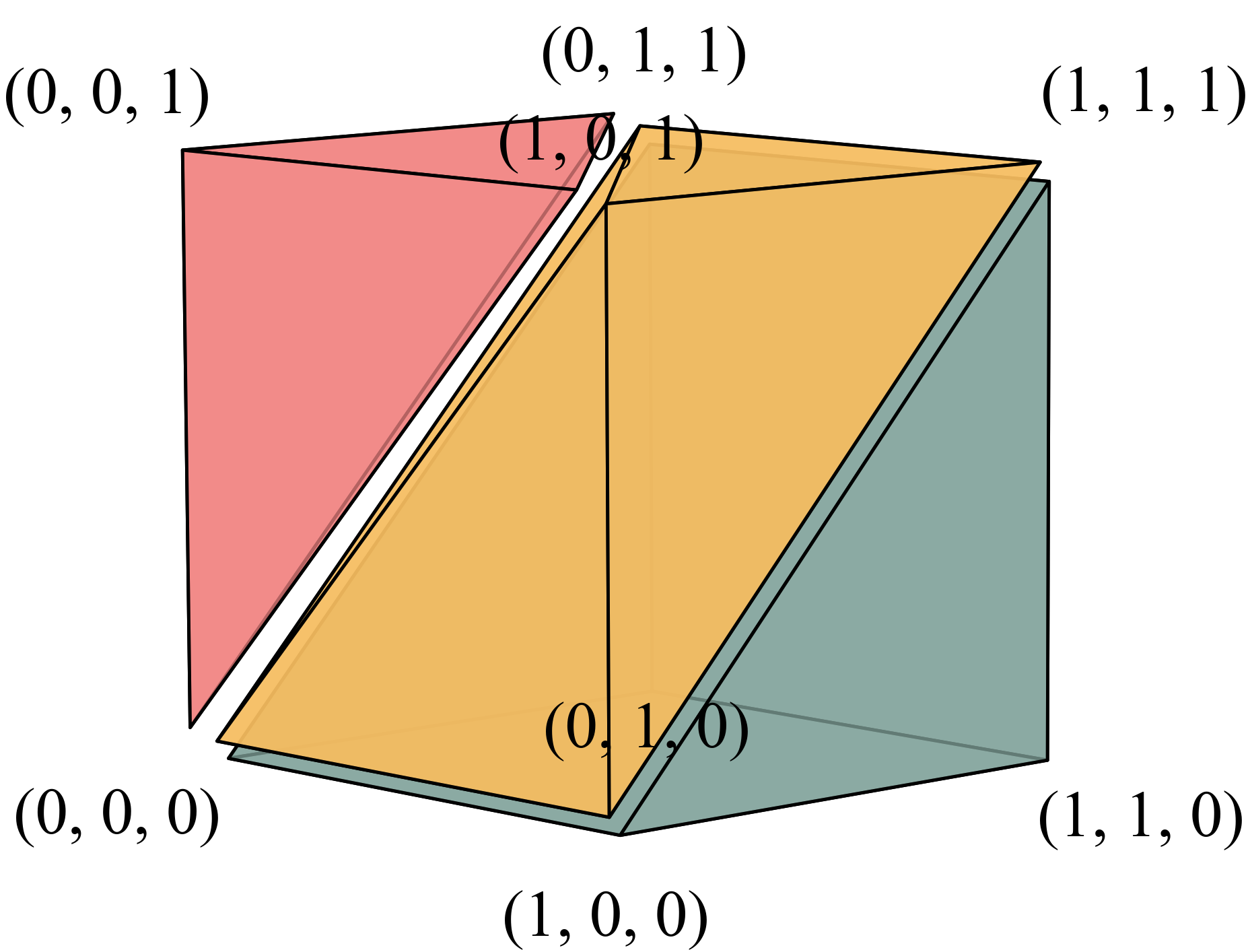}
                \caption{Subdivision of $g^1$.}
                \label{fig:increase-1}
            \end{subfigure}
        };

        \node (right) at (10.8,0) {
            \begin{subfigure}[t]{0.28\textwidth}
                \centering
                \includegraphics[width=\textwidth]{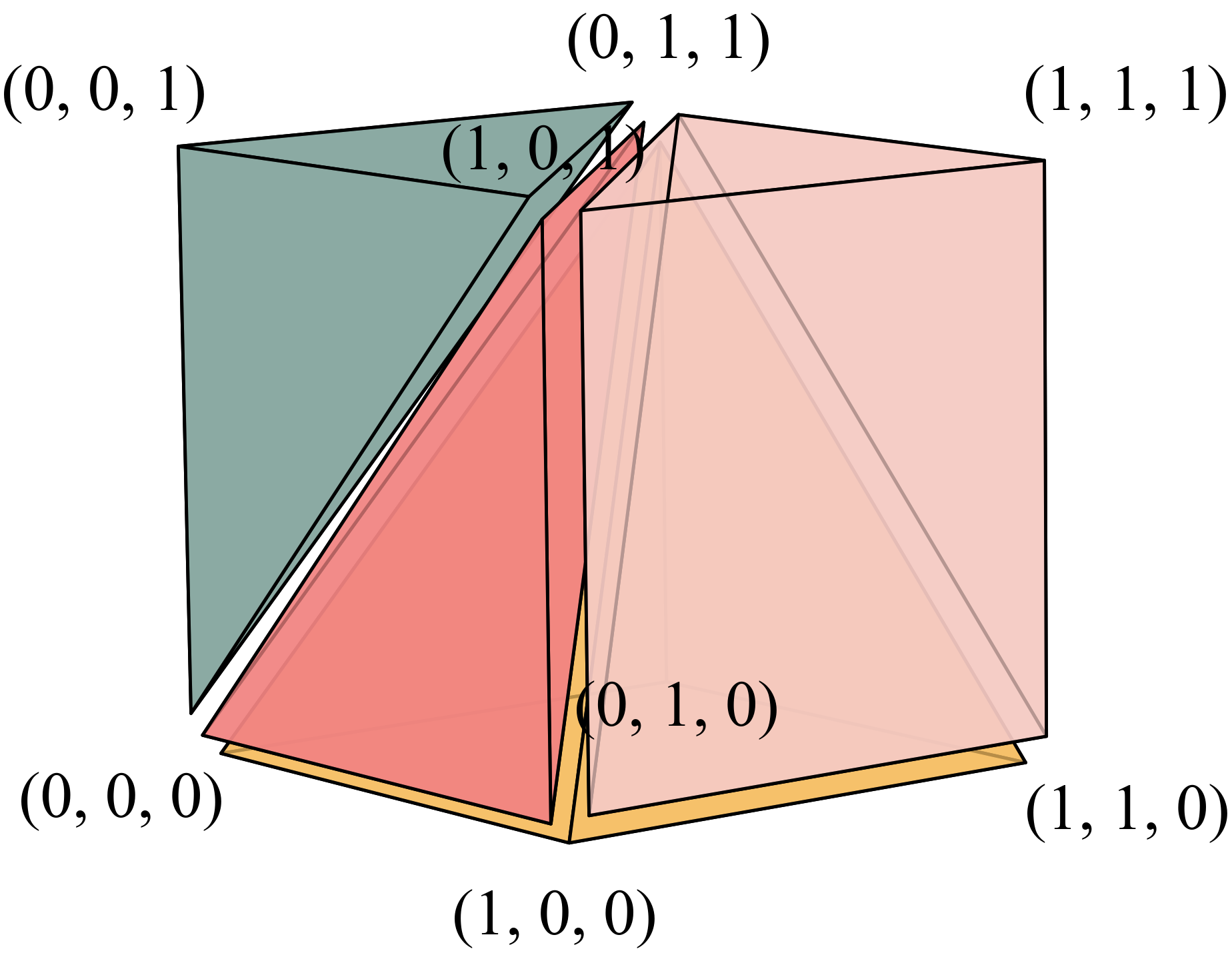}
                \caption{Subdivision of $g^2$.}
                \label{fig:increase-2}
            \end{subfigure}
        };

        \draw[->, thick] (left.east) -- (middle.west)
            node[midway, above, font=\scriptsize] { };

        \draw[->, thick] (middle.east) -- (right.west)
            node[midway, above, font=\scriptsize] { };
    \end{tikzpicture}
    
\caption{Polyhedral subdivisions of the unit hypercube $[0,1]^3$ induced by the convex envelopes of $f$, $g^1$, and $g^2$. As the perturbation parameter $k$ increases, the subdivision transitions from a classic Kuhn triangulation structure (a) to a configuration where the triangulation does not refine to the Kuhn's triangulation even though the convex envelope matches the sublinear envelope (b), and finally to one containing a polytope that excludes the origin (c), illustrating that convex and sublinear envelopes can be different.}    \label{fig:change_subdivision}
\end{figure}
\end{example}

In this section, we study the convex and sublinear envelopes of pointwise minima of $K$ submodular functions under various settings. In particular, Section~\ref{sec:subm} develops efficient combinatorial separation algorithms for the case $K=2$, addressing the challenge highlighted in Example~\ref{ex:single_point_violation}. Section~\ref{sec:inter} provides explicit envelope descriptions for the case $K=2$ when one of the functions is linear. Finally, Section~\ref{section:LP} presents compact linear programming (LP) formulations to separate the envelopes, assuming such formulations are available for each function $f_k$.

\subsection{Fast separation via  combinatorial algorithms}\label{sec:subm}
{\revise
In this subsection, we focus on devising fast separation algorithms for the pointwise minimum of two submodular functions. In Section~\ref{section:sublinear_oracle}, we use weighted polymatroid intersection (WPI) algorithms to construct a sublinear separation oracle. Section~\ref{section:conv_lattice} uses this separation oracle to resolve the problem of convexifying disjunctive submodular functions. Section~\ref{section:conv_intersecting} expands the sublinear envelope construction to a larger class of functions that define an intersecting family. Besides expanding the result to more general settings, our study of intersecting families allow us to construct sublinear envelope of a pointwise minimum of a submodular function and a linear function in closed-form. This structure appears in applied settings such as the inventory application discussed in Section~\ref{sec:inventory}.
}

\subsubsection{Sublinear Separation Oracle via Weighted Polymatroid Intersection}\label{section:sublinear_oracle}
Before presenting our separation results, we review some relevant results on the extended polymatroid. For a function $g:X \subseteq \{0,1\}^n \to \R$, we define the extended polymatroid associated with $g$ as:
\[
\EP_g:=\bigl\{ \alpha \in \R^n \bigm| \langle \alpha, v \rangle \leq g(v)\: \for v \in X \bigr\}. 
\]
Although the term \textit{extended polymatroid} is typically reserved for submodular functions, we use the notation $\EP_g$ for an arbitrary set function $g$ throughout this paper. This slight abuse of terminology simplifies the notation and will cause no ambiguity in our results. Given a weight vector $w \in \R^n$ and a pair of submodular functions $f_1, f_2$ mapping from $\{0,1\}^n$  to $\R$, the weighted polymatroid intersection problem finds a vector of maximum weights in the intersection of two extended polymatroids
\begin{equation}\label{eq:WPI}
    \max \bigl\{ \langle w, \alpha \rangle \bigm|  \alpha \in \EP_{f_1} \cap \EP_{f_2} \bigr\}. \tag{\textsc{WPI}}
\end{equation}
For solving~\eqref{eq:WPI}, \cite{cunningham1985primal} propose a primal--dual polynomial-time combinatorial algorithm that relies on an efficient submodular minimization oracle. Subsequently, \cite{frank1987application} transform this polynomial-time combinatorial procedure into a strongly polynomial-time algorithm via Simultaneous Diophantine Approximation. Later, \cite{fujishige1989strongly} further improve the approach by developing a purely combinatorial strongly polynomial-time algorithm, which utilizes tree projection and cost scaling techniques to solve a sequence of submodular flow problems with bounded cost coefficients. More recently, \cite{lee2015faster} introduce a cutting-plane method to significantly improve the theoretical running time for solving~\eqref{eq:WPI}.

Now, we are ready to present a sublinear separation oracle. To do so, we interpret a vector $\alpha \in \R^n$ in~\eqref{eq:WPI} as defining a linear function on $\R^n$, and accordingly treat the feasible region as the set of all possible linear underestimators of $\min\{f_1,f_2\}$. In other words,~\eqref{eq:WPI} finds a linear underestimator of $\min\{f_1,f_2\}$ whose value at  $w$ is maximized. This interpretation allows us to devise an efficient separation oracle for the sublinear envelope of the minimum of two submodular functions. Formally, for a convex function $g:Y\subseteq\R^n\to\R$ and a point $(\bar y,\bar\mu)$ with $\bar y\in Y$, the separation problem consists of either concluding that $\bar\mu\geq g(\bar y)$, or finding an affine underestimator $\alpha^\intercal y+\beta$ of $g$ satisfying $\bar\mu<\alpha^\intercal\bar y+\beta$.

\begin{proposition}\label{prop:sep_sublinear}
Consider $f:\{0,1\}^n \to \R$ defined as $f(x) := \min \bigl\{f_1(x), f_2(x)\bigr\}$ for $x \in \{0,1\}^n$, where, for $k=1,2$, $f_k$ is submodular and satisfies $f_k(\mathbf{0})\geq 0$. 
For a given point $(\bar{x}, \bar{\mu}) \in \Q^{n+1}$ with $\bar{x}\geq 0$, the separation problem for the sublinear envelope of $f$ reduces to solving~\eqref{eq:WPI} with a weight vector $\bar{x}$, a problem that can be solved in strongly polynomial time. 
\end{proposition}
\begin{proof}
    For a given point $(\bar{x}, \bar{\mu}) \in \Q^{n+1}$ with $\bar{x} \geq 0$, we solve~\eqref{eq:WPI} with weight vector $\bar{x}$, obtaining an optimal solution $\alpha^*$ and an optimal value $v^*$. This can be done in strongly polynomial time~\cite[Corollary 47.4d]{schrijver2003combinatorial}. Since $\{0,1\}^n$ is a finite set and $f(\mathbf{0})=\min\{f_1(\mathbf{0}),f_2(\mathbf{0})\}\ge 0$, Lemma~\ref{lemma:dualsconv} implies that $v^*=\sconv(f)(\bar{x})$. Thus, if $\bar{\mu} \geq v^*$ we can conclude that $\bar{\mu} \geq \sconv(f)(\bar{x})$. Otherwise, the optimal solution $\alpha^*$ yields a desired separating linear function, that is,  $\bar{\mu}< \langle \alpha^*, \bar{x}\rangle  $  and $\langle \alpha^*, x\rangle \leq f(x)$ for $x \in \{0,1\}^n$ which implies $\langle \alpha^*, x\rangle \leq \sconv(f)(x)$ for $x \in \R^n_+$. 
\end{proof}

\subsubsection{On the case of a Lattice family}\label{section:conv_lattice}
{\revise
For the minimum of two submodular functions, $f_1$ and $f_2$,  on a lattice family $X$, Algorithm~\ref{alg:sep-lattice} separates a given point $(\bar{x}, \bar{\mu})$ from  the epigraph of $\conv(f)$ in strongly polynomial time. For $k=1,2$, let $\bar F_k:\{0,1\}^{n+1}\to\R$ be the extension of $f_k$ defined as in~\eqref{eq:extension-sub}, and let $\bar F(t,x):=\min\{\bar F_1(t,x),\bar F_2(t,x)\}$. The algorithm invokes Proposition~\ref{prop:sep_sublinear} to separate the lifted point $\bigl((1,\bar{x}),\bar{\mu}\bigr)$ from the sublinear envelope of $\bar{F}$. The resulting inequality is then transformed to the original space separating the convex envelope of $f$. The correctness of the algorithm follows from Theorem~\ref{theo:extension-recovery}. 
}


\begin{algorithm}[!htbp]
\caption{Separation of the convex envelope over a lattice family}
\label{alg:sep-lattice}
\begin{algorithmic}[1]
\Require $f(x) = \min\{f_1(x),f_2(x)\}$ with $f_1, f_2:X \subseteq \{0,1\}^{n}\to\R$; point $(\bar x,\bar\mu)\in\mathbb{Q}^{n+1}$ with $\bar{x} \in \conv(X)$.
\Ensure A separation oracle for the convex envelope of $f$ with the given point $(\bar{x},\bar{\mu})$.
\State Set $\bar F_k\leftarrow \text{ the extension of $f_k$ given by~\eqref{eq:extension-sub}}$.
\State Set $\bar F\leftarrow\min\{\bar F_1,\bar F_2\}$.
\State Call the separation oracle in Proposition~\ref{prop:sep_sublinear} to separate $\sconv(\bar{F})$ at the point $\bigl((1,\bar x),\bar\mu\bigr)$. 
\If{$\bar{\mu} \geq \sconv(\bar{F})(1,\bar{x}) $}
    \State \Return $\bar{\mu} \geq \conv(f)(\bar{x})$.
\Else 
\State Let $\bigl\langle (\alpha_0, \alpha), (t,x) \bigr\rangle $ be the optimal linear function derived from \eqref{eq:WPI} as in Proposition~\ref{prop:sep_sublinear}. 
\State \Return the affine function $\langle \alpha, x\rangle + \alpha_0$
\EndIf
\end{algorithmic}
\end{algorithm}

\begin{theorem}\label{convex-sep_of_lat}
Consider $f(x):=\min\{f_1(x),f_2(x)\}$, where $f_k:X\to\R$ is submodular on a lattice family $X$. Algorithm~\ref{alg:sep-lattice} solves the separation problem of $\conv(f)$ in strongly polynomial time.
\end{theorem}

\begin{proof}
It follows from Theorem~\ref{theo:extension-recovery} that $\bar{F}_k$ is submodular. Thus, we can use the separation oracle given by Proposition~\ref{prop:sep_sublinear} to separate $\sconv(\bar{F})$. If $\bar{\mu} \geq \sconv(\bar{F})(1,\bar{x})$ then we can conclude $\bar{\mu} \geq \sconv(\bar{F})(1,\bar{x}) = \conv(f)(\bar{x})$, where the equality follows from Theorem~\ref{theo:extension-recovery}. Otherwise, let $\langle (\alpha_0,\alpha) ,(t,x)\rangle$ be the linear function returned by the oracle.  Then, $\bar{\mu} < \langle (\alpha_0,\alpha) ,(1,\bar{x})\rangle = \alpha_0 + \langle \alpha, \bar{x}\rangle$. Moreover, for $x\in\conv(X)$, 
\[
\alpha_0+\langle\alpha,x\rangle
\leq \sconv(\bar F)(1,x)
=\conv(f)(x),
\]
where the equality follows from Theorem~\ref{theo:extension-recovery}. Hence, $\mu\geq\alpha_0+\langle\alpha,x\rangle$ is valid for the epigraph of $\conv(f)$ and is violated by $(\bar x,\bar\mu)$. Therefore, Algorithm~\ref{alg:sep-lattice} solves the separation problem for $\conv(f)$. Last, the computational complexity of Algorithm~\ref{alg:sep-lattice} follows that of the sublinear separation oracle in Proposition~\ref{prop:sep_sublinear}. \qed
\end{proof}
The unconstrained case, in which $X=\{0,1\}^n$, follows directly as a special case of Theorem~\ref{convex-sep_of_lat}. In this setting, we construct the extended functions $\bar{F}_k$ using~\eqref{eq:extension-hypercube} from Remark~\ref{rmk:extension_hypercube},  rather than~\eqref{eq:extension-sub}.
%
%

\subsubsection{On the case of an Intersecting Family}\label{section:conv_intersecting}
{\revise
We consider the case where the domain is an intersecting family. A set $X \subseteq \{0,1\}^n$ is called an intersecting family if, for any $x,y \in X$, one has
\[
\text{if } x \wedge y \neq 0, \text{ then } x \vee y \in X, x \wedge y \in X.
\]
A function $g:X \to \R$ is called an intersecting submodular function if,
\[
g(x)+g(y)\ge g(x\vee y) + g(x\wedge y) \quad \for x,y \in X \text{ with }  x \wedge y \neq 0. 
\]
For intersecting submodular functions, we can reduce sublinear-envelope separation to the optimization problem~\eqref{eq:WPI} by using extension techniques from~\cite{schrijver2003combinatorial}. However, such extension techniques do not allow construction of convex envelope even utilizing the lifting framework of Section~\ref{sec:from-sublinear-to-convex}. This is because, unlike in the lattice-family setting, the lifting does not preserve intersecting submodularity. The following example demonstrates this obstruction.}
\begin{example}
Consider the function $f : \{0,1\}^2 \to \mathbb{R}$ defined by $f(x)=x_1x_2$. We first observe that $f$ is intersecting submodular. Indeed, for any $x,y\in\{0,1\}^2$ satisfying $x\wedge y\neq 0$, either $x=y$ or one of the two vectors is necessarily equal to $(1,1)$. Hence,
\[
f(x)+f(y)=f(x\vee y)+f(x\wedge y).
\]
Thus, $f$ is intersecting submodular on $\{0,1\}^2$. We now consider the lift of $f$. Let $\Lf:=\{(0,\mathbf{0})\}\cup\{(1,x)\mid x\in\{0,1\}^2\}$, and define the lifted function $F:\Lf\to\mathbb{R}$ by $F(0,\mathbf{0})=0$ and $F(1,x)=f(x)$ for $x\in\{0,1\}^2$. $F$ is not intersecting submodular. To see this, consider $x=(1,1,0)$ and $y=(1,0,1)$. We have $x\wedge y=(1,0,0)\neq 0$ and $x\vee y=(1,1,1)$. However,
\[
F(x)+F(y)=f(1,0)+f(0,1)=0
<1=f(1,1)+f(0,0)
=F(x\vee y)+F(x\wedge y).
\]
Hence, the lifted function $F$ is not intersecting submodular. \qed
\end{example}

{\revise Regardless, the sublinear envelope is independently of interest in this setting. We develop a separation procedure for this envelope for a pointwise minimum of two intersecting submodular functions $f_1$ and $f_2$ defined on an intersecting family $X\subseteq\{0,1\}^n$. The main idea is to extend the domain to the full hypercube. We assume, in addition, that either $\mathbf{0}\notin X$ or $f_k(\mathbf{0})\geq 0$ for $k=1,2$. This condition ensures that the extended functions have well-defined sublinear envelopes.}




The extension consists of two steps. The first step is to extend each intersecting submodular function into a lattice family submodular function using the \textit{Dilworth truncation}~\cite[Section 49.6]{schrijver2003combinatorial}. Consider all possible  sums of elements in an intersecting family $X$, that is,
\[
\check{X}:= \biggl\{\sum_{v\in V}v \biggm | V\subseteq X,\ \sum_{v\in V}v\le \mathbf{1}\biggr\}. 
\]
The Dilworth truncation of an intersecting submodular function $g: X \to \R$ is given as follows:
\begin{equation}\label{eq:dilworth-truncation}
    \check{g}(x) := \min_{V\subseteq X}
    \left\{\sum_{v\in V} g(v) \;\middle|\; \sum_{v\in V} v=x \right\}
    \quad\for x\in\check{X},
\end{equation}
where by convention, the sum over an empty index set is $\mathbf{0}$, and thus $\mathbf{0} \in \check{X}$ and $\check{g}(\mathbf{0}) = 0$. For $k=1,2$, let $\check{f}_k$ denote the Dilworth truncation of $f_k$. By Theorem 49.4 of~\cite{schrijver2003combinatorial}, $\check{X}$ is a lattice family and $\check{f}_k$ is submodular on $\check{X}$.

The second step is to use the construction in Theorem~\ref{theo:extension-recovery} to extend $\check{f}_k$ into a submodular function on $\{0,1\}^n$. More specifically, let $\rho:\{0,1\}^n\to\check X$ be the retraction associated with $\check X$, defined as in~\eqref{eq:retraction}, and constant $C=n^2(f^U-f^L)$ with $f^L\leq\check f_k(x)\leq f^U$ for $x\in\check X$ and $k=1,2$. For $k=1,2$, define
\begin{equation}\label{eq:intersecting-extension}
\bar f_k(x):=
\check f_k\bigl(\rho(x)\bigr)
+C\|x-\rho(x)\|_1,
\quad\for x\in\{0,1\}^n. 
\end{equation}

{\revise 
The Dilworth truncation allows extension of each $f$ to $\check{f}$ maintaining the sublinear envelope.
Then, Theorem~\ref{theo:extension-recovery} is used to construct $\bar{f}_k$, which is submodular on $\{0,1\}^n$ and satisfies $\bar{f}_k(\mathbf{0})=\check{f}_k(\mathbf{0})=0$ for $k=1,2$. Proposition~\ref{prop:sep_sublinear} can be used to derive the sublinear envelope of the pointwise minimum of $\bar{f}_1$ and $\bar{f}_2$. This construction then provides a sublinear envelope of $\min\{f_1,f_2\}$, as we show next. The detailed proof is included in Appendix~\ref{proof:sep_intersecting}.}
\begin{proposition}\label{prop:sep_intersecting}
Consider $f(x):=\min\{f_1(x),f_2(x)\}$, where $f_k$ is an intersecting family submodular on  $X$ with either $\mathbf{0} \notin X$ or $f_k(\mathbf{0}) \geq 0$. Then,  for a given point $(\bar x,\bar\mu)\in \Q^{n+1}$ with $\bar x\in\cone(X)$, the separation problem for $\sconv(f)$ reduces to~\eqref{eq:WPI} for $\bar{f}_1$ and $\bar{f}_2$ with weight vector $\bar x$, where $\bar{f}_k$ is defined as in~\eqref{eq:intersecting-extension}. 
\end{proposition}
{\revise Algorithms for solving~\eqref{eq:WPI} require an efficient evaluation oracle for each $\bar f_k$. Such an oracle can be constructed from an evaluation oracle for $\check f_k$, together with a preorder representation of $\check X$, as detailed in Appendix~\ref{proof:sep_intersecting}.
}
\subsection{Explicit Sublinear Envelope when One of Two Submodular Functions is Modular}\label{sec:inter}
{\revise
In this subsection, we consider the special case in which one of the two submodular functions is modular. Although the combinatorial algorithm developed in the previous subsection can be used to separate the corresponding envelope, we obtain a sharper characterization of the sublinear envelope by deriving its explicit closed-form expression. Our approach exploits a hidden intersecting family structure that arises in this class of functions. More specifically, Proposition~\ref{prop:same_envelope} reduces the problem to computing the sublinear envelope of an intersecting submodular function, which is addressed in Proposition~\ref{prop:intersecing-facet-1}. Theorem~\ref{theo:facet-1} then combines these two results to establish an explicit characterization of the sublinear envelope for this minimum.
}

Suppose that  $f_1(x) = a^\intercal x$ is a modular function defined on $\{0,1\}^n$ for some $a \in \R^n$. Let $L \subseteq \{0,1\}^n$ be a lattice family, and let $f_2:\{0,1\}^n \to \R\cup\{+\infty\}$ be such that
\begin{equation}\label{eq:condonf2}
     f_2(x) \text{ is submodular on } L \quad \text{ and } f_2(x) = + \infty \for x \notin L,  
\end{equation}
 and define $f(x) = \min \bigl\{f_1(x),f_2(x)\bigr\}$ for $x \in \{0,1\}^n$. We first enlarge the domain $L$ by adding all unit vectors. Let
 $X:=L \cup \bigl\{ e_i \mid i \in [n] \bigr\}$, and we define a function $g:X \to \R$ as follows:
\begin{equation}\label{eq:inter}
\begin{aligned}
g(x) =
\begin{cases}
\min \bigl\{ f_1(x), f_2(x) \bigr\} & \text{if } x \in \bigl\{ e_i \mid i \in [n] \bigr\}, \\
f_2(x) & \text{otherwise}.
\end{cases}
\end{aligned}
\end{equation}
The following proposition shows that this construction produces an intersecting submodular function without changing the sublinear envelope. 

\begin{proposition}\label{prop:same_envelope}
Consider a function $f(x) = \min\{f_1(x), f_2(x)\}$, where $f_1(x) = a^\intercal x$ for $x \in \{0,1\}^n$, and $f_2:\{0,1\}^n \to \R\cup\{+\infty\}$ satisfying conditions in~\eqref{eq:condonf2}. Then, the extension function $g$ given by~\eqref{eq:inter} is an intersecting submodular function and
\[
\sconv(f)(x) = \sconv(g)(x) \quad \for \,  x \in \R_+^n.
\]
\end{proposition}
\begin{proof}
We first show that $g$ is an intersecting submodular function over $X:=L \cup \bigl\{ e_i \mid i \in [n] \bigr\}$. Let $E = \{e_i \mid i \in [n]\}$. Consider any $u, v \in X$ such that $u \wedge v \neq \mathbf{0}$. We distinguish two cases. First, we assume that at least one of $u$ or $v$ belongs to $E$, and, without loss of generality, we can further assume that $u = e_i$ for some $i \in [n]$. Since $u \wedge v \neq \mathbf{0}$, the $i^{\text{th}}$ coordinate of $v$ must be equal to $1$, implying that $ v \geq e_i$. Consequently, $u \wedge v = e_i = u \in X$ and $u \vee v = v \in X$. Thus, the submodularity inequality holds trivially with equality. If neither $u$ nor $v$ belongs to $E$, then $u, v \in X \setminus E \subseteq L$. Since $L$ is a lattice family, it is closed under both join and meet. Therefore, $u \wedge v \in L \subseteq X$ and $u \vee v \in L \subseteq X$, and we have
\[
g(u) + g(v)
= f_2(u) + f_2(v)
\ge f_2(u \vee v) + f_2(u \wedge v)
\ge g(u \vee v) + g(u \wedge v),
\]
where the equality holds by the definition of $g$ in~\eqref{eq:inter}, and the inequality holds since $f_2$ is submodular over the lattice family $X$, and the last inequality follows from the fact that $g(x) \le f_2(x)$ for $ x \in X$.

We now prove the equivalence $\sconv(f)(x) = \sconv(g)(x)$ for $x\in \R_+^n$. By Lemma~\ref{lemma:dualsconv}, the sublinear  envelope $\sconv(f)$ admits the representation
\[
\sconv(f)(x) = \max_{\alpha}
\left\{
\alpha^\intercal x
\:\middle|\:
\begin{aligned}
    &\alpha^\intercal v \le a^\intercal v
     \quad\for v \in \{0,1\}^n \\
    &\alpha^\intercal v \le f_2(v)
    \quad \for v \in L
\end{aligned}
\right\}.
\]
Clearly, the constraints $\alpha^\intercal v \le a^\intercal v $ for $ v \in \{0,1\}^n,$
are equivalent to the coordinate-wise inequalities $\alpha_i \le a_i$ for $i \in [n]$. Indeed, if $\alpha^\intercal v \le a^\intercal v $ for $ v \in \{0,1\}^n$, then by choosing $v = e_i$ for each $i \in [n]$, we immediately obtain $\alpha_i \le a_i$. Conversely, suppose that $\alpha_i \le a_i$ for $ i \in [n]$. For any $v \in \{0,1\}^n$, we have
\[
\alpha^\intercal v
=
\sum_{i \in [n]:\, v_i = 1} \alpha_i
\le
\sum_{i \in [n]:\, v_i = 1} a_i
=
a^\intercal v,
\]
where the inequality follows that $\alpha_i\leq a_i$ for $i\in [n]$.  Therefore, for every $x \in \R_+^n$,
\[
\begin{aligned}
\sconv(f)(x)
&= \max_{\alpha}
\left\{
\alpha^\intercal x
\:\middle|\:
\begin{aligned}
    &\alpha^\intercal e_i \le a^\intercal e_i 
     \for i \in [n] \\
    &\alpha^\intercal v \le f_2(v) 
     \for v \in L
\end{aligned}
\right\}
\\
&= \max_{\alpha}
\left\{
\alpha^\intercal x
\:\middle|\:
\alpha^\intercal v \le g(v) 
\for v \in X
\right\}
\\
&= \sconv(g)(x),
\end{aligned}
\]
where the second equality follows directly from the definition of $g$, and the last equality follows from Lemma~\ref{lemma:dualsconv} together with the fact that $\cone(X) = \R_+^n$. This completes the proof. \qed
\end{proof}

As a consequence, it suffices to derive the sublinear envelope of the intersecting submodular function $g$ in~\eqref{eq:inter}. We next develop a general result for intersecting submodular functions. Let $g:X \to\R$ be an intersecting submodular function, where  $X$ is an intersecting family containing $\{e_i\mid i\in[n]\}$. Consider its Dilworth truncation $\check{g}: \check{X} \to \R$, as defined in~\eqref{eq:dilworth-truncation}.  
By Proposition~\ref{prop:lattice-extension-1} in Appendix~\ref{App:lat_extension},  the truncation  preserves the sublinear envelope of $g$.  Moreover, by Theorem 49.4 of~\cite{schrijver2003combinatorial}, it is a submodular function on $\check{X}$, which is $\{0,1\}^n$ since $e_i \in X$ for $i\in [n]$. These two properties allow us to derive the sublinear envelope of $g$ via a sorting-based greedy algorithm. 

 For a given $\bar{x} \in \cone(\check{X}) = \R^n_+$, let $\sigma$ be a permutation  of coordinates of $\bar{x}$ such that   $\bar{x}_{\sigma(1)} \geq \bar{x}_{\sigma(2)} \geq \cdots \geq \bar{x}_{\sigma(n)}$. In addition, let $\alpha^*$ be a vector in $\R^n$ such that 
\begin{equation}\label{eq:kuhnDual}
\alpha^*_{\sigma(i)} = \check{g}(v^\sigma_i) - \check{g}(v^\sigma_{i-1}) \qquad \for i \in [n], 
\end{equation}
where $v^\sigma_0 = \mathbf{0}$ and $v^\sigma_i = v^\sigma_{i-1}+e_{\sigma(i)}$ for $i \in [n]$. It turns out that  the sublinear envelope of $g$ at a given point
$\bar{x}\in\cone(X) = \R^n_+$ can be evaluated by the greedy procedure, that is, 
\begin{equation}\label{eq:intersecting-greedy}
\sconv(g)(\bar{x}) = \sconv(\check{g})(\bar{x}) = \max \bigl\{\langle \alpha,\bar{x}\rangle \bigm| \alpha \in \EP_{\check{g}} \bigr\} = \langle \alpha^*, \bar{x}\rangle,
\end{equation}
where the first equality follows from Proposition~\ref{prop:lattice-extension-1}, the second equality follows from Lemma~\ref{lemma:dualsconv},  and the third equality holds since the submodularity of $\check{g}$ implies that $\alpha^*$ is optimal~\cite{edmonds1970submodular}. Computing $\alpha^*$ requires $n$ evaluations of the value oracle for $\check{g}$. Each
evaluation of $\check{g}$ can, in turn, be obtained by solving at most $n$ lattice family submodular minimization problems~\cite[Section 49.7]{schrijver2003combinatorial}. Hence, the
sublinear envelope of $g$ can be computed
through a sequence of at most $n^2$ lattice family submodular minimization problems.

{\revise
Next, Proposition~\ref{prop:intersecing-facet-1} shows that $\alpha^*$, defined in \eqref{eq:kuhnDual}, can be computed without explicitly evaluating $\check{g}$. Instead, $n$ submodular minimization problems suffice to compute $\alpha^*$. This will be apparent when we show that $\mu\in\R^n$ defined below equates $\alpha^*$. To begin, we compute $\mu$ recursively in the order specified by $\sigma$ by specifying the $n$ problems claimed above. We start with $\beta^0=\mathbf{0}$. For each $t\in[n]$, we compute the $\sigma(t)^{\text{th}}$ coordinate of $\mu$ by solving the following lattice family submodular minimization problem}
\begin{equation}\label{eq:intersecting-facet}
\mu_{\sigma(t)}: = \min_{y\in\{0,1\}^n}
\Bigl\{ g\bigl(e_{\sigma(t)}+y\bigr) -   \langle \beta^{t-1},  y \rangle \Bigm| y\leq v^\sigma_{t-1},\ e_{\sigma(t)}+y\in X
\Bigr\},
\end{equation}
and then update $\beta^t$ as follows:
\[ 
\beta^t:=\beta^{t-1}+\mu_{\sigma(t)}e_{\sigma(t)}. 
\] 
{\revise The feasible region of~\eqref{eq:intersecting-facet} restricts $X$ to the $\sigma(t)^{\text{th}}$ coordinate slice and imposes an additional upper bound on the remaining coordinates with the constraint $y\leq v^\sigma_{t-1}$. The objective linearly perturbs $g$ using coordinates of $\mu$ computed in the earlier iterations, denoted as $\beta^{t-1}$ to track iteration order. This minimization problem appears in the construction of a value oracle for $\check{g}$; see~\cite[Equation~49.50]{schrijver2003combinatorial} and we show that this provides the needed shortcut to computing $\alpha^*$. We now formally establish that $\mu$ coincides with $\alpha^*$.}

\begin{proposition}\label{prop:intersecing-facet-1}
Consider an intersecting submodular function $g:X\rightarrow \R$, where $X$ is an intersecting family containing $\{e_i \mid i \in [n]\}$.   For a given vector $\bar{x}\in\cone(X)=\R^n_+$, let $\sigma$ be a permutation of $[n]$ sorting $\bar{x}$ into non-increasing order. Then, $\sconv(g)(\bar{x}) =\langle \mu,  \bar{x} \rangle$, where $\mu$ can be computed by solving $n$ lattice family submodular minimization problems given in~\eqref{eq:intersecting-facet}. 
\end{proposition}
 \begin{proof}

     First, we prove that the optimization problem in~\eqref{eq:intersecting-facet} is a submodular minimization over a lattice family. For $i \in [n]$, define
\[
Y_i = \bigl\{y \in \{0,1\}^n  \bigm| y \leq v_{i-1}^\sigma ,\ e_{\sigma(i)} + y \in X\bigr\}.
\]
Consider any $u,w \in Y_i$. We show that $Y_i$ is a lattice family. Because $u,w \leq v_{i-1}^\sigma$, both their join $u \vee w$ and meet $u \wedge w$ are bounded from above by $v_{i-1}^\sigma$. Moreover, 
\[ 
e_{\sigma(i)} + (u \vee w) = (e_{\sigma(i)} + u) \vee (e_{\sigma(i)} + w) \in X, 
\]
where the equality follows from $u_{\sigma(i)} = w_{\sigma(i)} = 0$, and the inclusion holds because $e_{\sigma(i)} + u$ and $e_{\sigma(i)} + w$ belong to $X$, $(e_{\sigma(i)} + u)\wedge (e_{\sigma(i)} + w)\geq e_{\sigma(i)} \neq \mathbf{0}$, and $X$ is an intersecting family. Similarly, we can argue that $e_{\sigma(i)} + (u \wedge w)\in X$. Therefore, $Y_i$ is a lattice family. Moreover, the submodularity of the objective function $g(e_{\sigma(i)}+y)-\langle \beta^{i-1}, y \rangle$ follows readily from the submodularity of $g$ and the modularity of the linear term.

Next, we prove by induction that $\alpha^*_{\sigma(t)} = \mu_{\sigma(t)}$ for $t \in [n]$. This, together with~\eqref{eq:intersecting-greedy}, completes the proof.   Clearly, for $t =1$, it follows from definitions that $\mu_{\sigma(t)} = g(e_{\sigma(t)}) = \check{g}(e_{\sigma(t)}) = \alpha^*_{\sigma(t)}$, where the second equality holds since $X$ contains $\{e_i \mid i \in [n]\}$. Now, consider $2\leq t\leq n$ and suppose that $\alpha^*_{\sigma(i)} = \mu_{\sigma(i)}$ for $i < t$. To show  that $\alpha^*_{\sigma(t)} \leq \mu_{\sigma(t)}$, we observe that for any $y \in Y_t$ and $x = y + e_{\sigma(t)}$, we have
     \[
     \begin{aligned}
         \alpha^*_{\sigma(t)} &\leq \check{g}(x) - \langle \alpha^*, y \rangle \leq g(x)- \langle \alpha^*, y \rangle\\
         &  = g(x)- \sum_{i <t}\alpha^{*}_{\sigma(i)}y_{\sigma(i)} = g(x)- \sum_{i <t}\mu_{\sigma(i)}y_{\sigma(i)}  = g(x) - \langle \beta^{t-1}, y \rangle,
     \end{aligned}
     \]
where the first inequality holds since $\alpha^*$ is feasible to $\EP_{\check{g}}$, the second inequality holds since  $x \in X$ and $\check{g}(v) \leq g(v)$ for $v \in X$, the first equality follows from $y \leq v^\sigma_{t-1}$, the second equality holds by induction, and the last equality holds by the definition of $\beta^{t-1}$.

For the reverse inequality, take an optimal decomposition $v^\sigma_t = \sum_{v\in V^*_t}v$ attaining  $\check{g}(v^\sigma_t)$. Now, we claim that $\langle \alpha^*, v \rangle = g(v)$ for every $v \in V^*_t$. To see this, we observe that 
\[
\sum_{v\in V_t^*}\langle\alpha^*,v\rangle =
\langle\alpha^*,v_t^\sigma\rangle =
\check{g}(v_t^\sigma) =
\sum_{v\in V_t^*}g(v),
\]
where the second equality holds by the definition of $\alpha^*$ in \eqref{eq:kuhnDual}. 
This, together with the fact $\langle\alpha^*,v\rangle\leq g(v)$ for every $v\in X$, implies $\langle\alpha^*,v\rangle=g(v)$ for all $v\in V_t^*$. The proof is complete by observing that there exists $v^* \in V^*_t$ such that $v^*_{\sigma(t)} = 1$ and, for $y^* :=v^*-e_{\sigma(t)}$, we have
\[
\begin{aligned}
    \alpha^*_{\sigma(t)} &= g(v^*) - \sum_{i \neq t}v^* \alpha^*_i = g(e_{\sigma(t)} + y^*) - \langle y^* , \alpha^* \rangle = g(e_{\sigma(t)} + y^*) - \sum_{i<t} y^*_{\sigma(i)} \alpha^*_{\sigma(i)} \\
   & = g(e_{\sigma(t)} + y^*) - \sum_{i<t} y^*_{\sigma(i)} \mu_{\sigma(i)} = g(e_{\sigma(t)} + y^*) -  \langle y^* , \beta^{t-1} \rangle \geq \mu_{\sigma(t)}, 
\end{aligned}
\]
where the third equality holds since $y^* \leq v^\sigma_{t-1}$, the fourth equality holds by induction, the last equality follows by the definition $\beta^{t-1}$, and the last inequality holds since $y^*$ is feasible to the minimization problem in~\eqref{eq:intersecting-facet}. \qed
 \end{proof}

We can now combine the two preceding results to obtain an explicit characterization of the sublinear envelope of the minimum of a modular function and a submodular function.
\begin{theorem}\label{theo:facet-1}
Consider a function $f(x)=\min\{f_1(x),f_2(x)\}$, where
$f_1(x)=a^\intercal x$ for $x\in\{0,1\}^n$, and
$f_2:\{0,1\}^n\to\R\cup\{+\infty\}$ satisfies the conditions in~\eqref{eq:condonf2}.
Let $g$ be defined in~\eqref{eq:inter}. For a given vector
$\bar{x}\in\R^n_+$, let $\sigma$ be a permutation sorting $\bar{x}$ in
non-increasing order. Then, $\sconv(f)(\bar{x})=\langle\mu,\bar{x}\rangle$ where $\mu$ can be computed by solving $n$ lattice family submodular
minimization problems given in~\eqref{eq:intersecting-facet}.
\end{theorem}

\begin{proof}
The result follows since $\sconv(f)(\bar{x})=\sconv(g)(\bar{x}) = \langle \mu, \bar{x}\rangle$, where the first equality follows from Proposition~\ref{prop:same_envelope}, and the second equality follows from Proposition~\ref{prop:intersecing-facet-1}. \qed
\end{proof}

\subsection{Compact LP-based separation}\label{section:LP}

In this subsection, we consider $K$ functions $f_k : X\subseteq \{0,1\}^n \to \R$ for $k \in [K]$, and their pointwise minimum $f(x) = \min \bigl\{ f_1(x), f_2(x), \ldots, f_K(x) \bigr\}$. Even when $f_1,\ldots,f_K$ are submodular functions, the separation problem for $\sconv(f)$ is equivalent to maximization over the intersection of $K$ polymatroids, for which no strongly polynomial-time algorithm is known. Therefore, we investigate whether a compact linear programming formulation exists when each $\conv(\epi(f_k))$ admits an individual compact formulation for $k \in [K]$.

Assume that the convex hull of the epigraph of $f_k$ admits a compact polyhedral formulation, that is,
\begin{equation}\label{eq:conv-epi-fk}
    \conv\bigl(\epi(f_k)\bigr) 
= \bigl\{ (\mu_k,x)  \bigm| \exists y_k \in \R^{m_k} \text{ s.t. }
d_k \le D_k \mu_k  + A_k x + B_k y_k 
\bigr\},
\end{equation}
where $A_k$, $B_k$, $D_k$ and $d_k$ are matrices or vectors of appropriate dimensions and $D_k\geq 0$ for $k\in [K]$. Given that each epigraph admits a compact extended formulation, we will show that the following compact linear programming formulation can be used to separate the sublinear envelope of $f$ at a given point $\bar{x} \in \cone(X)$.
\begin{equation}\label{eq:LP-sublinear}
    \max_{\alpha,(\lambda_k)_{k\in [K]}} \Bigl\{ \alpha^\intercal \bar{x}  \Bigm| d^\intercal_k \lambda_k \geq 0,\  \lambda_k \geq 0,\ D_k^\intercal \lambda_k =1,\ A_k^\intercal \lambda_k = -\alpha ,\ B^\intercal_k \lambda_k = 0\,  \for  k \in [K] \Bigr\}.
\end{equation}

\begin{proposition}\label{prop:LP_sub}
Consider functions $f_k:X\subseteq \{0,1\}^n \to \R$ for $k \in [K]$ and $f(x) = \min\{f_1(x), \ldots, f_K(x)\}$. Suppose that the convex hull of $\epi(f_k)$ is given by~\eqref{eq:conv-epi-fk} for each $k \in [K]$. Then, the separation problem for $\sconv(f)$ can be solved via the compact LP~\eqref{eq:LP-sublinear}.
\end{proposition}
\begin{proof}
The sublinear separation problem for $\sconv(f)$ can be written as 
\[
\max \Bigl\{ \alpha^\intercal \bar{x}  \Bigm| \alpha^\intercal v \leq  f_k(v)  \for  v\in X \text{ and } k \in [K] \Bigr\}.
\] 
This is equivalent to
\begin{equation}\label{eq:sub-LP-1}
    \max \Bigl\{ \alpha^\intercal \bar{x}  \Bigm| 0 \leq \min_{v \in X}\bigl\{ f_k(v) -\alpha^\intercal v \bigr\} \for  k \in [K] \Bigr\}.
\end{equation}
Using the compact convex hull description of the epigraph of $f_k$, we obtain an LP formulation for inner minimization problems, that is, for $k \in [K]$, 
\begin{equation}\label{eq:inner_primal}
\min \left\{ \mu_k - \alpha^\intercal v \;\middle|\; d_k \leq D_k \mu_k + A_k v + B_k y_k \right\}.
\end{equation}
For any $\alpha \in \R^n$, problem~\eqref{eq:inner_primal} is feasible and has a finite optimal value, since the inner problem of~\eqref{eq:sub-LP-1} is feasible and bounded. Therefore, strong duality holds, and the dual problem can be written as
\[
\max \bigl\{ d^\intercal_k \lambda_k  \bigm| \lambda_k \geq 0,\ D_k^\intercal \lambda_k =1,\ A_k^\intercal \lambda_k = -\alpha ,\ B^\intercal_k \lambda_k = 0  \bigr\},
\]
where $\lambda_k$ is the dual variable. Replacing the inner minimization problems with the dual LP problems, we obtain the LP formulation given in~\eqref{eq:LP-sublinear}.   
\end{proof}

Therefore, if each $\conv(\epi(f_k))$ admits a compact linear programming formulation, then the separation problem for $\sconv(f)$ also admits a compact LP formulation. We now construct a compact LP formulation of the convex envelope of $f$ via our lifting approach. As shown in Theorem~\ref{theorem:LP}, we can obtain the following compact LP formulation for separating the convex envelope of $f$ at a given point $\bar{x} \in \conv(X)$:
\begin{equation}\label{eq:LP-convex}
    \max_{\alpha_0,\alpha,(\lambda_k)_{k\in [K]}}
    \left\{
        \alpha_0+\alpha^\intercal \bar{x}
        \;\middle|\;
        d_k^\intercal \lambda_k \geq \alpha_0,\;
        D_k^\intercal \lambda_k = 1,\;
        A_k^\intercal \lambda_k = -\alpha,\;
        B_k^\intercal \lambda_k = 0,\;
        \lambda_k\geq 0
        \ \for k\in [K]
    \right\}.
\end{equation}

\begin{theorem}\label{theorem:LP}
Consider functions $f_k:X\subseteq \{0,1\}^n \to \R$ for $k \in [K]$ and $f(x) = \min\{f_1(x), \ldots, f_K(x)\}$. Suppose that the convex hull of $\epi(f_k)$ is given by~\eqref{eq:conv-epi-fk} for each $k \in [K]$. Then, the separation problem of $\conv(f)$ can be solved via the compact LP~\eqref{eq:LP-convex}. 
\end{theorem}
\begin{proof}
For each \(k\in [K]\), define the extended function
\(F_k:\Lf\to \mathbb{R}\) by
\[
    F_k(0,\mathbf{0})=0,
    \qquad
    F_k(1,x)=f_k(x)
    \quad \for x\in X .
\]
where $\Lf =\{(0,\mathbf{0})\}\cup \{(1,x) \mid x\in X\}.$
Let $F(t,x) = \min_{k\in [K]} F_k(t,x)$ for $(t,x)\in \Lf$. 
Proposition~\ref{prop:pers-lift} implies that $\conv(f)(x)=\sconv(F)(1,x)$ for $x\in \conv(X)$. Thus, separating \(\conv(f)\) at \(\bar{x}\in \conv(X)\) is equivalent to separating \(\sconv(F)\) at the lifted point \((1,\bar{x})\).

We next derive an extended formulation for $\conv(\epi(F_k))$. Since
$F_k(0,0)=0$ and $F_k(1,x)=f_k(x)$, we can write $\conv(\epi(F_k))=\conv(D_0\cup D_1),$
where $D_0=\{(\mu,t,x)\mid \mu\ge 0,\ t=0,\ x=0\},$
and $D_1=\{(\mu,t,x)\mid t=1,\ (\mu,x)\in \conv(\epi(f_k))\}$.
Because $X \subseteq \{0,1\}^n$ is bounded, the recession cones of both $D_0$ and $D_1$ are identical, consisting solely of the ray $\{(\mu, 0, \mathbf{0}) \mid \mu \ge 0\}$. Therefore, by disjunctive programming~\cite{balas1979disjunctive}, we obtain the extended formulation:
\[
\conv(\epi(F_k)) = \left\{(\mu, t, x)\:\middle |\: 
\begin{aligned}
    &\exists \mu^0, \mu^1 \in \mathbb{R},\; \lambda^0, \lambda^1 \in \mathbb{R}_+,\; x^1 \in \mathbb{R}^n,\; y^1_k \in \mathbb{R}^{m_k}  \\
    &\text{ s.t.}\mu = \mu^0+\mu^1, t = \lambda^1, x = x^1,\lambda^0+\lambda^1 = 1, \mu^0 \ge 0,\\
    &\lambda^1 d_k \leq D_k\mu^1 + A_kx^1+B_ky^1_k
\end{aligned}\right\}.
\]
Projecting out the disjunctive auxiliary variables $\mu^0, \mu^1, \lambda^0, \lambda^1$, and $x^1$ simplifies this formulation. From $\lambda^0+\lambda^1 = 1$ and $\lambda^0, \lambda^1 \ge 0$, we directly have $t = \lambda^1 \in [0,1]$. Furthermore, the condition $\mu^0 \ge 0$ implies $\mu^1 = \mu - \mu^0 \le \mu$. Since the epigraph coefficient $D_k \ge \mathbf{0}$, the inequality $\lambda^1 d_k \leq D_k\mu^1 + A_kx^1+B_ky^1_k$ holds for some $\mu^1 \le \mu$ if and only if it holds for $\mu^1 = \mu$. Substituting $\mu^1$ with $\mu$ and $x^1$ with $x$ yields
\begin{equation*}
\begin{aligned}
    \conv\bigl(\epi(F_k)\bigr)
    =
    \bigl\{
        (\mu,t,x)
        \bigm|
        &\ 0\leq t\leq 1,\;
        \exists y_k\in \mathbb{R}^{m_k}
        \text{ s.t.} \:\: t d_k \leq D_k\mu + A_kx + B_ky_k
    \bigr\}.
\end{aligned}
\end{equation*}
We can now apply Proposition~\ref{prop:LP_sub} to the extended functions
\(F_k\), \(k\in[K]\) and  obtain the following LP for separating
\(\sconv(F)\) at \((1,\bar{x})\):
\[
\max_{\alpha_0,\alpha,(\lambda_k,\rho_k,\sigma_k)_{k\in[K]}}\left\{\alpha_0+\alpha^\intercal\bar{x}\: \middle |\:\text{s.t.}
\begin{aligned}
    \quad
    & D_k^\intercal\lambda_k = 1, A_k^\intercal\lambda_k = -\alpha, B_k^\intercal\lambda_k = 0, \\
    & -d_k^\intercal\lambda_k+\rho_k-\sigma_k=-\alpha_0, -\sigma_k\geq 0, \\
    & \lambda_k\geq 0,\ \rho_k\geq 0,\ \sigma_k\geq 0,
\end{aligned}\for k\in[K]\right\}
\]
Since \(\sigma_k\geq 0\) and \(-\sigma_k\geq 0\), we have \(\sigma_k=0\). Hence, $\rho_k=d_k^\intercal\lambda_k-\alpha_0$. Eliminating the auxiliary multipliers \(\rho_k\) and \(\sigma_k\), we obtain the compact LP~\eqref{eq:LP-convex}.
\end{proof}

\section{Applications}\label{section:app}
In this section, we consider two applications of our theoretical developments. First, in Section~\ref{sec:inventory}, we consider a class of stochastic programming problems arising in inventory management. There, our characterization of the sublinear envelope for the minimum of submodular functions in Section~\ref{sec:inter} enables the design of efficient evaluation oracles, which in turn yields an alternative proof of a recent result by~\cite{punt2025multi}. 
This alternate proof highlights the connection of the new result with intersecting families, which are well-studied in combinatorial optimization, and allows for more general domains of the demand function. Second, in Section~\ref{sec:bilinear}, we specialize our results to bilinear functions, which form an important class of disjunctive submodular functions. We focus on a class of bilinear functions whose interaction structure can be characterized by cycles. In particular, we derive an explicit closed-form expression for their convex envelope in the original space. Consequently, separation for this class can be performed in $\mathcal{O}(n)$ time, rather than by using the weighted polymatroid intersection algorithm suggested by Theorem~\ref{convex-sep_of_lat} or solving the high-dimensional LP proposed in Theorem~\ref{theorem:LP}.

\subsection{Single-period multi-product newsvendor problems with demand and supply restrictions}\label{sec:inventory}
Consider a single-period multi-product newsvendor problem defined as follows
\begin{equation}
\max \Bigl\{\mathbb{E}_{d \sim \mathbb{D}}\bigl[R(q,d)\bigr] - c^\intercal q \Bigm| q \geq 0 \Bigr\},
    \tag{\textsc{Newsvendor}}
\end{equation}
where $q$ denotes an order decision on $n$ products, $d$ is a random demand vector with a distribution $\mathbb{D}$ that will be specified later, $R(q,d)$ denotes the revenue obtained from a given order quantity $q$ and a realized demand $d$, and $c$ is a cost vector. 
To determine an optimal ordering policy, a computationally tractable evaluation oracle for the revenue function is essential~\cite{punt2025multi}. When the demands and supplies of the $n$ products are mutually independent, the revenue function is separable, allowing the multi-product problem to decompose into $n$ independent, single-product newsvendor subproblems. In such settings, evaluating the expected revenue is straightforward. Conversely, when the demands and supplies exhibit cross-product dependencies, the revenue evaluation problem becomes structurally complicated and, in general, computationally challenging.

In this subsection, we focus on a particular class of revenue models in which  unsatisfied demands are lost, and both demand and supply restrictions are allowed to be general set functions. More specifically, we allow  a realized demand to be a set function $d : \{0,1\}^{n} \to \R_+$, where for $x \in \{0,1\}^n$, $d(x)$ denotes the total demand associated with the product bundle $\{i \mid x_i = 1\}$. Similarly, for an order quantity $q \geq 0$, the available supply is governed by a joint capacity function $s_q: \{0,1\}^n \to \R_+$. Given a vector of revenues $p = (p_1, \ldots, p_n)$ where  $p_i\geq 0$ is the revenue obtained from selling one unit of item $i$, the firm solves the following allocation problem to maximize the revenue,
\begin{equation}
    R(p \mid q,d) := \max \Bigl\{ p^\intercal y \Bigm| y^\intercal x \leq \min \bigl\{d(x), s_q(x) \bigr\} \for x  \in \{0,1\}^n, y \geq 0  \Bigr\}. 
\end{equation}
In practical settings, both $d$ and $s_q$ are naturally assumed to be nondecreasing set functions, since adding products to a collection cannot reduce the aggregate demand or available supply. However, to capture realistic operational environments, neither $d$ nor $s_q$ is necessarily linear. On the supply side, firms often employ resource pooling strategies to mitigate the risk of resource underutilization~\cite{jiang2023achieving}. In such settings, the supplies of different products become coupled through shared resource constraints, and under certain structural assumptions the resulting supply function can be characterized as a submodular function \cite{jiang2023achieving}. On the demand side, interactions among products may induce substitution or complementarity effects, causing the demand function to exhibit a  set function structure rather than a simple additive form \cite{bassok1999single}. Furthermore, when customers exhibit product indifference, aggregate demand can be represented by a submodular function~\cite{punt2025multi}. 

In the next, we show that the revenue function is equivalent to  the sublinear envelope of the minimum of demand and supply functions.  
\begin{proposition}\label{prop:news}
Let $d$ be a non-decreasing demand realization and for an order quantity $q \geq 0$, let $s_q$ be a non-decreasing supply function. Define $(d \wedge s_q)(x):= \min \bigl\{d(x),s_q(x)\bigr\}$ for $x \in \{0,1\}^n$. Then, 
\[
R(p \mid q,d) = \sconv(d \wedge s_q) (p) \qquad \for \, p \geq 0,
\]
which can be evaluated in strongly polynomial time if in addition $d$ and $s_q$ are submodular.  
\end{proposition}

\begin{proof}
We have
\begin{align*}
R(p \mid q, d)& = \max \{ p^\intercal y \mid y \in \EP_d \cap \R_+^n,\; y \in \EP_{s_q} \cap \R_+^n \} \\
&= \max \{ p^\intercal y \mid y \in \EP_{d},\; y \in \EP_{s_q} \} \\
& = \sconv(d\wedge s_q)(p),
\end{align*}
where the first equality follows from the definitions of $\EP_{d}$ and $\EP_{s_q}$, and the last equality follows from Lemma~\ref{lemma:dualsconv}. To justify the second equality, observe that every feasible solution of the optimization problem on the left-hand side is also feasible for the problem on the right-hand side. Therefore, it suffices to show that for any feasible solution $y$ of the problem on the right-hand side, there exists another feasible solution $z\in \R^n_+$ whose objective value satisfies $p^\intercal z \geq p^\intercal y$.

Consider an arbitrary vector $y \in \R^n$ satisfying $y^\intercal v \le d(v)$ and $y^\intercal v \le s_q(v)$ for all $v \in \{0,1\}^n$. Define a vector $z \in \R^n_+$ by
\[
z_i = \max\{y_i,0\} \quad\for i \in [n].
\]
Since $p \in \R_+^n$, it is immediate that $p^\intercal z \ge p^\intercal y$. Thus, it remains to verify the feasibility of $z$. We first establish feasibility with respect to the constraints induced by the demand function $d$. Let $v \in \{0,1\}^n$ be arbitrary. Define a vector $u \in \{0,1\}^n$ component-wise by
\[
u_i = \begin{cases}
v_i, & \text{if } y_i > 0,\\
0, & \text{if } y_i \le 0,
\end{cases}\quad\for i\in [n].
\]
Then, we have
\[
z^\intercal v = \sum_{y_i>0} y_i v_i =
y^\intercal u \leq d(u) \leq d(v).
\]
The first inequality follows from the fact that $y \in \EP_d$. The second follows from $u \leq v$ and  $d$ is monotonic. Therefore, we conclude that $z^\intercal v \le d(v)$ for $v \in \{0,1\}^n$. Using the same argument and the monotonicity of $s_q$, we obtain $z^\intercal v \le s_q(v)$ for $v \in \{0,1\}^n$. Hence, $z$ is feasible for the optimization problem on the left-hand side and satisfies $p^\intercal z \ge p^\intercal y$ since $p\ge 0$. This establishes the second equality. Finally, by Proposition~\ref{prop:sep_sublinear}, it follows that $R(p \mid q, d)$ can be evaluated in strongly polynomial time.\qed
\end{proof}

We now specialize Theorem~\ref{theo:facet-1} to the setting studied by~\cite{punt2025multi}, where the available supply is linear, i.e., $s_q(x) = q^\intercal x$ for $x \in \{0,1\}^n$. This specialization allows us to derive an explicit, closed-form expression for the expected revenue function $R(q,d)$.  Before presenting this special case, we illustrate in the next example that  Theorem~\ref{theo:facet-1} handles more complicated cases than those treated by~\cite{punt2025multi}.  

\begin{example}\label{ex:3}
Consider $f:\{0,1\}^3 \to \R$ defined as $f(x) = \min \{f_1(x),f_2(x)\}$, where the modular component is given by
\[
f_1(x)=5x_1+5x_2+5x_3 \quad\for x\in\{0,1\}^3,
\]
and the submodular component is defined as
\[
f_2(x)=2x_1+2x_2+7x_3-x_1x_2\:\for x\in L,\quad \text{and }f_2(x) = +\infty\:\for x\notin L.
\]
Under two distinct choices for the lattice family $L$, we evaluate the sublinear envelope of $f$.  

Case 1: The unconstrained lattice $L = \{0,1\}^3$. By enumerating all permutations and applying Theorem~\ref{theo:facet-1}, we obtain the sublinear envelope of $f$ explicitly as:
\[
\sconv(f)(x) = \max\left\{ 2x_1 + x_2 + 5x_3,\; x_1 + 2x_2 + 5x_3 \right\}.
\]

Case 2: The constrained lattice $L = \{x \in \{0,1\}^3 \mid x_2 \le x_3\}$. Under this domain restriction, enumerating all permutations and applying Theorem~\ref{theo:facet-1} yields:
\[
\sconv(f)(x) = \max\left\{
2x_1 + 5x_2 + 3x_3,\;
2x_1 + 3x_2 + 5x_3,\;
x_1 + 5x_2 + 4x_3,\;
x_1 + 4x_2 + 5x_3
\right\}.
\]

Consequently, appending even a single linear inequality ($x_2 \le x_3$) to the domain alters the polyhedral structure of the sublinear envelope substantially. Specifically, the linear functions that define the facets of $\sconv(f)$ in the unconstrained case are no longer facet-defining under the restricted domain. Instead, a completely different and larger set of linear functions is required to precisely characterize the sublinear envelope.\qed
\end{example}

Following the construction of Theorem~\ref{theo:facet-1}, we incorporate the supply and demand functions into one single intersecting submodular function. Since both functions are unconstrained,  the  function defined in~\eqref{eq:inter} becomes an intersecting submodular function  $g:\{0,1\}^n\to\mathbb{R}$ given as follows
\begin{equation*}
g(x)= \begin{cases}
\min\{d(x),\, q^\intercal x\} & \text{if } x\in\{e_i \mid i\in[n]\},\\
d(x) & \text{otherwise}.
\end{cases}
\end{equation*}
For a given price vector $p \in \R^n_+$, we assume that $p_1 \geq p_2 \geq \cdots p_n \geq 0$; otherwise we sort the items to satisfy this relation. It follows readily from Theorem~\ref{theo:facet-1} that $\sconv(d \wedge s_q)(p) = \langle \mu^*, p\rangle $, where $\mu^* = (\mu^*_1, \ldots, \mu^*_n)$ is obtained recursively as follows. Initially, $\beta^0 = \mathbf{0}$, $v_0 =\mathbf{0}$. Next, for $t \in [n]$,  compute
\begin{equation}\label{eq:news-closed}
\mu^*_t = \min_{y \in \{0,1\}^n}
\left\{ g\bigl(e_t+y\bigr) - \langle \beta^{t-1}, y\rangle \;\middle|\; y \le v_{t-1},\;  e_t + y\in \{0,1\}^n \right\},
\end{equation}
and set $\beta^t = \beta^{t-1} + \mu_t^* e_t$ and $v_t = v_{t-1} + e_t$.

\begin{corollary}\label{coro:regular}
    Let $d$ be a non-decreasing submodular demand realization and for an order quantity $q \geq 0$, let $s_q(x) = q^\intercal x$. For a given price $p\in \R^n_+$ such that $p_1 \geq p_2 \geq \cdots \geq p_n$, we have $R(p\mid q, d)= \langle \mu^*,  p  \rangle $, where $\mu^*$ is recursively defined as in~\eqref{eq:news-closed}.
\end{corollary}
\begin{proof}
    From Proposition~\ref{prop:news}, we have $    R(p \mid q,d) = \sconv(d \wedge s_q) (p)$ for $p \geq 0$,  where $(d \wedge s_q)(x):= \min \bigl\{d(x),s_q(x)\bigr\}$ for $x \in \{0,1\}^n$. Then, the result follows directly from Theorem~\ref{theo:facet-1}.\qed
\end{proof}

\subsection{Bilinear Programming}\label{sec:bilinear}
Consider a 0-1 bilinear function
\[
f(x)=\sum_{(i,j)\in E} a_{ij} x_i x_j\quad\for x\in\{0,1\}^n,
\]
where $E$ is the edge set of a graph $G = ([n], E)$.
A standard approach for handling binary bilinear functions is to introduce a new variable $y_{ij}$ for each bilinear term $x_i x_j$, and enforce the relationship between $y$ and $x$ using the McCormick inequalities~\cite{mccormick1976computability}. The resulting linear programming relaxation is known as the McCormick relaxation.

An approach to tighten McCormick relaxation is to consider the Boolean Quadratic Polytope (BQP)~\cite{padberg1989boolean},
\[
\conv\left(\bigl\{(y,x)\mid y_{ij} = x_i x_j\:\for\: (i,j)\in E,\ x\in \{0,1\}^n\bigr\}\right),
\]
and introduce inequalities valid for this polytope to the relaxation. The odd cycle inequalities, first introduced by \cite{padberg1989boolean}, are among the most widely used classes of valid inequalities. When the graph associated with the bilinear function forms a cycle, the McCormick inequalities together with the odd cycle inequalities provide a complete characterization of BQP. However, when several cycles are present, this approach may require the introduction of $|E|$ additional variables, which in the worst case amounts to $n(n-1)/2$ variables. Consequently, the problem size grows quadratically as $n$ increases.

Our approach avoids introducing the additional variables $y$ while strengthening the McCormick relaxation. The main idea is to exploit disjunctive submodular structures inherent in $f$. In this subsection, we develop the theoretical foundations of this approach. We defer the computational study to Section~\ref{sec:computation}, where we investigate how these disjunctive submodular structures can be exploited to strengthen relaxations for bilinear programs. 

Recall a bilinear function $f(x) = \sum_{(i,j)\in E} a_{ij} x_i x_j$ is submodular if and only if $a_{ij} \le 0$ for all $(i,j)\in E$~\cite{boros2002pseudo}. Therefore, if $f$ is not submodular, there must exist an index set $I \subseteq [n]$ and an index set $J_i \subseteq [n]$ such that $a_{ij} > 0$ for all $i\in I$ and $j \in J_i$. In fact, a bilinear function $f$ is disjunctive submodular function, as it is submodular on the collection of faces 
\[
\bigl\{x \in [0,1]^n \bigm| x_i =0 \for i \in I_0 \text{ and } x_i = 1 \for i \in I \setminus I_0,\ I_0 \subseteq I \bigr\}.
\]
To represent this special disjunctive submodular function as the pointwise minimum of submodular functions, instead of using Theorem~\ref{them:extension-lattice}, we could use the following fact:
\[
\sum_{j\in J_i} a_{ij} x_ix_j  = x_i \left( \sum_{j\in J_i} a_{ij} x_j \right)
=
\min\left\{
\sum_{j\in J_i} a_{ij} x_i,\;
\sum_{j\in J_i} a_{ij} x_j
\right\} \qquad \for i \in I,
\]
where the last equality holds since for $x_i = 0$, both sides are equal to zero, and for $x_i = 1$, the left-hand side reduces to $\sum_{j\in J_i} a_{ij} x_j$, and as $\sum_{j\in J_i} a_{ij} x_j \le \sum_{j\in J_i} a_{ij} x_i$, the right-hand side evaluates to the same quantity. Consequently, any bilinear function can be reformulated as the minimum of several submodular bilinear functions. In particular, when the minimum is taken over two submodular functions, the separation problem over $\conv(f)$ can be solved in strongly polynomial time.
\begin{corollary}\label{eq:sep_bilinear}
    Let $f(x) = x_1\cdot a^\intercal x + g(x)$ on $x\in \{0, 1\}^n$, where $a\in \R^n_+$ and $g$ is a submodular bilinear function. Then, the separation problem of $\conv(f)$ can be solved in strongly polynomial time.
\end{corollary}
\begin{proof}
    Since $x \in \{0,1\}^n$ and $a \in \R_+^n$, we can express $f(x)$ as
\[
f(x)
=
\min\left\{
x_1 \cdot \mathbf{1}^\intercal a + g(x),\;
a^\intercal x + g(x)
\right\}.
\]
Because $g(x)$ is a submodular function, both $x_1 \cdot \mathbf{1}^\intercal a + g(x)$ and $a^\intercal x + g(x)$ are submodular functions as well. Therefore, Theorem~\ref{convex-sep_of_lat} yields a strongly polynomial-time separation oracle for $\conv(f)$. \qed
\end{proof}

Next, we consider a special case where the graph $G$ is a cycle. Without loss of generality, we can assume that at most one edge has a positive coefficient; otherwise, the  switching operation  $x_i \mapsto 1-x_i$ can be applied. Under this assumption, the bilinear function can be represented as the minimum of two submodular functions, allowing Corollary~\ref{eq:sep_bilinear} to be applied. Exploiting the cycle structure further, however, enables us to refine the construction and derive a closed-form expression directly in the $x$-variable space. Moreover, separation for the convex envelope can be performed in $\mathcal{O}(n)$ time, rather than by using the weighted polymatroid intersection algorithm suggested by Theorem~\ref{convex-sep_of_lat} or solving the high-dimensional LP proposed in Theorem~\ref{theorem:LP}.

\begin{theorem}\label{theo:cycle-envelope}
    Let $f(x) = \sum_{i\in [n-1]}a_{i}x_ix_{i+1} + x_1x_{n}$, where for $i \in [n-1]$, $a_{i} < 0$. Then, the convex envelope of $f(x)$ is given by $r(x)$,
    \begin{equation}\label{eq:cycle-env}
        \begin{aligned}
            r(x) &:= \max\{0, x_1+x_n - 1\} + \sum_{i\in [n-1]}a_{i}\min\{x_i, x_{i+1}\} \\
            &+ a \cdot \max\left\{0, \sum_{i\in [n-1]}\min\{x_i, x_{i+1}\} - \max\{0, x_1+x_n - 1\} - \sum_{i=2}^{n-1}x_i\right\},\\
        \end{aligned}\tag{\textsc{Cycle-Env}}
    \end{equation}
    where $a := \min\bigl\{|a_{1}|, \ldots, |a_{n-1}|, 1\bigr\}$.
\end{theorem}
\begin{proof}
The function $r(x)$ consists of three terms, denoted as, $r_1(x):=\max\{0, x_1+x_n - 1\}$, $r_2(x):=\sum_{i\in [n-1]}a_{i}\min\{x_i, x_{i+1}\}$ and $r_3(x) := \sum_{i\in [n-1]}\min\{x_i, x_{i+1}\} - \max\{0, x_1+x_n - 1\} - \sum_{i=2}^{n-1}x_i$. 

First, we show that $r(x)\leq \conv(f)(x)$ for $x \in [0,1]^n$. Clearly, $r(x)$ is convex since $r(x)$ can be expressed the pointwise maximum of two convex functions $r_1(x)+r_2(x)$ and $r_1(x) + r_2(x)+a \cdot r_3(x)$, where the convexity of the latter function follows from the definition of $a$. Moreover, $r(x) \leq f(x)$ for every $x \in \{0,1\}^n$. To see this, we observe that for $x \in [0,1]^n$ with $x_n =0$
\[
r_3(x) = \sum_{i \in [n-2]} \min\{x_i,x_{i+1}\} - \sum_{i =2}^{n-1} x_i = \sum_{i \in [n-2]} \bigl(\min\{x_i,x_{i+1}\} - x_{i+1} \bigr)  \leq 0;
\]
and with $x_n =1$, 
\[
r_3(x) = \sum_{i \in [n-2]} \min\{x_i,x_{i+1}\} + x_{n-1} - x_1- \sum_{i =2}^{n-1} x_i = \sum_{i \in [n-2]} \bigl(\min\{x_i,x_{i+1}\} - x_{i} \bigr)  \leq 0.
\]
Therefore, for $x \in \{0,1\}^n$, $r(x) = r_1(x) + r_2(x) = f(x)$, where the first equality holds due to $r_3(x)\leq 0$, and the second equality holds since $x_ix_j = \min\{x_i,x_j\} = \max\{0,x_i+x_j-1\}$  holds for binary $x$. Since $r(\cdot)$ is a convex underestimator of $f(\cdot)$, we have $r(x) \leq \conv(f)(x)$ for every $x \in [0,1]^n$. 

Now, we show that $r(x) \geq \conv(f)(x)$ for every $x \in [0,1]^n$. To show this, we express $f$ as the pointwise minimum of two submodular functions and invoke Theorem~\ref{theorem:LP} to provide an LP formulation for the convex envelope of $f$.  More specifically, we express $f(x)$  as $ \min \bigl\{f_1(x),f_2(x)\bigr\}$, where,
\[
f_1(x) := \sum_{i\in [n-1]} a_i x_i x_{i+1} + x_1
\quad \text{and} \quad
f_2(x) := \sum_{i\in [n-1]} a_i x_i x_{i+1} + x_n.
\]
It follows from Section 5.1.2 of~\cite{boros2002pseudo} that the convex hull of the epigraph $f_i$ can be described as follows:
\[
\begin{aligned}
    \conv\bigl(\epi(f_1)\bigr) &= \Bigl\{(\mu_1,x) \Bigm| \mu_1\geq \sum_{i\in [n-1]} a_i y_i + x_1,\ x \in [0,1]^n,\ y_i \leq \min \{x_i,x_{i+1}\} \for i \in [n-1]    \Bigr\}  \\
    \conv\bigl(\epi(f_2)\bigr) &= \Bigl\{(\mu_2,x) \Bigm| \mu_2\geq \sum_{i\in [n-1]} a_i y_i + x_n,\ x \in [0,1]^n,\ y_i \leq \min \{x_i,x_{i+1}\} \for i \in [n-1]  \Bigr\}.
\end{aligned}
\]
For a given $\bar{x} \in [0,1]^n$, Theorem~\ref{theorem:LP} yields an LP formulation whose optimal value is equal to $\conv(f)(\bar{x})$. Moreover, the dual of this LP can be expressed as follows:
\begin{subequations}\label{dual}
\begin{align}
D(\bar{x}) := \min_{\substack{\lambda,s,z,\\ \lambda',s',z'}}
\quad
&\sum_{i=1}^{n-1} a_i(\lambda_i+\lambda_i') + s_1+s_n' \nonumber\\
\text{s.t.}\quad & z+z' = 1, \label{dual-1}\\
& s_i+s_i' = \bar{x}_i \qquad \for i \in [n], \label{dual-2}\\
& \lambda_i \le s_i \text{ and } \lambda_i' \le s_i' \qquad  \for i\in[n-1],\label{dual-3}\\
& \lambda_i \le s_{i+1} \text{ and }  \lambda_i' \le s_{i+1}' \qquad   \for i\in[n-1],\label{dual-4}\\
& z \ge s_i \text{ and } z' \ge s_i'  \qquad   \for i\in[n], \label{dual-5}\\
& z \ge 0, s \ge 0,z' \ge 0, s' \ge 0.\label{dual-6}
\end{align}
\end{subequations}
By Theorem~\ref{theorem:LP} and the weak duality of LP, we have that  $ D(x) \ge \conv(f)(x)$ for $x\in[0,1]^n$. In the following, we will argue that  $r(\bar{x})\ge D(\bar{x})$, showing  that $r(\bar{x})\ge \conv(f)(\bar{x})$. We consider the following four cases according to the values of $\bar{x}_1+\bar{x}_n$ and $r_3(\bar{x})$: 
\begin{enumerate}
    \item \label{cycle:case1} $\bar{x}_1+\bar{x}_n \leq 1$ and $r_3(\bar{x})\leq 0$;
    \item \label{cycle:case2} $\bar{x}_1+\bar{x}_n \leq 1$ and $r_3(\bar{x})\geq 0$;
    \item \label{cycle:case3} $\bar{x}_1+\bar{x}_n \geq 1$ and $r_3(\bar{x})\leq 0$;
    \item \label{cycle:case4} $\bar{x}_1+\bar{x}_n \geq 1$ and $r_3(\bar{x})\geq 0$.
\end{enumerate}
 For convenience, let $\Delta \bar{x}_i := \bar{x}_{i+1} - \bar{x}_i$.

\textbf{Case~\ref{cycle:case1}:} In this case, we have $r(\bar{x}) = r_2(\bar{x})$. Now, consider the following dual variables:
    \[
    \begin{aligned}
        &s_1 = 0 \text{ and } s_i = \min\Bigl\{\bar{x}_n, \bar{x}_i, \sum_{k=1}^{i-1}\max(0, \Delta \bar{x}_k)\Bigr\} \quad \for i \geq 2, \\
        & s'_i = \bar{x}_i - s_i \, \for i \in [n] \quad  \lambda_i = \min\{s_i, s_{i+1}\} \text{ and } \lambda_i' = \min\{s'_i, s'_{i+1}\} \quad \for i \in [n-1], \\
        &z = \bar{x}_n \text{ and } z' = 1-\bar{x}_n. 
    \end{aligned}
    \]
These solutions are feasible to the dual problem. Clearly, constraints~\eqref{dual-1}-\eqref{dual-4} and~\eqref{dual-6} are satisfied. Moreover, $s_1 = 0\leq \bar{x}_n = z$ and $s_i\leq \bar{x}_n = z$ for $i \geq 2$. Last, $s'_1 = \bar{x}_1 \leq 1 -\bar{x}_n$, and for $i \geq 2$, 
    \[
    \begin{aligned}
        s'_i &= \bar{x}_i - \min\Bigl\{\bar{x}_n, \bar{x}_i, \sum_{k=1}^{i-1}\max(0, \Delta \bar{x}_k)\Bigr\} = \max \Bigl\{\bar{x}_i - \bar{x}_n, 0, \bar{x}_i - \sum_{k=1}^{i-1}\max(0, \Delta \bar{x}_k)\Bigr\} \\
        & = \max \Bigl\{\bar{x}_i - \bar{x}_n, 0, \bar{x}_1 + \sum_{k=1}^{i-1} \Delta \bar{x}_k - \sum_{k=1}^{i-1}\max(0, \Delta \bar{x}_k)\Bigr\}  \leq \max\{1-\bar{x}_n, 0, \bar{x}_1\} =  z',
    \end{aligned}
    \] 
where the equalities hold by definitions, the inequality holds since $\bar{x}_i \leq 1$ and $\sum_{k=1}^{i-1} \bigl( \Delta \bar{x}_k - \max(0,\Delta \bar{x}_k) \bigr) \leq 0$, and the last equality holds due to $1-\bar{x}_n \geq \bar{x}_1$. 

Now, we show the objective value of the feasible solution is $r_2(\bar{x})$.  Observe that  $s_n = \bar{x}_n$ and $s'_n = 0$ since in this case, we have 
\[
\sum_{k \in [n-1]}\max(0,\Delta \bar{x}_k) = \sum_{k \in [n-1]} \bigl(\bar{x}_{k+1}-\min \{\bar{x}_k,\bar{x}_{k+1}\}\bigr) = \bar{x}_n - r_3(\bar{x}) \geq \bar{x}_n,
\]
where the second equality holds due to $\max\{0,\bar{x}_1+\bar{x}_n-1\} = 0$, and the inequality holds due to $r_3(\bar{x}) \leq 0$. Moreover, for $i \in [n-1]$, we have $\lambda_i + \lambda'_i = \min \{\bar{x}_i,\bar{x}_{i+1}\}$ since $\Delta \bar{x}_i \geq 0$  implies that 
\[
\begin{aligned}
    s_i  &\leq  \min\left\{\bar{x}_n, \bar{x}_i+\Delta \bar{x}_i, \sum_{k=1}^{i-1}\max(0, \Delta \bar{x}_k)+\Delta \bar{x}_i\right\} = s_{i+1}, \\
    s'_i &= \max\{\bar{x}_i - \bar{x}_n, 0, \bar{x}_i - \sum_{k=1}^{i-1}\max(0, \Delta \bar{x}_k)\} \\
    &\leq \max\left\{\bar{x}_i-\bar{x}_n+\Delta\bar{x}_i, 0, \bar{x}_i+\Delta \bar{x}_i - \sum_{k=1}^{i-1}\max(0, \Delta \bar{x}_k)\right\} = s'_{i+1}.
\end{aligned}
\]
On the other hand, $\Delta \bar{x}_i \leq 0$ implies that $s_i\geq s_{i+1}$ and $s'_i \geq s'_{i+1}$.  Therefore, we obtain a feasible objective value of $r_2(\bar{x})$. 
This completes the proof as   
\[
\conv(f)(\bar{x})\leq D(\bar{x}) \leq r_2(\bar{x}) = r(\bar{x}). 
\]

    \textbf{Case~\ref{cycle:case2}:} In this case, $r(\bar{x}) =r_2(\bar{x}) + a r_3(\bar{x})$.  Let $i^*$ denote the index corresponding to the minimum value among $\{|a_1|, \ldots, |a_{n-1}|, 1\}$. In the case of multiple minimizers, we select an arbitrary one, and all subsequent arguments continue to hold. In particular, if the minimum value is attained at the last element $1$, we define $i^* = n$. Then, we define the following dual variables: 
    \[
    \begin{aligned}
        &s_1 = 0\quad  
    s_i = \min\Bigl\{\bar{x}_n, \bar{x}_i, \sum_{j=1}^{i-1} \max(0, \Delta \bar{x}_j) + \delta_i\Bigr\} \for i \geq 2\\
    &  s'_i = \bar{x}_i - s_i \for i \in [n] \quad \lambda_i = \min\{s_i, s_{i+1}\} \text{ and } \lambda_i' = \min\{s'_i, s'_{i+1}\} \quad \for i \in [n-1]  \\
    &    z = \bar{x}_n \text{ and } z' = 1-\bar{x}_n ,
    \end{aligned}
    \]
    where
    \[
    \delta_i = \begin{cases}
    r_3(\bar{x})  & \text{if }  i > i^*, \\
    0  & \text{otherwise}.
    \end{cases}
    \]
In the following, we will show these solutions are feasible with an objective value of $r_2(\bar{x}) + ar_3(\bar{x})$, thus completing the proof since 
\[
\conv(f)(\bar{x})\leq D(\bar{x}) \leq r_2(\bar{x}) + ar_3(\bar{x})  = r(\bar{x}). 
\]

Clearly,  constraints~\eqref{dual-1}-\eqref{dual-4} and~\eqref{dual-6} are satisfied. Moreover, $s_1 = 0\leq \bar{x}_n = z$ and $s_i\leq \bar{x}_n = z$. Last, $s'_1 = \bar{x}_1' \leq 1-\bar{x}_n$ and for $i \geq 2$, by using $\delta_i \geq 0$ together with the same argument as in Case~\ref{cycle:case1}, we obtain
  \[
    s'_i =  \max\Bigl\{\bar{x}_i - \bar{x}_n, 0, \bar{x}_i - \sum_{j=1}^{i-1} \max(0, \Delta \bar{x}_j) - \delta_i\Bigr\}\leq \max\{1-\bar{x}_n,0, \bar{x}_1 \}  = z'.
    \]
    
    Now, we derive the objective value of this feasible solution. Observe that for $i \geq 2$,   $s_i = \sum_{j=1}^{i-1} \max(0, \Delta \bar{x}_j) + \delta_i$.  This is because
\begin{align*}
\sum_{j=1}^{i-1} \max(0,\Delta \bar{x}_j) + \delta_i
&\leq \sum_{j=1}^{i-1} \max(0,\Delta \bar{x}_j) + r_3(\bar{x}) \\
&\leq \sum_{j=1}^{i-1} \max(0,\Delta \bar{x}_j)
    + \bar{x}_n
    - \sum_{j=1}^{n-1}
    \bigl(\bar{x}_{j+1}-\min\{\bar{x}_{j+1},\bar{x}_j\}\bigr) \\
&= \bar{x}_n - \sum_{j=i}^{n-1}\max(0,\Delta \bar{x}_j) \\
&\leq \bar{x}_n - \max\Bigl\{0, \sum_{j=i}^{n-1}\Delta \bar{x}_j\Bigr\}\\
&= \min\{\bar{x}_n, \bar{x}_i\},
\end{align*}
where the first inequality follows from the fact that $r_3(\bar{x}) \geq 0$, the second inequality holds as $\max\{0,\bar{x}_1+\bar{x}_n-1\}=0$, the first equality follows from the identity $\max(0,\Delta \bar{x}_j)=\bar{x}_{j+1}-\min\{\bar{x}_j,\bar{x}_{j+1}\}$, and the second equality follows from the identity $\bar{x}_i+\sum_{j=i}^{n-1}\Delta \bar{x}_j=\bar{x}_n$. In the following, we will discuss two cases, $i^*=n$ and $i^*<n$. For the first case, for $i \in [n]$, $\delta_i = 0$. Thus, it follows from Case~\ref{cycle:case1} that $\lambda_i + \lambda_i' = \min \{\bar{x}_i,\bar{x}_{i+1}\}$. Moreover,  $\delta_n = 0$ implies that 
\[
s_n' = \bar{x}_n - \sum_{j=1}^{n-1}\max(0,\Delta \bar{x}_j) = \bar{x}_n - \sum_{j =1}^{n-1} \bigl(\bar{x}_{j+1}-\min \{\bar{x}_j,\bar{x}_{j+1}\}\bigr) = r_3(\bar{x}).
\]
Therefore, we can conclude the objective value is $r_2(\bar{x}) + ar_3(\bar{x})$, where $a = 1$ in this case. Consider the second case $i^*<n$. For $i\in [n-1]$ with $i\neq i^*$, we have $s_{i+1}-s_i=\max\{0,\Delta \bar{x}_i\}$, and $s'_{i+1}-s'_i =\Delta \bar{x}_i-\max\{0,\Delta \bar{x}_i\}$. Hence,
\[
\begin{aligned}
    &s_{i+1}\geq s_i
\quad\text{and}\quad
s'_{i+1}\geq s'_i
\quad
\text{if } \Delta \bar{x}_i\geq 0,\\
&s_{i+1}\leq s_i
\quad\text{and}\quad
s'_{i+1}\leq s'_i
\quad
\text{if } \Delta \bar{x}_i\leq 0.
\end{aligned}
\]
This implies that $\lambda_i+\lambda_i'=\min\{\bar{x}_i,\bar{x}_{i+1}\}$. For the case $i\in [n-1]$ and $i=i^*$, we have
\[
s_{i+1}-s_i
=
\max\{0,\Delta \bar{x}_i\}+r_3(\bar{x})
\geq 0,
\]
showing that $\lambda_i=s_i$. In addition, 
\begin{align*}
s'_{i+1}-s'_i &=
\bar{x}_{i+1}-\bar{x}_i
-\max\{0,\Delta \bar{x}_i\}
-r_3(\bar{x}) \\
&= \min\{\bar{x}_{i+1},\bar{x}_i\} -\bar{x}_i -r_3(\bar{x}) \\
&\leq 0,
\end{align*}
where the inequality follows since $\min\{\bar{x}_{i+1},\bar{x}_i\}\leq \bar{x}_i$ and $r_3(\bar{x})\geq 0$. Therefore, we obtain
\[
\lambda_i+\lambda_i' = s_i+s'_{i+1} =\bar{x}_{i+1} - \max\{0, \Delta \bar{x}_i\} - r_3(\bar{x}) = \min\{\bar{x}_i,\bar{x}_{i+1}\} - r_3(\bar{x}).
\]
Moreover, 
\[
s_n' = \bar{x}_n - \sum_{j=1}^{n-1}\max(0,\Delta \bar{x}_j) - r_3(\bar{x}) = \bar{x}_n - \sum_{j=1}^{n-1}\max(0,\Delta \bar{x}_j) - \bar{x}_n + \sum_{j =1}^{n-1} \bigl(\bar{x}_{j+1}-\min \{\bar{x}_j,\bar{x}_{j+1}\}\bigr)) = 0.
\]
Therefore,  we obtain a feasible objective value of $r_2(\bar{x})+ar_3(\bar{x})$. This completes the proof as 
\[
    \conv(f)(\bar{x})\leq D(\bar{x}) \leq r_2(\bar{x})+ar_3(\bar{x}) = r(\bar{x}).
\]

\textbf{Case~\ref{cycle:case3}.}
In this case, we use the variable switching transformation $\tilde{x}=\mathbf{1}-x$ to reduce Case~\ref{cycle:case3} to Case~\ref{cycle:case1}. To do so, we define switched functions as follows. For $\tilde{x} \in \{0,1\}^n$, let $\tilde{f}(\tilde{x}):=f(\mathbf{1}-\tilde{x})$, $\tilde{r}(\tilde{x}) = r(\mathbf{1} - \tilde{x})$, and for $k =1,2,3$, $\tilde{r}_k(\tilde{x}) = r_{k}(\mathbf{1} - \tilde{x})$. Since convexification commutes with affine transformations, we have $\conv(f)(x) = \conv(\tilde{f})(\mathbf{1} - x)$ for $x \in [0,1]^n$. The switched functions differ from their original counterparts by affine functions. 
For $\tilde{x}\in\{0,1\}^n$, we have
\begin{align*}
\tilde{f}(\tilde{x}) =\sum_{i\in[n-1]}a_i(1-\tilde{x}_i)(1-\tilde{x}_{i+1})+(1-\tilde{x}_1)(1-\tilde{x}_n)=f(\tilde{x})-L(\tilde{x}),
\end{align*}
where $L(\tilde{x}):=\sum_{i\in[n-1]}a_i(\tilde{x}_i+\tilde{x}_{i+1}-1)+\tilde{x}_1+\tilde{x}_n-1$. Since $L$ is affine, it follows that $\conv(\tilde{f})(\tilde{x}) = \conv(f)(\tilde{x}) - L(\tilde{x})$ for $\tilde{x} \in [0,1]^n$. We next show that $\tilde{r}(\tilde{x})=r(\tilde{x})-L(\tilde{x})$ for all $\tilde{x}\in [0, 1]^n$. First, observe that 
\begin{align*}
\tilde{r}_1(\tilde{x})
&=\max\{0,1-\tilde{x}_1-\tilde{x}_n\}\\
&=\max\{0,\tilde{x}_1+\tilde{x}_n-1\}-(\tilde{x}_1+\tilde{x}_n-1)\\
&=r_1(\tilde{x})-(\tilde{x}_1+\tilde{x}_n-1),
\end{align*}
where the second equality follows from $\max\{0,-y\}=\max\{0,y\}-y$. In addition, 
\begin{align*}
\tilde{r}_2(\tilde{x})
&=\sum_{i\in[n-1]}a_i\bigl(1-\max\{\tilde{x}_i,\tilde{x}_{i+1}\}\bigr)\\
&=\sum_{i\in[n-1]}a_i\bigl(1-\tilde{x}_i-\tilde{x}_{i+1}+\min\{\tilde{x}_i,\tilde{x}_{i+1}\}\bigr)\\
&=r_2(\tilde{x})-\sum_{i\in[n-1]}a_i(\tilde{x}_i+\tilde{x}_{i+1}-1),
\end{align*}
where the second equality follows from $\max\{u,v\}=u+v-\min\{u,v\}$. Finally,
\begin{align*}
\tilde{r}_3(\tilde{x})
&=\sum_{i\in[n-1]}\bigl(1-\max\{\tilde{x}_i,\tilde{x}_{i+1}\}\bigr)
  -\max\{0,1-\tilde{x}_1-\tilde{x}_n\}
  -\sum_{i=2}^{n-1}(1-\tilde{x}_i)\\
&=\sum_{i\in[n-1]}
  \bigl(1-\tilde{x}_i-\tilde{x}_{i+1}+\min\{\tilde{x}_i,\tilde{x}_{i+1}\}\bigr)\\
&\qquad
  -\bigl(\max\{0,\tilde{x}_1+\tilde{x}_n-1\}-(\tilde{x}_1+\tilde{x}_n-1)\bigr)
  -\sum_{i=2}^{n-1}(1-\tilde{x}_i)\\
&=\sum_{i\in[n-1]}\min\{\tilde{x}_i,\tilde{x}_{i+1}\}
  -\max\{0,\tilde{x}_1+\tilde{x}_n-1\}
  -\sum_{i=2}^{n-1}\tilde{x}_i\\
&=r_3(\tilde{x}),
\end{align*}
where the second equality follows from the two identities used above. Thus, we can conclude that $\tilde{r}(\tilde{x}) = \tilde{r}_1(\tilde{x}) + \tilde{r}_2(\tilde{x}) + \tilde{r}_3(\tilde{x}) = r(\tilde{x}) - L(\tilde{x})$. 

Now, we are ready to prove that Case~\ref{cycle:case3}. Let $\bar{x}$ satisfy the conditions in  Case~\ref{cycle:case3}. Then, $(1-\bar{x}_1) + (1-\bar{x}_n) = 2-\bar{x}_1 -\bar{x}_n \leq 1$ and $r_3(\mathbf{1}-\bar{x}) = \tilde{r}_3(\bar{x}) =r_3(\bar{x}) \leq 0$. In other words, $\mathbf{1}-\bar{x}$ satisfies the conditions in  Case~\ref{cycle:case1}. Therefore,
\begin{equation}\label{eq:switchCycle}
\begin{aligned}
    r(\bar{x}) &= \tilde{r}(\mathbf{1}-\bar{x}) = r(\mathbf{1}-\bar{x}) - L(\mathbf{1}-\bar{x}) \geq \conv(f)(\mathbf{1}-\bar{x}) - L(\mathbf{1}-\bar{x}) \\
    &= \conv(\tilde{f})(\mathbf{1}-\bar{x}) =\conv(f)(\bar{x}),
\end{aligned}
\end{equation}
where the first equality holds by definition, the second holds since $\tilde{r}(\tilde{x}) = r(\tilde{x})-L(\tilde{x})$ for every $\tilde{x} \in [0,1]^n$, the inequality follows  by applying the result established in Case~\ref{cycle:case1} to $\mathbf{1}-\bar{x}$, and the last two equalities have been established above.

\textbf{Case~\ref{cycle:case4}.}
We prove this case by reducing it to Case~\ref{cycle:case2}. Let $\bar{x}\in[0,1]^n$ satisfy the conditions of Case~\ref{cycle:case4}. Then, $(1-\bar{x}_1) + (1-\bar{x}_n) = 2-\bar{x}_1 -\bar{x}_n \leq 1$ and $r_3(\mathbf{1}-\bar{x}) = \tilde{r}_3(\bar{x}) =r_3(\bar{x}) \geq 0$. Hence, $\mathbf{1}-\bar{x}$ satisfies the conditions in  Case~\ref{cycle:case2}. Then, $r(\bar{x}) \ge \conv(f)(\bar{x})$ using the same arguments as in \eqref{eq:switchCycle} except that, in this case, the inequality is shown using Case~\ref{cycle:case2}.  This completes the proof. \qed
\end{proof}

\section{Computational Results}\label{sec:computation}

In this section, we evaluate the strength of the cycle inequalities~\eqref{eq:cycle-env} by comparing them against the standard McCormick relaxation for unconstrained binary bilinear programs. All computational experiments were implemented in Julia~\cite{bezanson2017julia},
with optimization models formulated using JuMP~\cite{lubin2023jump} and solved
using Gurobi 12.0.3~\cite{Gurobi2025}. The experiments were conducted on a
machine equipped with an Intel(R) Core(TM) Ultra 7 265K CPU at 3.90 GHz and
48 GB of RAM.

Given a graph $G=([n],E)$, we define $f(x)=\sum_{(i,j)\in E}a_{ij}x_ix_j$ for $x\in\{0,1\}^n$, and the associated bilinear program as $\min\bigl\{f(x) \mid x\in\{0,1\}^n\bigr\}$. The standard McCormick relaxation introduces an auxiliary variable $y_{ij}$ for each edge $(i,j)\in E$, resulting in the LP relaxation:
\begin{equation}\label{eq:mcRela}
\begin{aligned}
\min \quad & \sum_{(i,j)\in E} a_{ij}y_{ij} \\
\text{s.t.}\quad
& \max\{0,x_i+x_j-1\}\leq y_{ij}\leq \min\{x_i,x_j\},
\quad \for (i,j)\in E,\\
& x\in[0,1]^n .
\end{aligned}
\tag{\textsc{McCormick}}
\end{equation}
Besides being weak, a key limitation of the McCormick relaxation is its dimensionality: it requires $|E|$ auxiliary variables, and, therefore, its size grows quadratically with $n$ for dense graphs. In contrast, our cycle inequalities are derived directly in the $x$-space. We integrate them into a compact "Two-$n$" relaxation that uses only $2n+1$ variables, regardless of graph density. This formulation is given by:
\begin{equation}\label{eq:2nRela}
\begin{aligned}
\min_{s,y,x}\quad
& s,\\
\text{s.t.}\quad
& s\geq \sum_{i\in[n]}y_i,\\
& y_i\geq 
\max\bigg\{
L_i x_i,\,
\sum_{j>i,(i,j)\in E}a_{ij}x_j+U_i(x_i-1)
\bigg\},
\quad \for i\in[n],\\
& x\in[0,1]^n .
\end{aligned}
\tag{\textsc{Two-$n$}}
\end{equation}
where $L_i$ and $U_i$ denote the lower and upper bounds, respectively, of $\sum_{j>i,(i,j)\in E}a_{ij}x_j$. The variables $x_i$ are ordered according to the non-increasing degree of the corresponding nodes. Furthermore, when the integrality constraint $x\in\{0,1\}^n$ is enforced, Formulation~\eqref{eq:2nRela} becomes an exact mixed-integer linear formulation of the original binary bilinear optimization problem.


Since the number of valid cycle inequalities is exponential, we employ a cutting-plane approach that starts from the solution of~\eqref{eq:2nRela} and iteratively separates violated inequalities. We limit cycle lengths to a chosen parameter $l$ and pre-enumerate all such cycles. For a fractional solution $(\bar{x},\bar{y},\bar{s})$, we identify violated cuts associated with the global epigraph variable $s$, individual epigraph variables $y_i$, and specific cycles. The associated graph for $s$ is $G$. For $y_i$, it is a star and, for a cycle, the graph is a set of stars (corresponding to the epigraph variables $y_i$) that covers the cycle. For each associated subgraph, we decompose the edge set into a collection of edge-disjoint cycles and remaining edges. The separation procedure constructs affine underestimators by combining the cycle-specific underestimators (derived from Theorem~\ref{theo:cycle-envelope}) with the McCormick inequalities for the remaining edges. This aggregation yields lower-bounding inequalities for $s$, $y_i$, or the sum of $y_i$ that covers the cycle. Details of the decomposition and the separation oracle are provided in Appendix~\ref{App:sep-cycle}. For individual $y_i$ and cycle-associated sums of $y_i$, we additionally generate affine overestimators by applying the same separation procedure to the negative of the corresponding bilinear function and reversing the resulting inequality. For $s$, we generate only underestimating inequalities since $s$ is the minimized global epigraph variable. The process repeats for at most $T$ iterations or until no violated inequalities remain.

We investigate seven distinct classes of graphs $G$. The first three classes are dense and consist of complete graphs with random weights, bipartite graphs with random weights, and complete graphs with edge weights derived from a Hadamard matrix. These are adopted from \cite{luedtke2012some,boland2017bounding}, where it was shown that the McCormick relaxation may be particularly weak on these instances. Motivated by \cite{gupte2020extended}, we additionally study cactus graphs and Halin graphs, which, although sparse, contain cycles so that \eqref{eq:mcRela} often exhibits a relaxation gap. To further extend the cactus graph family, we introduce two additional variants in which each cycle contains at most one chord and at most two chords, respectively. For all graph classes, except for the one based on the Hadamard matrix, the weight $a_{ij}$ for each edge $(i,j) \in E$ is independently set to $1$ with probability $\theta$ and to $-1$ with probability $1-\theta$. The specific topological construction of each graph class is detailed as follows:

\begin{itemize}
    \item \textbf{Complete graph with random weights.} 
    $G = ([n], E)$ where $E = \bigl\{(i, j) \mid i < j, i,j\in [n]\bigr\}$.

    \item \textbf{Bipartite graph with random weights.}
    $G = (V_1 \cup V_2, E)$ with $V_1 = \{1, \ldots, \lfloor n/2\rfloor\}$, $V_2 = \{\lfloor n/2\rfloor + 1, \ldots, n\}$, and $E=\{(i, j) \mid i\in V_1, j\in V_2\}$.

    \item \textbf{Complete graph with Hadamard matrix.}
    $G$ is complete. Let $k = \lceil \log_2(n) \rceil$. Each vertex $i \in [n]$ is associated with a binary vector $v_i \in \{0,1\}^k$ (binary representation of $i-1$). The edge weight is defined as the Hadamard matrix, where $a_{ij} = (-1)^{v_i^\intercal v_j}$.

    \item \textbf{Cactus graph with random weights.}
    Constructed iteratively starting with $V=\{1\}$ and $E=\emptyset$. The remaining $n-1$ vertices are partitioned into random disjoint subsets $S_1, \ldots, S_k$ with $2 \le |S_m| \le 5$. For each $S_m$, a vertex $u$ is chosen uniformly from current $V$. A simple cycle is formed on $\{u\} \cup S_m$, its edges are added to $E$, and $V \leftarrow V \cup S_m$.

    \item \textbf{Cactus graph with one chord per cycle.}
    Built as above. After cycle generation, for each cycle of length $\ge 4$, two uniformly random non-adjacent vertices are selected, and the corresponding edge is added to $E$.

    \item \textbf{Cactus graph with two chords per cycle.}
    Constructed identically to the one-chord variant, except up to two uniformly random edges are added to each cycle of length $\ge 4$.

    \item \textbf{Halin graph with random weights.}
    $G = (V, E_T \cup E_C)$, where $E_T$ and $E_C$ induce a tree and a cycle subgraph respectively. $E_T$ is generated by starting with a root connected to three leaves and iteratively attaching a new node to a random internal node until $|V|=n$. Let $L \subset V$ be the set of leaves. A depth-first search yields an ordering $v_1, \ldots, v_{|L|}$ of $L$. The cycle edges $E_C = \{(v_i, v_{i+1}) \mid 1 \le i < |L|\} \cup \{(v_{|L|}, v_1)\}$ connect the leaves in sequence.
\end{itemize}

For the dense graph classes (Tables~\ref{tab:com-1}--\ref{tab:sep-had}), we restrict cuts to those associated with the global variable $s$ and individual variables $y_i$, as cycle-specific cuts offer marginal gains on these graphs. We test instances with $n \in \{60, 80, 100\}$ and edge probability $\theta \in \{0.4, 0.5, 0.6\}$, generating 10 instances per parameter set. The maximum cycle length $l$ is set to 3 for complete graphs, where triangles dominate other odd-cycle inequalities, and is set to 4 for bipartite graphs. We perform $T \in \{20, 40, 60\}$ cut-generation iterations, and add these inequalities to \eqref{eq:2nRela} resulting in the formulation  \textsc{Cycle}-$T$. Observe that each cut-generation iteration typically generates several cuts.

For the sparse graph classes (Tables~\ref{tab:cac}--\ref{tab:halin}), we include all three classes of cuts and test cycle lengths $l \in \{3, 5\}$ with $T=10$ iterations. We test instances with $n \in \{1000, 1500, 2000\}$ and $\theta \in \{0.4, 0.5, 0.6\}$, generating $10$ instances per parameter set.

We report the average optimal value (Opt), the LP bounds from \eqref{eq:mcRela} and \textsc{Cycle}-$T$, and their solution times. We quantify the improvement using the \textit{Gap Closed} metric defined below:
\[
\text{Gap Closed} = \frac{\text{Bound}_{\text{Cycle-$T$}} - \text{Bound}_{\text{McCormick}}}{\text{Opt} - \text{Bound}_{\text{McCormick}}}.
\]

Tables~\ref{tab:com-1},~\ref{tab:bip-1} and~\ref{tab:had-1} demonstrate that cycle inequalities close over $50\%-99\%$ of the optimality gap left by the McCormick relaxation for complete graphs, while $15\%-30\%$ of the gap for bipartite graphs. Crucially, the \textsc{Cycle}-20 formulation solves faster than the McCormick relaxation despite the additional cutting-plane iterations. While gap closure decreases slightly as $n$ and $\theta$ increase, the method remains robust, closing $\approx 60\%$ of the gap for large complete and Hadamard graphs.

Figures~\ref{fig:complete-graph}--\ref{fig:Hadamard} show that increasing iterations to 40 or 60 further improves gap closure (e.g., while \textsc{Cycle}-20 closes only $15\%-30\%$ of the gap left by \eqref{eq:mcRela}, \textsc{Cycle}-60 closes as much as $30\%-75\%$ of this gap for bipartite graphs) with solution time comparable to \eqref{eq:mcRela}. The separation procedure is highly efficient; total separation times remain under 10 seconds even for 60 iterations on $n=100$ instances (Tables~\ref{tab:sep-com}, \ref{tab:sep-bip}, \ref{tab:sep-had}).

\begin{table}[htbp]
    \centering
    \begin{tabular*}{0.95\textwidth}{@{\extracolsep{\fill}}cc c cc cccc@{}}
    \toprule
    \multirow{2}{*}{$n$} & \multirow{2}{*}{$\theta$} & \multirow{2}{*}{Opt} & \multicolumn{2}{c}{McCormick} & \multicolumn{4}{c}{Cycle-20} \\
    \cmidrule(lr){4-5} \cmidrule(lr){6-9}
    & & & Bound & Time(s) & Bound & Gap closed(\%) & Time(s) & \#cuts \\
    \midrule
    \multirow{3}{*}{60} & 0.4 & -742.0 & -1065.7 & 0.02 & -744.4 & 99.2 & 0.00 & 656.1 \\
     & 0.5 & -295.4 & -898.4 & 0.01 & -479.7 & 69.4 & 0.01 & 711.4 \\
     & 0.6 & -113.2 & -710.0 & 0.02 & -332.7 & 63.2 & 0.01 & 729.1 \\
    \multirow{3}{*}{80} & 0.4 & -1285.6 & -1895.9 & 0.04 & -1306.6 & 96.6 & 0.01 & 944.2 \\
     & 0.5 & -436.6 & -1584.6 & 0.03 & -826.0 & 66.1 & 0.01 & 937.0 \\
     & 0.6 & -139.0 & -1250.8 & 0.02 & -581.3 & 60.2 & 0.01 & 957.1 \\
    \multirow{3}{*}{100} & 0.4 & -2107.2 & -2999.7 & 0.08 & -2111.2 & 99.6 & 0.01 & 1183.7 \\
     & 0.5 & -650.2 & -2494.9 & 0.05 & -1292.2 & 65.2 & 0.02 & 1141.0 \\
     & 0.6 & -195.6 & -2002.4 & 0.05 & -933.9 & 59.1 & 0.02 & 1170.1 \\
    \bottomrule
    \end{tabular*}
    \caption{Comparison of \eqref{eq:mcRela} and \textsc{Cycle}-20 on complete graphs}
    \label{tab:com-1}
\end{table}

\begin{figure}[htbp]
\centering

\begin{minipage}[c]{0.5\textwidth}
\centering
\includegraphics[width=0.88\linewidth]{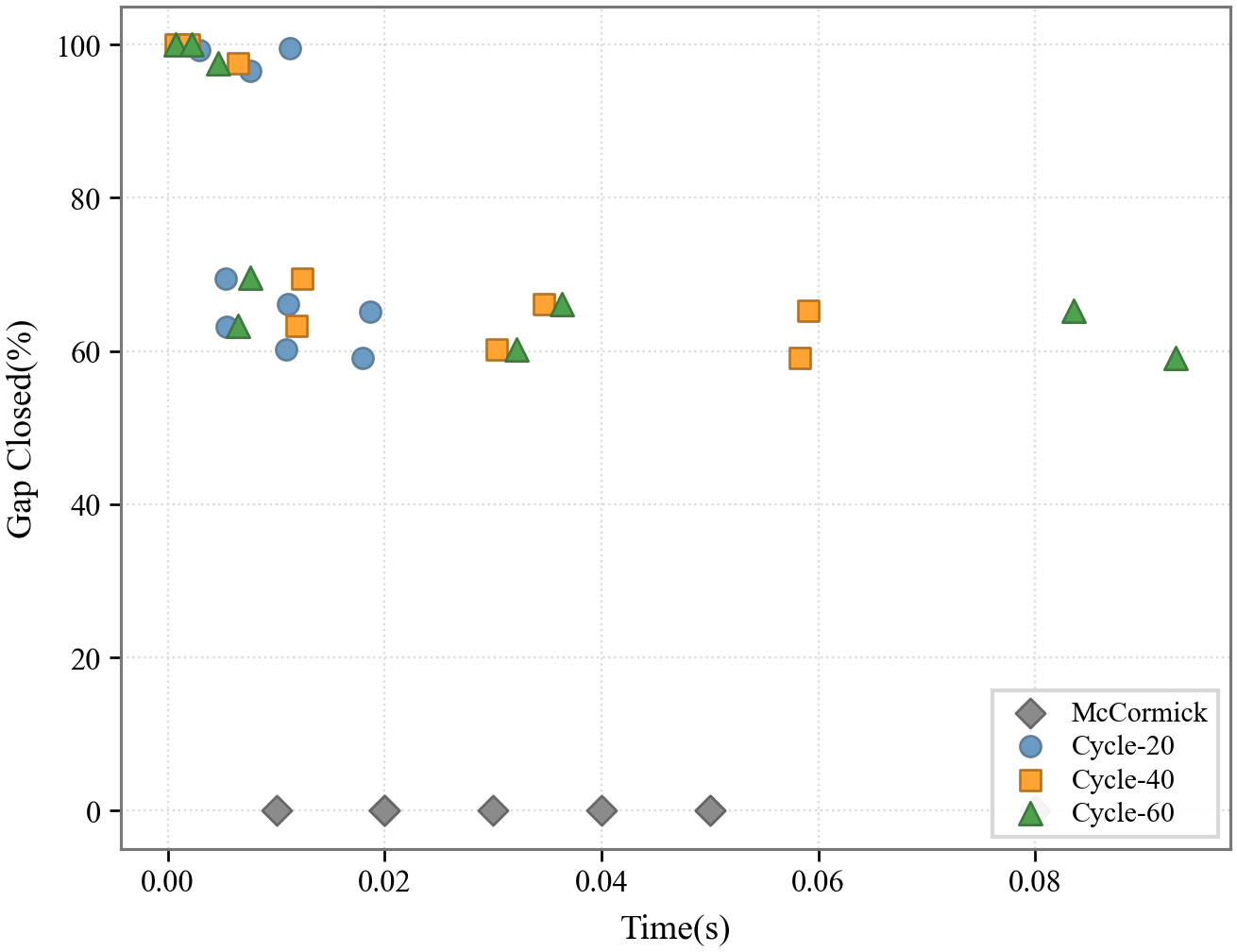}
\caption{Comparison of different methods on complete graphs}
\label{fig:complete-graph}
\end{minipage}
\hfill
\begin{minipage}[c]{0.44\textwidth}
\centering
\scriptsize

\begin{tabular*}{\linewidth}{@{\extracolsep{\fill}}ccccc@{}}
\toprule
\multirow{2}{*}{$n$} & \multirow{2}{*}{$\theta$} & \multicolumn{3}{c}{SepTime(s)} \\
\cmidrule(lr){3-5}
& & Cycle-20 & Cycle-40 & Cycle-60 \\
\midrule
\multirow{3}{*}{60} & 0.4 & 0.118 & 0.184 & 0.236 \\
 & 0.5 & 0.116 & 0.390 & 0.680 \\
 & 0.6 & 0.121 & 0.391 & 0.673 \\
\multirow{3}{*}{80} & 0.4 & 0.222 & 0.491 & 0.708 \\
 & 0.5 & 0.252 & 0.909 & 1.800 \\
 & 0.6 & 0.244 & 0.843 & 1.673 \\
\multirow{3}{*}{100} & 0.4 & 0.422 & 0.774 & 1.086 \\
 & 0.5 & 0.476 & 1.738 & 3.626 \\
 & 0.6 & 0.459 & 1.673 & 3.491 \\
\bottomrule
\end{tabular*}
\captionof{table}{Separation time for complete graphs with $T\in \{20,40,60\}$ rounds.}
\label{tab:sep-com}
\end{minipage}
\end{figure}

\begin{table}[htbp]
    \centering
    \begin{tabular*}{0.95\textwidth}{@{\extracolsep{\fill}}cc c cc cccc@{}}
    \toprule
    \multirow{2}{*}{$n$} & \multirow{2}{*}{$\theta$} & \multirow{2}{*}{Opt} & \multicolumn{2}{c}{McCormick} & \multicolumn{4}{c}{Cycle-20} \\
    \cmidrule(lr){4-5} \cmidrule(lr){6-9}
    & & & Bound & Time(s) & Bound & Gap closed(\%) & Time(s) & \#cuts \\
    \midrule
    \multirow{3}{*}{60} & 0.4 & -399.0 & -540.2 & 0.01 & -494.2 & 32.6 & 0.00 & 613.3 \\
     & 0.5 & -195.4 & -449.4 & 0.01 & -381.8 & 26.6 & 0.00 & 609.3 \\
     & 0.6 & -98.4 & -366.5 & 0.01 & -301.1 & 24.4 & 0.00 & 604.3 \\
    \multirow{3}{*}{80} & 0.4 & -693.2 & -964.4 & 0.02 & -909.6 & 20.2 & 0.00 & 818.3 \\
     & 0.5 & -290.4 & -798.8 & 0.02 & -703.2 & 18.8 & 0.00 & 810.1 \\
     & 0.6 & -139.6 & -635.0 & 0.02 & -537.5 & 19.7 & 0.00 & 802.3 \\
    \multirow{3}{*}{100} & 0.4 & -1041.8 & -1497.0 & 0.03 & -1420.7 & 16.8 & 0.01 & 1019.7 \\
     & 0.5 & -410.2 & -1249.8 & 0.03 & -1131.2 & 14.1 & 0.01 & 1013.9 \\
     & 0.6 & -183.8 & -1006.0 & 0.03 & -875.4 & 15.9 & 0.01 & 1007.7 \\
    \bottomrule
    \end{tabular*}
    \caption{Comparison of \eqref{eq:mcRela} and \textsc{Cycle}-20 for bipartite graphs}
    \label{tab:bip-1}
\end{table}

\begin{figure}[htbp]
\centering

\begin{minipage}[c]{0.5\textwidth}
\centering
\includegraphics[width=0.88\linewidth]{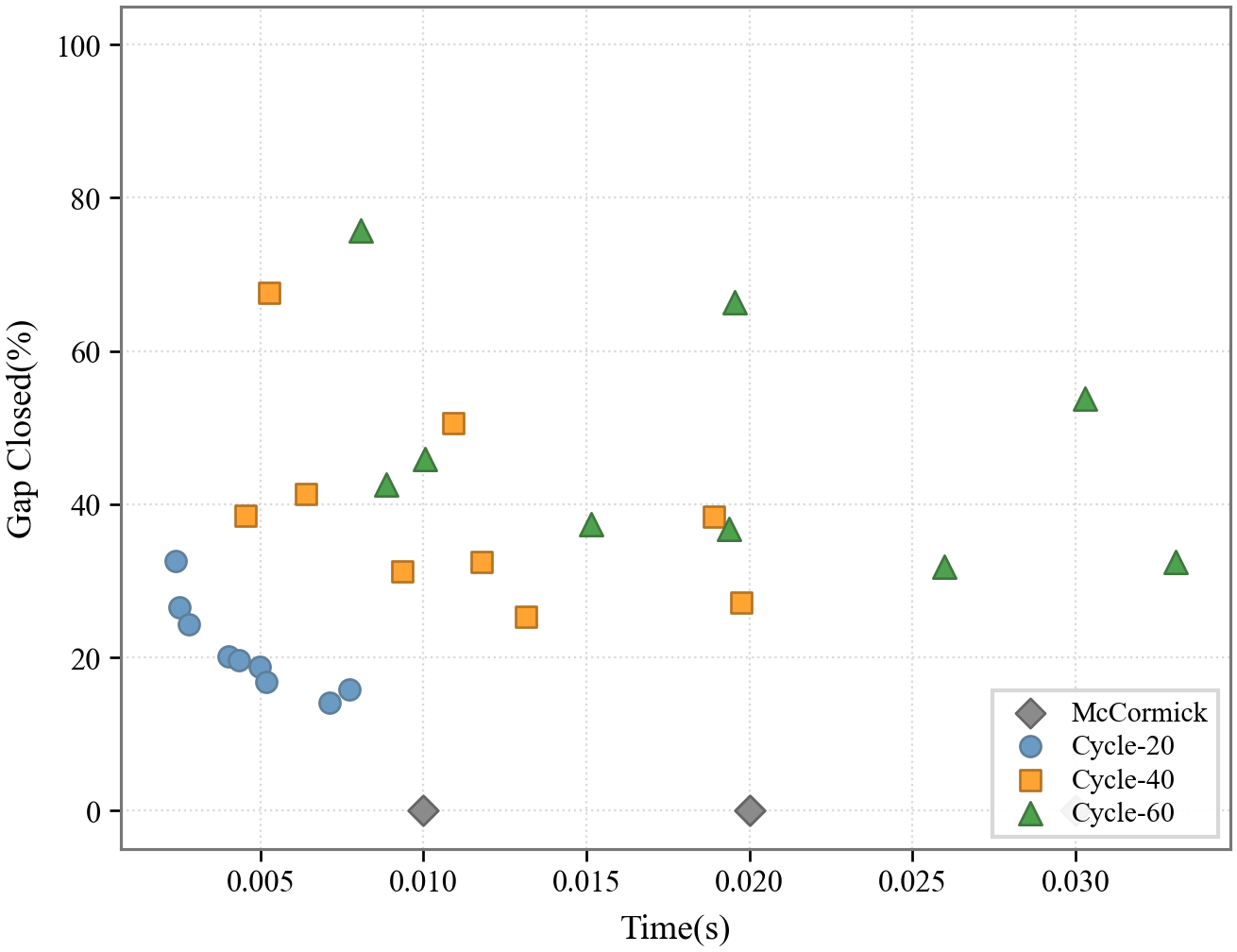}
\caption{Comparison of different methods on bipartite graphs}
\label{fig:bipartite-graph}
\end{minipage}
\hfill
\begin{minipage}[c]{0.44\textwidth}
\centering
\scriptsize

\begin{tabular*}{\linewidth}{@{\extracolsep{\fill}}ccccc@{}}
\toprule
\multirow{2}{*}{$n$} & \multirow{2}{*}{$\theta$} & \multicolumn{3}{c}{SepTime(s)} \\
\cmidrule(lr){3-5}
& & Cycle-20 & Cycle-40 & Cycle-60 \\
\midrule
\multirow{3}{*}{60} & 0.4 & 0.359 & 0.782 & 1.270 \\
 & 0.5 & 0.359 & 0.809 & 1.321 \\
 & 0.6 & 0.356 & 0.782 & 1.285 \\
\multirow{3}{*}{80} & 0.4 & 1.117 & 2.383 & 3.803 \\
 & 0.5 & 1.140 & 2.436 & 3.926 \\
 & 0.6 & 1.131 & 2.397 & 3.797 \\
\multirow{3}{*}{100} & 0.4 & 2.719 & 5.725 & 9.130 \\
 & 0.5 & 2.748 & 5.835 & 9.352 \\
 & 0.6 & 2.760 & 5.741 & 9.092 \\
\bottomrule
\end{tabular*}
\captionof{table}{Separation time for bipartite graphs with $T\in \{20,40,60\}$ rounds.}
\label{tab:sep-bip}
\end{minipage}
\end{figure}

\begin{table}[htbp]
    \centering
    \begin{tabular*}{0.95\textwidth}{@{\extracolsep{\fill}}cc c cc cccc@{}}
    \toprule
    \multirow{2}{*}{$n$} & \multirow{2}{*}{$\theta$} & \multirow{2}{*}{Opt} & \multicolumn{2}{c}{McCormick} & \multicolumn{4}{c}{Cycle-20} \\
    \cmidrule(lr){4-5} \cmidrule(lr){6-9}
    & & & Bound & Time(s) & Bound & Gap closed(\%) & Time(s) & \#cuts \\
    \midrule
    \multirow{1}{*}{60} & 0.5 & -148.0 & -868.0 & 0.01 & -444.2 & 58.9 & 0.01 & 721.0 \\
    \multirow{1}{*}{80} & 0.5 & -244.0 & -1536.0 & 0.02 & -769.7 & 59.3 & 0.01 & 907.0 \\
    \multirow{1}{*}{100} & 0.5 & -346.0 & -2450.0 & 0.04 & -1235.7 & 57.7 & 0.02 & 1084.0 \\
    \bottomrule
    \end{tabular*}
    \caption{Comparison of \eqref{eq:mcRela} and \textsc{Cycle}-20 for complete graphs with Hadamard Matrix}
    \label{tab:had-1}
\end{table}

\begin{figure}[htbp]
\centering

\begin{minipage}[c]{0.5\textwidth}
\centering
\includegraphics[width=0.88\linewidth]{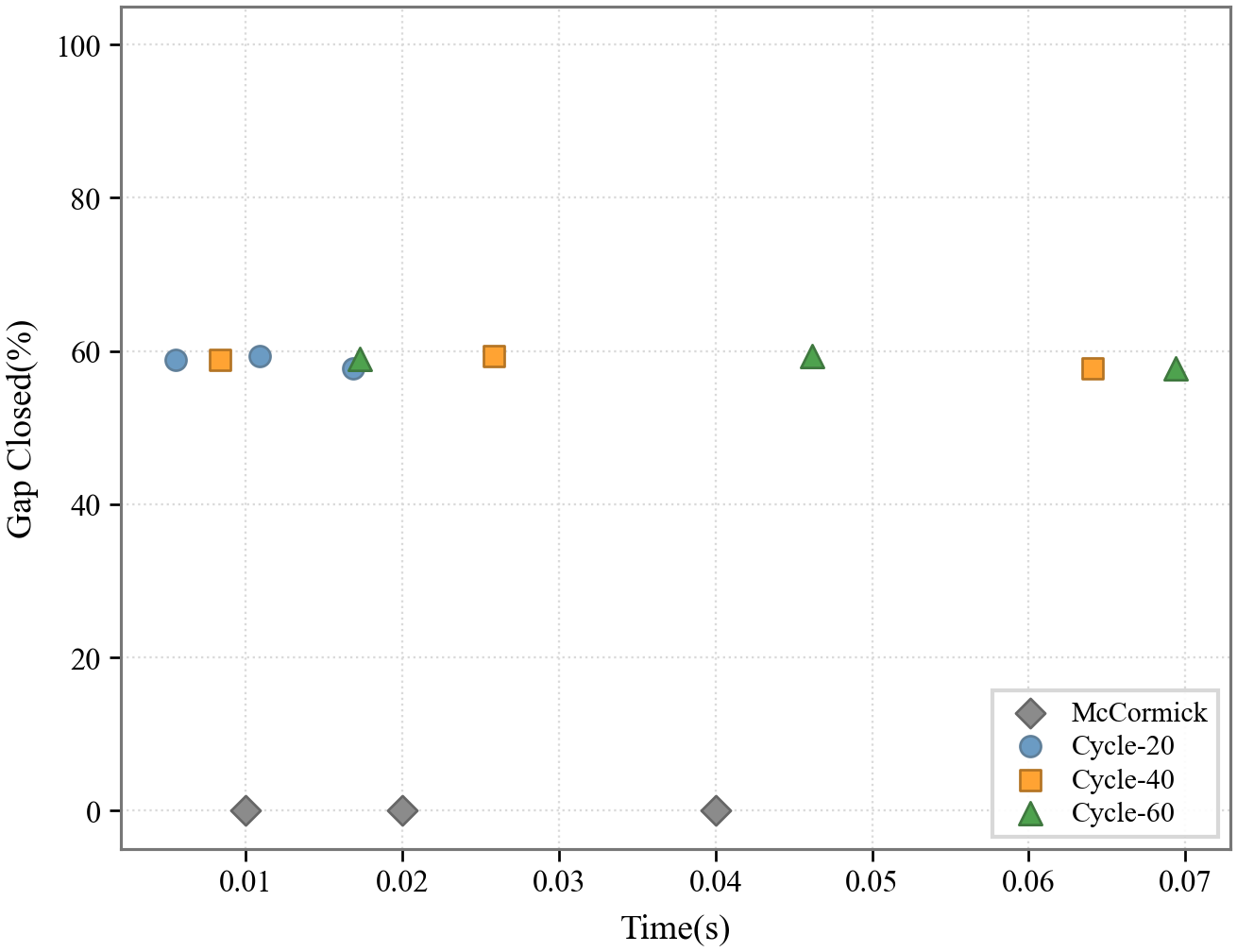}
\caption{Comparison of different methods on Hadamard matrix}
\label{fig:Hadamard}
\end{minipage}
\hfill
\begin{minipage}[c]{0.44\textwidth}
\centering
\scriptsize

\begin{tabular*}{\linewidth}{@{\extracolsep{\fill}}ccccc@{}}
\toprule
\multirow{2}{*}{$n$} & \multirow{2}{*}{$\theta$} & \multicolumn{3}{c}{SepTime(s)} \\
\cmidrule(lr){3-5}
& & Cycle-20 & Cycle-40 & Cycle-60 \\
\midrule
\multirow{1}{*}{60} & 0.5 & 0.123 & 0.420 & 0.753 \\
\multirow{1}{*}{80} & 0.5 & 0.240 & 0.901 & 1.783 \\
\multirow{1}{*}{100} & 0.5 & 0.444 & 1.598 & 3.664 \\
\bottomrule
\end{tabular*}
\captionof{table}{Separation time for complete graphs with Hadamard Matrix, where $T\in \{20,40,60\}$ rounds.}
\label{tab:sep-had}
\end{minipage}
\end{figure}

For the sparse graph classes (Tables~\ref{tab:cac}--\ref{tab:halin}), we include all cut types and compare \eqref{eq:mcRela} with \textsc{Cycle}-10 where maximum cycle length $l$ is chosen to be $3$ and $5$ respectively. Cycle inequalities of length 3 generate weak cuts for sparse graphs, often closing less than $60\%$ of the gap, because these graphs contain long chordless cycles. Increasing the maximum cycle length to 5 substantially improves relaxation quality, closing up to $98\%$ of the gap. The \textsc{Cycle}-10 formulation remains significantly faster than the McCormick relaxation, with solution times often below 0.004 seconds that are rounded to 0.00 in the tables. The cutting-plane procedure converges within 10 iterations for most sparse graphs, making the overall separation overhead negligible. This overhead is, therefore, not reported.

\begin{table}[htbp]
    \centering
    \resizebox{\textwidth}{!}{
    \begin{tabular}{cc c cc cccc cccc}
    \toprule
    \multirow{2}{*}{$n$} & \multirow{2}{*}{$\theta$} & \multirow{2}{*}{Opt} & \multicolumn{2}{c}{McCormick} & \multicolumn{4}{c}{Cycle-10 ($l=3$)} & \multicolumn{4}{c}{Cycle-10 ($l=5$)} \\
    \cmidrule(lr){4-5} \cmidrule(lr){6-9} \cmidrule(lr){10-13}
    & & & Bound & Time(s) & Bound & Gap closed(\%) & Time(s) & \#cut & Bound & Gap closed(\%) & Time(s) & \#cuts \\
    \midrule
    \multirow{3}{*}{1000} & 0.4 & -1130.0 & -1131.8 & 0.01 & -1131.0 & 44.4 & 0.00 & 278.4 & -1130.5 & 72.2 & 0.00 & 547.7 \\
     & 0.5 & -878.0 & -881.5 & 0.01 & -879.6 & 54.3 & 0.00 & 283.9 & -878.8 & 77.1 & 0.00 & 463.3 \\
     & 0.6 & -704.4 & -707.2 & 0.01 & -705.7 & 53.6 & 0.00 & 277.5 & -704.8 & 85.7 & 0.00 & 444.8 \\
    \multirow{3}{*}{1500} & 0.4 & -1690.2 & -1692.4 & 0.01 & -1691.1 & 59.1 & 0.00 & 374.6 & -1690.6 & 81.8 & 0.00 & 719.8 \\
     & 0.5 & -1344.0 & -1347.5 & 0.02 & -1345.4 & 60.0 & 0.00 & 399.7 & -1344.8 & 77.1 & 0.00 & 677.9 \\
     & 0.6 & -1076.8 & -1080.1 & 0.01 & -1078.4 & 51.5 & 0.00 & 407.8 & -1077.2 & 87.9 & 0.00 & 670.3 \\
    \multirow{3}{*}{2000} & 0.4 & -2244.2 & -2247.2 & 0.01 & -2245.4 & 60.0 & 0.00 & 523.2 & -2244.6 & 86.7 & 0.00 & 998.7 \\
     & 0.5 & -1811.8 & -1815.8 & 0.01 & -1813.1 & 67.5 & 0.00 & 538.9 & -1812.2 & 90.0 & 0.00 & 947.7 \\
     & 0.6 & -1414.0 & -1419.0 & 0.02 & -1416.0 & 60.0 & 0.00 & 559.2 & -1414.6 & 88.0 & 0.00 & 889.7 \\
    \bottomrule
    \end{tabular}
    }
    \caption{Comparison of~\eqref{eq:mcRela}, \textsc{Cycle}-10 derived from triangles, and \textsc{Cycle}-10 derived from cycles of length $5$ on cactus graph}
    \label{tab:cac}
\end{table}

\begin{table}[htbp]
    \centering
    \resizebox{\textwidth}{!}{
    \begin{tabular}{cc c cc cccc cccc}
    \toprule
    \multirow{2}{*}{$n$} & \multirow{2}{*}{$\theta$} & \multirow{2}{*}{Opt} & \multicolumn{2}{c}{McCormick} & \multicolumn{4}{c}{Cycle-10 ($l=3$)} & \multicolumn{4}{c}{Cycle-10 ($l=5$)} \\
    \cmidrule(lr){4-5} \cmidrule(lr){6-9} \cmidrule(lr){10-13}
    & & & Bound & Time(s) & Bound & Gap closed(\%) & Time(s) & \#cut & Bound & Gap closed(\%) & Time(s) & \#cuts \\
    \midrule
    \multirow{3}{*}{1000} & 0.4 & -1215.4 & -1218.3 & 0.01 & -1216.1 & 75.9 & 0.00 & 780.3 & -1215.5 & 96.6 & 0.00 & 2507.5 \\
     & 0.5 & -960.6 & -964.2 & 0.01 & -962.1 & 59.7 & 0.00 & 759.4 & -960.7 & 98.1 & 0.00 & 2173.2 \\
     & 0.6 & -737.2 & -743.6 & 0.01 & -740.4 & 49.6 & 0.00 & 767.1 & -738.0 & 87.6 & 0.00 & 1995.8 \\
    \multirow{3}{*}{1500} & 0.4 & -1837.2 & -1842.0 & 0.01 & -1839.0 & 63.5 & 0.00 & 1173.1 & -1837.3 & 97.9 & 0.00 & 3779.8 \\
     & 0.5 & -1428.8 & -1436.3 & 0.01 & -1432.3 & 53.9 & 0.00 & 1171.1 & -1429.4 & 92.5 & 0.00 & 3324.5 \\
     & 0.6 & -1098.4 & -1106.8 & 0.02 & -1102.4 & 51.8 & 0.00 & 1221.7 & -1098.9 & 94.2 & 0.00 & 3344.0 \\
    \multirow{3}{*}{2000} & 0.4 & -2482.0 & -2486.8 & 0.02 & -2483.7 & 64.6 & 0.00 & 1583.3 & -2482.5 & 90.5 & 0.00 & 5193.5 \\
     & 0.5 & -1917.6 & -1924.5 & 0.02 & -1920.4 & 59.0 & 0.00 & 1534.8 & -1917.8 & 96.9 & 0.00 & 4480.0 \\
     & 0.6 & -1492.6 & -1505.9 & 0.02 & -1497.1 & 65.9 & 0.00 & 1565.9 & -1493.7 & 91.8 & 0.00 & 4295.0 \\
    \bottomrule
    \end{tabular}
    }
    \caption{Comparison of~\eqref{eq:mcRela}, \textsc{Cycle}-10 derived from triangles, and \textsc{Cycle}-10 derived from cycles of length $5$ on cactus graph with one chord}
    \label{tab:cac_one}
\end{table}

\begin{table}[htbp]
    \centering
    \resizebox{\textwidth}{!}{
    \begin{tabular}{cc c cc cccc cccc}
    \toprule
    \multirow{2}{*}{$n$} & \multirow{2}{*}{$\theta$} & \multirow{2}{*}{Opt} & \multicolumn{2}{c}{McCormick} & \multicolumn{4}{c}{Cycle-10 ($l=3$)} & \multicolumn{4}{c}{Cycle-10 ($l=5$)} \\
    \cmidrule(lr){4-5} \cmidrule(lr){6-9} \cmidrule(lr){10-13}
    & & & Bound & Time(s) & Bound & Gap closed(\%) & Time(s) & \#cut & Bound & Gap closed(\%) & Time(s) & \#cuts \\
    \midrule
    \multirow{3}{*}{1000} & 0.4 & -1476.2 & -1490.4 & 0.01 & -1477.4 & 91.7 & 0.00 & 2154.0 & -1476.8 & 95.9 & 0.00 & 7400.0 \\
     & 0.5 & -1106.2 & -1133.6 & 0.02 & -1108.4 & 92.1 & 0.00 & 2351.5 & -1107.1 & 96.5 & 0.01 & 7263.2 \\
     & 0.6 & -831.0 & -872.2 & 0.01 & -835.6 & 88.8 & 0.00 & 2559.6 & -833.2 & 94.7 & 0.01 & 7457.9 \\
    \multirow{3}{*}{1500} & 0.4 & -2230.8 & -2249.4 & 0.02 & -2232.2 & 92.7 & 0.00 & 3235.9 & -2231.4 & 96.7 & 0.01 & 10952.7 \\
     & 0.5 & -1675.8 & -1719.9 & 0.03 & -1679.7 & 91.2 & 0.01 & 3542.0 & -1678.4 & 94.0 & 0.01 & 10828.3 \\
     & 0.6 & -1255.6 & -1321.2 & 0.02 & -1262.4 & 89.7 & 0.01 & 3839.9 & -1259.2 & 94.6 & 0.01 & 11339.1 \\
    \multirow{3}{*}{2000} & 0.4 & -2952.8 & -2979.1 & 0.04 & -2954.3 & 94.4 & 0.00 & 4158.5 & -2953.8 & 96.4 & 0.01 & 14656.4 \\
     & 0.5 & -2262.8 & -2314.8 & 0.04 & -2269.1 & 87.9 & 0.01 & 4643.5 & -2265.8 & 94.3 & 0.01 & 14548.8 \\
     & 0.6 & -1654.2 & -1745.7 & 0.04 & -1665.4 & 87.8 & 0.01 & 5201.1 & -1661.1 & 92.5 & 0.02 & 15530.9 \\
    \bottomrule
    \end{tabular}
    }
    \caption{Comparison of~\eqref{eq:mcRela}, \textsc{Cycle}-10 derived from triangles, and \textsc{Cycle}-10 derived from cycles of length $5$ on cactus graph with two chords}
    \label{tab:cac_two}
\end{table}

\begin{table}[htbp]
    \centering
    \resizebox{\textwidth}{!}{
    \begin{tabular}{cc c cc cccc cccc}
    \toprule
    \multirow{2}{*}{$n$} & \multirow{2}{*}{$\theta$} & \multirow{2}{*}{Opt} & \multicolumn{2}{c}{McCormick} & \multicolumn{4}{c}{Cycle-10 ($l=3$)} & \multicolumn{4}{c}{Cycle-10 ($l=5$)} \\
    \cmidrule(lr){4-5} \cmidrule(lr){6-9} \cmidrule(lr){10-13}
    & & & Bound & Time(s) & Bound & Gap closed(\%) & Time(s) & \#cut & Bound & Gap closed(\%) & Time(s) & \#cuts \\
    \midrule
    \multirow{3}{*}{1000} & 0.4 & -1367.8 & -1375.5 & 0.01 & -1368.6 & 89.6 & 0.00 & 769.1 & -1368.2 & 94.8 & 0.00 & 2245.5 \\
     & 0.5 & -1075.0 & -1086.9 & 0.01 & -1076.2 & 89.6 & 0.00 & 821.7 & -1075.3 & 97.5 & 0.00 & 2056.9 \\
     & 0.6 & -817.8 & -832.1 & 0.01 & -820.1 & 83.9 & 0.00 & 818.1 & -818.4 & 95.8 & 0.00 & 1913.0 \\
    \multirow{3}{*}{1500} & 0.4 & -2041.2 & -2051.8 & 0.02 & -2042.5 & 87.7 & 0.00 & 1104.6 & -2041.7 & 95.3 & 0.00 & 3329.2 \\
     & 0.5 & -1602.8 & -1621.3 & 0.02 & -1604.9 & 88.8 & 0.00 & 1212.5 & -1603.1 & 98.4 & 0.00 & 3069.4 \\
     & 0.6 & -1227.0 & -1247.0 & 0.02 & -1230.2 & 83.8 & 0.00 & 1214.1 & -1227.9 & 95.4 & 0.00 & 2863.6 \\
    \multirow{3}{*}{2000} & 0.4 & -2780.8 & -2794.3 & 0.03 & -2782.4 & 88.1 & 0.00 & 1506.3 & -2781.2 & 97.0 & 0.00 & 4518.9 \\
     & 0.5 & -2129.2 & -2153.0 & 0.03 & -2131.9 & 88.7 & 0.00 & 1627.4 & -2130.0 & 96.8 & 0.00 & 4147.6 \\
     & 0.6 & -1639.4 & -1670.4 & 0.03 & -1644.2 & 84.5 & 0.00 & 1669.4 & -1641.1 & 94.5 & 0.00 & 3984.0 \\
    \bottomrule
    \end{tabular}
    }
    \caption{Comparison of~\eqref{eq:mcRela}, \textsc{Cycle}-10 derived from triangles, and \textsc{Cycle}-10 derived from cycles of length $5$ on Halin graph}
    \label{tab:halin}
\end{table}

\section{Conclusion}
We introduced disjunctive submodularity to exploit submodular structure over subdomains of the 0-1 hypercube. Our developments rely on a key equivalence between facially submodular functions and pointwise minimum of submodular functions and a lifting argument. Using these, we developed a strongly polynomial separation algorithm to separate convex and sublinear envelopes of disjunctive submodular functions consisting of two disjunctions. We derived closed-form expressions in several settings. These include the the case when one of the functions is linear and the convex hull characterization of a 0-1 bilinear function defined over a cycle interaction graph. The latter was based on a polynomial-sized LP formulation to separate the minimum of $K$ submodular functions when individual functions can be separated using polynomial-sized LP formulations. We applied our advances to a multi-product inventory problem to recover and extend a recent result and showed computationally that our cycle inequalities close significant gap left by the McCormick relaxation.

Several directions are particularly promising for further exploration. First, a systematic structural exploration of typical disjunctive submodular function structures beyond 0-1 bilinear programming and inventory optimization is needed. Second, algorithms are required to automatically identify such structures when they occur in combinatorial or nonlinear optimization settings. Third, efficient implementations of weighted polymatroid intersection algorithm, our lifting, and extension techniques used to derive separation inequalities within a branch-and-bound framework could help assess the benefit of these tighter relaxations for broader classes of optimization problems.


\section{Declarations}
\textbf{Financial interest} The authors did not receive support from any organization for the submitted work.

\noindent \textbf{Competing Interests} The authors have no competing interests, conflict of interest, or non-financial interests to declare that are relevant to the content of the article.

\begin{appendices}
\section{Proof of Theorem~\ref{them:extension-lattice}}\label{proof:theo-extension}
We first establish some basic properties of the retraction operator.
\begin{lemma}\label{lem:retraction-properties}
Let $X$ be a lattice family given by index sets $I^0, I^1$ and a relation $D$. Let $\rho:\{0,1\}^n \to X$  be the retraction operator associated with $X$ defined as in~\eqref{eq:retraction}. Then, $\rho$ satisfies the following properties:
\begin{enumerate}
    \item \label{property:v-1} $\rho(x) \in X$ for all $x \in \{0,1\}^n$.
    \item \label{property:v-2} $\rho(a \vee b) = \rho(a) \vee \rho(b)$ for all $a,b \in \{0,1\}^n$.
    \item \label{property:v-3} $\rho(a \wedge b) \le \rho(a) \wedge \rho(b)$ for all $a,b \in \{0,1\}^n$.
    \item \label{property:v-4} $\rho(x) = x$ for all $x \in X$.
\end{enumerate}
\end{lemma}
 
\begin{proof}
We first prove property~\eqref{property:v-1}. Consider any $x \in \{0,1\}^n$. To show $\rho(x) \in X$, we verify that $\rho(x)$ satisfies the zero constraints on $I^0$, the one constraints on $I^1$, and the relation constraints on $D$. For every $i\in I^0$, we have $(x\wedge\top)_i=0$. Since every successor
of a vertex in $I^0$ also belongs to $I^0$, it follows that
$(x\wedge\top)_j=0$ for every $(i,j)\in D$. Hence, $(c(x\wedge\top))_i =
\max_{j:(i,j)\in D}(x\wedge\top)_j = 0$.  Because $\bot_i=0$, we obtain $\rho(x)_i =
(c(x\wedge\top))_i\vee\bot_i = 0$. Next, let $i\in I^1$. Since $\bot_i=1$, it follows readily that $\rho(x)_i =(c(x\wedge\top))_i\vee\bot_i = 1$.  Finally, let $(i,j)\in D$. Since $D$ is transitive, $\{t:(j,t)\in D\}\subseteq\{t:(i,t)\in D\}$, and therefore
\[
(c(x\wedge\top))_i
=
\max_{t:(i,t)\in D}(x\wedge\top)_t
\ge
\max_{t:(j,t)\in D}(x\wedge\top)_t
=
(c(x\wedge\top))_j.
\]
Moreover, $\bot\in X$, so $\bot_i\ge\bot_j$.  Hence
$ \rho(x)_i = (c(x\wedge\top))_i\vee\bot_i \ge (c(x\wedge\top))_j\vee\bot_j = \rho(x)_j$. 
Thus $\rho(x)\in X$.

We next prove property~\eqref{property:v-2}. Let
$a,b\in\{0,1\}^n$. Then
\begin{align*}
\rho(a\vee b)
&=
c\bigl((a\vee b)\wedge\top\bigr)\vee\bot\\
&=
c\bigl((a\wedge\top)\vee(b\wedge\top)\bigr)\vee\bot\\
&=
\bigl(c(a\wedge\top)\vee c(b\wedge\top)\bigr)\vee\bot\\
&=
\rho(a)\vee\rho(b),
\end{align*}
where the third equality follows because for each $i \in [n]$, $\bigl(c(u\vee v) \bigr)_i = \max_{j:(i,j)\in D}(u_j\vee v_j) = (c(u))_i\vee(c(v))_i$.

For property~\eqref{property:v-3}, we similarly obtain
\begin{align*}
\rho(a\wedge b)
&=
c\bigl((a\wedge b)\wedge\top\bigr)\vee\bot\\
&\le
\bigl(c(a\wedge\top)\wedge c(b\wedge\top)\bigr)\vee\bot\\
&=
\bigl(c(a\wedge\top)\vee\bot\bigr)
\wedge
\bigl(c(b\wedge\top)\vee\bot\bigr)\\
&=
\rho(a)\wedge\rho(b),
\end{align*}
where the inequality follows from $\bigl(c(u\wedge v)\bigr)_i =
\max_{j:(i,j)\in D}(u_j\wedge v_j)
\le
\bigl(c(u)\bigr)_i\wedge \bigl(c(v)\bigr)_i$.

Finally, let $x\in X$. Since $x_i=0$ for every $i\in I^0$, it follows readily that $x\wedge\top=x$.  Moreover, $c(x) = x$ as we have 
\[
(c(x))_i = \max_{j:(i,j)\in D}x_j = x_i,
\]
where reflexivity gives $(i,i)\in D$ and hence
$(c(x))_i\ge x_i$, while the precedence constraints imply $(c(x))_i\le x_i$. Last,  since $x_i=1$ for every $i\in I^1$, $x\vee\bot=x$. Therefore, $\rho(x) = c(x\wedge\top)\vee\bot = c(x)\vee\bot = x$. 
This proves property~\eqref{property:v-4}. 
\end{proof}

\begin{proof}[Proof of Theorem~\ref{them:extension-lattice}] 
First, we prove that $f(x) = \min_{k\in[K]}f_k(x)$ for all $x \in X$. For any $k \in [K]$ and $x \in X$, if $\rho_k(x) = x$ then $f_k(x) = f(x)$. Otherwise, 
\[
f_k(x)-f(x) = f(\rho_k(x))-f(x) + C\|x-\rho_k(x)\|_1 \ge -C + C\|x-\rho_k(x)\|_1 \ge 0,
\]
where the first inequality follows from the fact that $\rho_k(x)\in X_k \subseteq X$ and $f(\rho_k(x))-f(x)\ge f^L-f^U \geq -C$, and the second inequality holds because $x$ and $\rho_k(x)$ are binary vectors, and $\rho_k(x) \ne x$ strictly implies $\|x-\rho_k(x)\|_1\ge 1$. Conversely, since $X = \bigcup_{k\in[K]}X_k$,  for every $x\in X$ there exists some index $k^*$ such that $x\in X_{k^*}$. Therefore, we obtain
\[
f_{k^*}(x) = f(\rho_{k^*}(x)) + C\|x-\rho_{k^*}(x)\|_1 = f(x),
\]
where the second equality holds since the property~\eqref{property:v-4} of Lemma~\ref{lem:retraction-properties} implies that $\rho_{k^*}(x)=x$. 
Hence, $\min_{k\in[K]}f_k(x) \le f_{k^*}(x) = f(x)$. This completes the proof.

Next, we prove $f_k$ is submodular on $X$. The proof consists of two steps. 
\medskip

\noindent\textbf{Step 1. Submodularity on $X_k^\downarrow$.} 
Let $X_k^\downarrow:=\{x \in X \mid x \le \top_k\}$, and let $x,y\in X_k^\downarrow$. Since $x,y\le\top_k$, the retraction can be simplified, that is, $\rho_k(x)=c_k(x)\vee\bot_k\ge x$ and $\rho_k(y)=c_k(y)\vee\bot_k\ge y$. Using the submodularity of $f$ on $X_k$, we obtain 
\begin{equation}\label{eq:step1-main}
\begin{aligned}
f\bigl(\rho_k(x)\bigr)+f\bigl(\rho_k(y)\bigr) &\geq f\bigl(\rho_k(x) \wedge \rho_k(y) \bigr) +  f\bigl(\rho_k(x) \vee \rho_k(y) \bigr) = f\bigl(\rho_k(x) \wedge \rho_k(y) \bigr) +  f\bigl(\rho_k(x \vee  y) \bigr),  \\
& \geq f(\rho_k(x\wedge y))
+f(\rho_k(x\vee y))
-C\|\rho_k(x)\wedge\rho_k(y)-\rho_k(x\wedge y)\|_1,
\end{aligned}
\end{equation}
where the first inequality holds by using the property~\eqref{property:v-1} of Lemma~\ref{lem:retraction-properties}, the equality follows from the property~\eqref{property:v-2} of Lemma~\ref{lem:retraction-properties}, and the second inequality holds since $f(a) - f(b) \leq C\|a-b\|_1$ for $a,b \in X$. 
It remains to handle the penalty terms. Observe that 
\[
\begin{aligned}
\|\rho_k(x)-x\|_1+\|\rho_k(y)-y\|_1
&= \sum_{i\in[n]}(\rho_k(x)_i-x_i)
+
\sum_{i\in[n]}(\rho_k(y)_i-y_i) \\
& = \sum_{i\in[n]}\bigl(\rho_k(x)_i \wedge \rho_k(y)_i -x_i \wedge y_i\bigr) + \sum_{i\in[n]}\bigl(\rho_k(x)_i \vee \rho_k(y)_i -x_i \vee y_i\bigr) \\
&=
\|(\rho_k(x)\wedge\rho_k(y))-(x\wedge y)\|_1
+
\|(\rho_k(x)\vee\rho_k(y))-(x\vee y)\|_1,
\end{aligned}
\]
where the first and third equalities follow from $\rho_k(x)\ge x$ and $\rho_k(y)\ge y$, and the second equality follows from the identity $(a-c)+(b-d) = (a\wedge b-c\wedge d) + (a\vee b-c\vee d)$ for real values $a,b,c,d$. Furthermore, since $x \wedge y \leq \rho_k(x\wedge y)
\le
\rho_k(x)\wedge\rho_k(y)$, we obtain
\[
\|\rho_k(x)\wedge\rho_k(y)-\rho_k(x\wedge y)\|_1 = \|\rho_k(x)\wedge \rho_k(y)-(x\wedge y)\|_1 - \|\rho_k(x\wedge y)-(x\wedge y)\|_1.
\]
Substituting these identities into~\eqref{eq:step1-main}, together with
\(
\rho_k(x)\vee\rho_k(y)=\rho_k(x\vee y),
\)
gives
\[
f_k(x)+f_k(y)
\ge
f_k(x\wedge y)+f_k(x\vee y),
\]
proving that $f_k$ is submodular on $X_k^\downarrow$.

\medskip
\noindent\textbf{Step 2. Submodularity on $X$.} The key observation is the following decomposition: for any $z\in X$,
\begin{equation}
\label{eq:decomp-extended-fun}
\begin{aligned}
f_k(z)
&=f(\rho_k(z))+C\|z-\rho_k(z)\|_1\\
&=f(\rho_k(z\wedge\top_k))+C\|z-\rho_k(z)\|_1\\
&=f(\rho_k(z\wedge\top_k))
+C\|(z\wedge\top_k)-\rho_k(z)\|_1
+C\|z-z\wedge\top_k\|_1\\
&=f_k(z\wedge\top_k)
+C\|z-z\wedge\top_k\|_1,
\end{aligned}
\end{equation}
where the second equality holds by the definition of $\rho_k$, the third equality holds  since the vector
$(z\wedge\top_k)-\rho_k(z)$ is supported only on coordinates where
$(\top_k)_i=1$ whereas the vector
$z-(z\wedge\top_k)$ is supported only on coordinates where
$(\top_k)_i=0$, and the last equality follows from the definition of $f_k$. Now, for $x,y\in X$, we obtain
\begin{align*}
f_k(x)+f_k(y)
&=
f_k(x\wedge\top_k)
+
f_k(y\wedge\top_k)
+
C\|x-x\wedge\top_k\|_1
+
C\|y-y\wedge\top_k\|_1\\
&\ge
f_k\big((x\wedge\top_k)\wedge(y\wedge\top_k)\big)
+
f_k\big((x\wedge\top_k)\vee(y\wedge\top_k)\big)\\
&\qquad
+
C\|x-x\wedge\top_k\|_1
+
C\|y-y\wedge\top_k\|_1\\
&=
f_k\big((x\wedge y)\wedge\top_k\big)
+
f_k\big((x\vee y)\wedge\top_k\big)\\
&\qquad
+
C\|(x\wedge y)-(x\wedge y)\wedge\top_k\|_1
+
C\|(x\vee y)-(x\vee y)\wedge\top_k\|_1\\
&=
f_k(x\wedge y)+f_k(x\vee y),
\end{align*}
where the first and last equalities follow from~\eqref{eq:decomp-extended-fun}, and the inequality follows from Step~1 since
$x\wedge\top_k,y\wedge\top_k\in X_k^\downarrow$. The second equality holds since
\[
\begin{aligned}
\|x-x\wedge\top_k\|_1
+
\|y-y\wedge\top_k\|_1 &=
\sum_{i:(\top_k)_i=0}(x_i+y_i) =
\sum_{i:(\top_k)_i=0}
\bigl((x_i\wedge y_i)+(x_i\vee y_i)\bigr)\\
& = \|(x\wedge y)-(x\wedge y)\wedge\top_k\|_1 + \|(x\vee y)-(x\vee y)\wedge\top_k\|_1.
\end{aligned}
\]
This completes the proof. 
\end{proof}

\section{Proof of Lemma~\ref{lemma:dualsconv}}\label{App:proofs}
\begin{proof}[Proof of Lemma~\ref{lemma:dualsconv}]
The proof consists of two parts. We first establish the conic-combination representation of $\sconv(f)$ and then characterize its closure through linear underestimators of $f$. 

Define
\[
h_1(x)
:=
\inf_{\lambda}
\Bigl\{
\sum_{v\in V}\lambda_v f(v)
\;\Bigm|\;
x=\sum_{v\in V}\lambda_v v,\;
V\subseteq S,\;
\lambda_v\geq 0 \ \for v\in V
\Bigr\}\quad \for x\in \cone(S).
\]
Clearly, $\sconv(f)(x)\leq h_1(x)$ for every $x\in\cone(S)$. Indeed, for any feasible conic representation $x=\sum_{v\in V}\lambda_v v$, 
\[
\sconv(f)(x)
\leq
\sum_{v\in V}\lambda_v\sconv(f)(v)
\leq
\sum_{v\in V}\lambda_v f(v),
\]
where the first inequality follows from subadditivity and positive homogeneity of $\sconv(f)$, and the second inequality follows from $\sconv(f)(v)\leq f(v)$ for $v\in S$. Taking the infimum over all feasible conic representations of $x$ gives $\sconv(f)(x)\leq h_1(x)$.

For the reverse inequality, we show that $h_1$ is a sublinear underestimator of $f$. For positive homogeneity, let $\rho\geq 0$. If $\rho=0$, the zero multipliers give $h_1(\mathbf{0})\leq 0$. On the other hand, by Assumption~\ref{ass:proper}, there exists a linear function $l$ satisfying $l(v)\leq f(v)$ for all $v\in S$. Thus, for any feasible representation $\sum_{v\in V}\lambda_v v=\mathbf{0}$,
\[
\sum_{v\in V}\lambda_v f(v)
\geq
\sum_{v\in V}\lambda_v l(v)
=
l(\mathbf{0})
=
0,
\]
which implies $h_1(\mathbf{0})\geq 0$. Hence, $h_1(\mathbf{0})=0$. If $\rho>0$, scaling the multipliers by $\rho$ gives a one-to-one correspondence between the feasible representations of $x$ and those of $\rho x$, and therefore $h_1(\rho x)=\rho h_1(x)$.

For subadditivity, let $x,y\in\cone(S)$ and consider feasible conic representations
$x=\sum_{u\in U}\mu_u u$ and $y=\sum_{w\in W}\gamma_w w$, where $\mu_u,\gamma_w\geq0$. Combining these two representations gives a feasible conic representation of $x+y$, and therefore
\[
h_1(x+y)
\leq
\sum_{u\in U}\mu_u f(u)
+
\sum_{w\in W}\gamma_w f(w).
\]
Taking the infimum over the two representations yields $h_1(x+y)\leq h_1(x)+h_1(y)$. Together with positive homogeneity, $h_1$ is sublinear on $\cone(S)$. Moreover, $h_1$ is a sublinear underestimator of $f$ since $x=1\cdot x$ is feasible. Thus, $h_1(x)\leq\sconv(f)(x)$ for all $x\in\cone(S)$.

For the second representation, define
\[
\Lambda
:=
\bigl\{
\alpha\in\R^n
\bigm|
\langle\alpha,v\rangle\leq f(v)
\ \for v\in S
\bigr\},
\qquad
h_2(x)
:=
\sup_{\alpha\in\Lambda}\langle\alpha,x\rangle
\quad \for x\in\R^n.
\]
Assumption~\ref{ass:proper} ensures that $\Lambda\neq\emptyset$. Moreover, $h_2$ is closed and sublinear, since it is the pointwise supremum of continuous linear functions on $\R^n$.

We first establish $h_2\leq\cl\sconv(f)$. For every $\alpha\in\Lambda$, the linear function $x\mapsto\langle\alpha,x\rangle$ is a sublinear underestimator of $f$. Hence, $\langle\alpha,x\rangle\leq\sconv(f)(x)$ for all $x\in\cone(S)$. Taking the supremum over $\alpha\in\Lambda$ gives $h_2(x)\leq\sconv(f)(x)$. Since $h_2$ is closed, it follows that $h_2(x)\leq\cl\sconv(f)(x)$ for $x\in \cone(S)$.

For the reverse inequality, we argue by contradiction. Suppose that there exists
$\bar{x}\in\cone(S)$ such that $h_2(\bar{x})<\cl\sconv(f)(\bar{x})$. Choose $\gamma\in\R$ satisfying $h_2(\bar{x})<\gamma<\cl\sconv(f)(\bar{x})$. We will use the separation theorem to construct a vector
$\bar{\alpha}\in\Lambda$ such that
$\gamma<\langle\bar{\alpha},\bar{x}\rangle$, contradicting
$h_2(\bar{x})<\gamma$.

Since $(\bar{x},\gamma)$ lies outside the closed convex cone
$\epi(\cl\sconv(f))$, the separation theorem~\cite[Corollary 11.4.1]{rockafellar1997convex}
implies that there exists a nonzero vector $(a,b)\in\R^n\times\R$ such that
\begin{equation}\label{eq:cone_sep}
\langle a,x\rangle+bs
\geq 0
>
\langle a,\bar{x}\rangle+b\gamma
\quad
\for (x,s)\in\epi(\cl\sconv(f)).
\end{equation}
We first show that $b>0$. Since $(x,s)\in\epi(\cl\sconv(f))$ implies
$(x,s+t)\in\epi(\cl\sconv(f))$ for every $t\geq0$, the left inequality in
\eqref{eq:cone_sep} implies $b\geq0$. If $b=0$, applying the left inequality at
$(\bar{x},\cl\sconv(f)(\bar{x}))$ gives
$\langle a,\bar{x}\rangle\geq0$, whereas the strict inequality in
\eqref{eq:cone_sep} gives $\langle a,\bar{x}\rangle<0$, a contradiction.
Hence, $b>0$.

Set $\bar{\alpha}:=-a/b$. Dividing \eqref{eq:cone_sep} by $b$ gives $\gamma<\langle\bar{\alpha},\bar{x}\rangle$ and $s\geq\langle\bar{\alpha},x\rangle$ for $(x,s)\in\epi(\cl\sconv(f))$. Therefore, $\bar{\alpha}\in\Lambda$ since $\langle\bar{\alpha},x\rangle
\leq\cl\sconv(f)(x)
\leq\sconv(f)(x)
\leq f(x)$ for $x\in S$. Then, we have
\[
\gamma
<
\langle\bar{\alpha},\bar{x}\rangle
\leq
\sup_{\alpha\in\Lambda}\langle\alpha,\bar{x}\rangle
=
h_2(\bar{x}),
\]
which contradicts $h_2(\bar{x})<\gamma$. This completes the proof.

\end{proof}

\section{Proof of Theorem~\ref{theo:extension-recovery}}\label{proof:theo-extension-recover}

The proof of Theorem~\ref{theo:extension-recovery} relies on two lemmas. The first one provides an upper norm bound for linear underestimators of function $f$ on a lattice family domain $X$. This bound is then used to show that the factor $(n+1)^2$ in the definition of the extension function $\bar{F}$ provides sufficient slack to guarantee its preservation property. As argued before Theorem~\ref{theo:extension-recovery}, we may assume throughout the proof that both the lattice family and its lifted domain admit representations by partial orders.

We first present the upper  bound on linear underestimators. Consider a function $f:X \subseteq \{0,1\}^n \to \R$, where $X$ is a lattice family admitting a partial order representation, that is, there exists a partial order $D$ such that $X = \{x \in \{0,1\}^n \mid x_i \geq x_j \for (i,j) \in D\}$.  Its all possible linear underestimators can be modeled as the following polyhedron:
\begin{equation}\label{eq:sublinear-polyhedron}
    Q_f:=\bigl\{\alpha\in\mathbb{R}^n \bigm| \langle\alpha,x\rangle\le f(x)
\ \text{for all }x\in X \bigr\}. 
\end{equation}
For an arbitrary $f$, an extreme point is determined by $n$ linearly independent active constraints, and Cramer's rule, together with the bound on entries of inverses of nonsingular $0$-$1$ matrices in~\cite{alon1997anti}, yields the exponential bound $2n(\sqrt{n}/2)^n\cdot \max\{|f^L|,|f^U|\}$ on the extreme points, where $f^L \leq f(x) \leq f^U$ for $x \in X$. The following lemma shows that the lattice  structure of $X$ improves this exponential bound to a quadratic one.

\begin{lemma}\label{lem:bounded-extreme-point-lattice}
Consider a nonempty  polyhedron $Q_f$ defined as in~\eqref{eq:sublinear-polyhedron}, where $f$ is an arbitrary function defined on a lattice family $X$ with a partial order representation. Let $f^L$ and $f^U$ be two bounds on $f$, that is, $f^L \leq f(x) \leq f^U$ for $x \in X$.  Then, for every extreme point $\alpha^*$ of $Q_f$,
\[
    \|\alpha^*\|_\infty
    \leq\max\bigl\{|f^U|, |f^U-n^2(f^U-f^L)|, n^2(f^U-f^L)\bigr\}.
\]
\end{lemma}

\begin{proof}
We start with showing that $Q_f$ in fact has at least one extreme point.  Let $\sigma$ be a permutation of coordinates that is consistent with the partial order $D$ of coordinates defining $X$.  For each $\ell\in[n]$, define $z^\ell:=\sum_{q=1}^{\ell}e_{\sigma(q)}$. Each $z^\ell$ satisfies all precedence constraints defining $X$, and hence $z^\ell\in X$. Moreover, the vectors $z^1,\ldots,z^n$ are linearly independent. Therefore, the normal vectors of constraints in~\eqref{eq:sublinear-polyhedron} span $\mathbb{R}^n$, so its feasible region has at least one extreme point~\cite[Theorem 2.6]{bertsimas1997introduction}.

Now, let $\alpha^*$ be  an extreme point of $Q_f$. The proof proceeds as follows. First, we construct a chain of $n$ feasible points $w^1, \ldots, w^n$ of $X$. Second, we derive upper and lower bounds on $\langle \alpha^*, w^t\rangle$ for $t \in [n]$. Last, the bound on   $\|\alpha^*\|_\infty$ follows by taking difference between consecutive terms.

We first  construct the desired chain by using tight constraints at $\alpha^*$. Since $\alpha^*$ is an extreme point of $Q_f\subseteq\mathbb{R}^n$, there exist $n$ linearly independent points $v^1,\ldots,v^n$ of $X$ such that for each $i\in [n]$ the inequality $\langle \alpha ,v^i\rangle \leq f(v^i)$ is tight at $\alpha^*$. It follows readily that the $n \times n$ matrix $V$ whose $i^{\text{th}}$ row is $v^i$ is nonsingular. For the $j^{\text{th}}$  column vector of the matrix $V$, we define its support as 
\[
    I_j:= \bigl\{i \in[n]\bigm| v^i_j=1 \bigr\}.
\]
Note that the nonsingularity of $V$ implies that  $I_j \neq \emptyset$ for $j \in [n]$ and $I_1, \ldots, I_n$ are pairwise distinct. Thus, strict set inclusion among the supports defines a strict partial order on the  index set $[n]$, that is $j\prec k$  if and only if $I_j\subsetneq I_k$. Let $\pi$ be a linear extension of this partial order, so that
\begin{equation}\label{eq:bounds-linearorder}
I_{\pi(j)}\subsetneq I_{\pi(k)} \quad\Longrightarrow\quad j<k.    
\end{equation}
With this ordering, we define a chain of $n$ points $w^1, \ldots, w^n$ as $w^t = \sum_{j = t}^n e_{\pi(j)}$ for $t \in [n]$.

To show that $w^1,\ldots,w^n$ belong to the lattice $X$, we express each $w^t$ as a finite sequence of componentwise maximum and minimum operations on  $(v^1,\ldots,v^n)$. For $t\in[n]$, define
\begin{equation}\label{eq:bounds-lattice-definition}
\tilde{w}^t
:=
\bigvee_{j=t}^n
\left(
\bigwedge_{i\in I_{\pi(j)}} v^i
\right),
\end{equation}
where recall that the meet and join are defined as componentwise minimum and maximum, respectively. Since each $I_{\pi(j)}$ is nonempty, this is a well-defined lattice expression involving the points $(v^1,\ldots,v^n)$. Because $v^1,\ldots,v^n \in X$ and $X$ is a lattice family, we have $\tilde{w}^t\in X$.  It remains to show that $\tilde{w}^t=w^t$.  To do so, we characterize each coordinate of $\tilde{w}^t$ as follows. Fix $t \in [n]$. For each $\pi(k)$ coordinate, we have
\begin{equation}\label{eq:bounds-equal-definition}
\tilde{w}^t_{\pi(k)}=1
\quad\Longleftrightarrow\quad
\exists j \geq t
\text{ such that }
I_{\pi(j)}\subseteq I_{\pi(k)}.
\end{equation}
This is because by the definition of supports, 
\[
\left(
\bigwedge_{i\in I_{\pi(j)}}v^i
\right)_{\pi(k)}=1
\quad\Longleftrightarrow\quad
I_{\pi(j)}\subseteq I_{\pi(k)}.
\]
Now, we verify two cases. The first case treats $k$ with $k \geq t$. In this case, we let $j = k$. Then $I_{\pi(j)} \subseteq I_{\pi(k)}$, and thus by~\eqref{eq:bounds-equal-definition}, we have $\tilde{w}^t_{\pi(k)} = 1$. Next, we consider $k$ with $k< t$. Suppose, by contradiction, that $\tilde{w}^t_{\pi(k)} = 1$. Then, by~\eqref{eq:bounds-equal-definition}, there exists $j \geq t$ such that $I_{\pi(j)} \subseteq I_{\pi(k)}$. Since $j \geq t > k$, we have $j \neq k$. Since the supports are pairwise distinct, the inclusion is strict, $I_{\pi(j)} \subsetneq I_{\pi(k)}$. By the definition of the linear extension in~\eqref{eq:bounds-linearorder}, this implies $j <k$, contradicting with $j \geq t>k$. Hence, $\tilde{w}^t_{\pi(k)}$ is $1$ if $k \geq t$ and is $0$ otherwise. In other words, $\title{w}^t = w^t$.

The next step is to  derive upper and lower bounds on $\langle\alpha^*,w^t\rangle$. Let $\Delta:=f^U-f^L$. An upper bound follows immediately from 
\[
\langle\alpha^*,w^t\rangle\leq f(w^t)\leq f^U \quad \for t \in [n].
\]
It remains to derive a lower bound. Consider any $z \in X$ that is obtained by taking componentwise maximum and minimum operations on elements in $\{v^1, \ldots, v^n\}$, and let $T$ denote the total number of occurrences of the vectors in the lattice expression. We prove that $\langle\alpha^*,z\rangle\geq f^U-T\Delta$ by induction on $T$. For the base case $T=1$, we have $z=v^i$ for some $i\in[n]$. Since the corresponding constraint is tight at $\alpha^*$, $\langle\alpha^*,z\rangle = f(v^i) \geq f^L = f^U-\Delta$. For the inductive step, suppose first that $ z=u\wedge v$, where the lattice expressions defining $u$ and $v$ contain $T_1$ and $T_2$ occurrences of the vectors \(v^1,\ldots,v^n\), respectively. Thus $T = T_1 + T_2$. Therefore, we have
\[
    \begin{aligned}
    \langle\alpha^*,z\rangle =  \langle\alpha^*,u\wedge v\rangle
    &=
    \langle\alpha^*,u\rangle
    +
    \langle\alpha^*,v\rangle
    -
    \langle\alpha^*,u\vee v\rangle\\
    &\geq
    (f^U-T_1\Delta)
    +(f^U-T_2\Delta)
    -f^U\\
    &=
    f^U-(T_1+T_2)\Delta,
    \end{aligned}
\]
the first equality follows from the identity $u+v=(u\wedge v)+(u\vee v)$, and the inequality follows from the induction hypothesis and $\langle\alpha^*,u\vee v\rangle\leq f(u\vee v)\leq f^U$. The case $z=u\vee v$ follows analogously, using $\langle\alpha^*,u\wedge v\rangle\leq f^U$. Since $w^t = \tilde{w}^t$, we now apply the above bound to the lattice representation in~\eqref{eq:bounds-lattice-definition}. Hence, the total number of occurrences of vectors $v^1, \ldots, v^n$ is at most $n^2$. Therefore, we can conclude 
\[
\langle \alpha^*, w^t \rangle \geq f^U -  n^2\Delta \quad \for t \in [n]. 
\]

Last, we compute bounds on coordinates of $\alpha^*$. For $t \in [n-1]$, 
\[
\alpha^*_{\pi(t)} = \langle \alpha^*, e_{\pi(t)} \rangle =  \langle \alpha^*, w^t \rangle - \langle \alpha^* w^{t+1} \rangle \leq f^U - (f^U - n^2\Delta) = n^2\Delta,
\]
where the second equality holds by the definition of $w^t$ and the inequality follows from the bounds on $\langle \alpha, w^t\rangle $. Similarly, we have 
\[
\alpha^*_{\pi(t)} =   \langle \alpha^*, w^t \rangle - \langle \alpha^* w^{t+1} \rangle \geq (f^U - n^2\Delta) - f^U= -n^2\Delta,
\]
Hence, $\vert \alpha^*_{\pi(t)} \vert \leq n^2\Delta$ for $t \in [n-1]$. It remains to bound the coordinate $\alpha^*_{\pi(n)}$. Since $w^n =e_{\pi(n)}$, we have $\alpha^*_{\pi(n)} = \langle \alpha^*, w^n \rangle$. Therefore, $f^U- n^2\Delta  \leq \alpha^*_{\pi(n)} \leq f^U$, showing  $|\alpha^*_{\pi(n)}|\leq\max\{|f^U|,|f^U-n^2\Delta|\}$. This gives the desired bound on $\|\alpha^*\|_\infty$.
\end{proof}

\begin{lemma}\label{lem:sublinear-envelope-extension}
Let $f: X \to \R$ be a function defined as $f(x):=\min_{k\in[K]}f_k(x)$ for $x \in X$, where for each $k$, $f_k$ is a function defined on a lattice family $X$ admitting a partial order representation. Assume that $\mathbf{0} \in X$ and $f(\mathbf{0}) \geq 0$.  For each $k\in[K]$, define the extension $\bar f_k:\{0,1\}^n\to\R$ by
\begin{equation*}
    \bar f_k(x):=f_k\bigl(\rho(x)\bigr)+C\|x-\rho(x)\|_1
    \quad\for x\in\{0,1\}^n,
\end{equation*}
where $\rho:\{0,1\}^n\to X$ is the retraction operator associated with $X$, constructed as in~\eqref{eq:retraction}, and 
\[
C \geq \max\bigl\{|f^U|, |f^U-n^2(f^U-f^L)|, n^2(f^U-f^L)\bigr\}
\]
with $f^L\leq f_k(x)\leq f^U$ for $x\in X$ and $k\in[K]$. Let $\bar{f}(x)=\min_{k\in[K]}\bar{f}_k(x)$. Then, $\sconv(f)(x)=\sconv(\bar f)(x)$ for $x\in\cone(X)$. If, in addition, $f^L \leq 0 \leq f^U$, then we can choose $C \geq n^2(f^U - f^L)$. 
\end{lemma}

\begin{proof}
The assumption $\mathbf{0} \in X$ and $f(\mathbf{0})\geq 0$ ensures $\sconv(f)$ and $\sconv(\bar{f})$ are well-defined.  Fix $x\in\cone(X)$. By Lemma~\ref{lemma:dualsconv},
\begin{subequations}
    \begin{align}
    \sconv(f)(x)
    &=
    \max\left\{
        \langle\alpha,x\rangle
        \;\middle|\;
        \langle\alpha,v\rangle\leq f_k(v)
        \ \for k\in[K],\ v\in X
    \right\},
    \label{eq:primal-sconv-f}\\
    \sconv(\bar f)(x)
    &=
    \max\left\{
        \langle\alpha,x\rangle
        \;\middle|\;
        \langle\alpha,v\rangle\leq \bar f_k(v)
        \ \for k\in[K],\ v\in\{0,1\}^{n}
    \right\}.
    \label{eq:primal-sconv-bar-f}
    \end{align}
\end{subequations}
Since $\bar f_k(v)=f_k(v)$ for every $v\in X$, every feasible solution of~\eqref{eq:primal-sconv-bar-f} is also feasible for~\eqref{eq:primal-sconv-f}. Hence, $\sconv(\bar f)(x)\leq\sconv(f)(x)$. It remains to prove the reverse inequality.

 Since the optimal value in~\eqref{eq:primal-sconv-f} is finite and, by Lemma~\ref{lem:bounded-extreme-point-lattice}, its feasible region has at least one extreme point, there exists an extreme point which is also optimal~\cite[Theorem 2.8]{bertsimas1997introduction}. Let $\alpha^*$ be such an optimal extreme point of~\eqref{eq:primal-sconv-f}. We next show the feasibility of $\alpha^*$ for~\eqref{eq:primal-sconv-bar-f}. This implies that $\sconv(f)(x)=\langle\alpha^*,x\rangle\leq\sconv(\bar f)(x)$, completing the proof. For any $v\in\{0,1\}^{n}$ and $k\in[K]$, we have
\[
    \begin{aligned}
    \langle\alpha^*,v\rangle
    &=
    \langle\alpha^*,\rho(v)\rangle
    +
    \langle\alpha^*,v-\rho(v)\rangle\\
    &\leq
    f_k(\rho(v))
    +
    \|\alpha^*\|_\infty\|v-\rho(v)\|_1\\
    &\leq f_k(\rho(v))
    + \max\bigl\{|f^U|,\,
    |f^U-n^2(f^U-f^L)|,\,
    n^2(f^U-f^L)\bigr\} \cdot\|v-\rho(v)\|_1\\
    &\leq
    f_k(\rho(v))
    +
    C\|v-\rho(v)\|_1
    =
    \bar f_k(v),
    \end{aligned}
\]
where the first inequality follows from $\rho(v)\in X$, the feasibility of $\alpha^*$ to~\eqref{eq:primal-sconv-f}, and  H\"{o}lder's inequality
which asserts that $\langle a,b\rangle\leq\|a\|_{\infty}\|b\|_1$ for $a,b\in\R^n$, and the second inequality holds by Lemma~\ref{lem:bounded-extreme-point-lattice} and $f = \min_{k\in [K]}f_k$.

 Last, suppose that $f^L \leq 0 \leq f^U$. Then $|f^U|\leq n^2(f^U-f^L)$ and $\bigl|f^U-n^2(f^U-f^L)\bigr| \leq n^2(f^U-f^L)$. Consequently, $\max\bigl\{|f^U|,\, |f^U-n^2(f^U-f^L)|,\, n^2(f^U-f^L)\bigr\} \leq n^2(f^U-f^L)$. Therefore, in this case, it suffices to choose $C \geq n^2(f^U-f^L)$.
\end{proof}

\begin{proof}[Proof of Theorem~\ref{theo:extension-recovery}]
Clearly, the submodularity of $\bar{F}_k$ follows from Theorem~\ref{them:extension-lattice}. To complete the proof, we derive the convex envelope of $f$ using the sublinear envelope of $\bar{F}$. Applying Lemma~\ref{lem:sublinear-envelope-extension} to the lattice family $\Lf$ and the functions $F_k$, $k\in[K]$, with their extensions $\bar F_k$, we obtain $\sconv(F)(t,x)=\sconv(\bar F)(t,x)$ for $(t,x)\in\cone(\Lf)$. Thus, for $x\in\conv(X)$, we obtain $\conv(f)(x)=\sconv(F)(1,x) = \sconv(\bar F)(1,x)$, where the first equality follows from Proposition~\ref{prop:pers-lift}, and the second equality holds since $(1,x) \in \cone(\Lf)$. This completes the proof. 
\end{proof}

\setcounter{theorem}{0}
\setcounter{proposition}{0}

\renewcommand{\thetheorem}{\Alph{section}.\arabic{theorem}}
\renewcommand{\theproposition}{\Alph{section}.\arabic{proposition}}

\section{Dilworth truncation of intersecting submodular functions}\label{App:lat_extension}
Let $X\subseteq \{0,1\}^n$ be an intersecting family and $f:X\to\R$ be an intersecting submodular function. We define the Dilworth truncation $\check f:\check X\to\R$ as in~\eqref{eq:dilworth-truncation}. In this appendix, we present three properties of the Dilworth truncation based on the results in~\cite{schrijver2003combinatorial}: preservation of sublinear envelope, a pre-order representation of $\check X$, and a strongly polynomial-time value oracle for $\check f$.

We first show that the Dilworth truncation preserves the associated extended polymatroid and, consequently, the sublinear envelope.
\begin{proposition}
\label{prop:lattice-extension-1}
Consider an intersecting submodular function $f:X\subseteq\{0, 1\}^n\to\R$ with either $\mathbf{0} \notin X$ or $f(\mathbf{0})\geq 0$. Let $\check f:\check X\to\R$ be its Dilworth truncation. Then $\EP_f=\EP_{\check f}$. Moreover,  $\sconv(f)(x) = \sconv(\check{f})(x)$ for $x \in \cone(X) =\cone(\check{X})$. 
\end{proposition}
\begin{proof}
We first prove $\EP_f = \EP_{\check{f}}$. To show $\EP_f\subseteq\EP_{\check f}$, we consider a point $\alpha \in\EP_f$. For any point $x \in \check{X}$ and  any $V\subseteq X$ satisfying $\sum_{v\in V}v=x$, we have $\langle\alpha,x\rangle
=\sum_{v\in V}\langle\alpha,v\rangle
\leq\sum_{v\in V}f(v)$, where the inequality holds since $\alpha \in \EP_f$. Thus, for any given $x \in \check{X}$, taking the minimum over all such $V$ gives $\langle\alpha,x\rangle\leq\check f(x)$. Hence, $\alpha\in\EP_{\check f}$, and therefore $\EP_f\subseteq\EP_{\check f}$. Conversely, let $\alpha\in\EP_{\check f}$. For any $x\in X$, the
singleton $V=\{x\}$ is feasible in the definition of $\check f(x)$, so
$\check f(x)\leq f(x)$. Thus
$\langle\alpha,x\rangle\leq\check f(x)\leq f(x)$ for every $x\in X$,
which implies $\alpha\in\EP_f$. Therefore,
$\EP_{\check f}\subseteq\EP_f$.

We next show that $\cone(X)=\cone(\check X)$. Since every $v\in X$ belongs to $\check X$ through the singleton decomposition $V=\{v\}$, we have
$X\subseteq\check X$,
and therefore $\cone(X)\subseteq\cone(\check X)$.
Conversely, every $x\in\check X$ admits a representation $x=\sum_{v\in V}v$
for some $V\subseteq X$. Hence $x\in\cone(X)$, which gives
$
\check X\subseteq\cone(X)$
and therefore $\cone(\check X)\subseteq\cone(X)$.
Thus, $\cone(X)=\cone(\check X)$.

It remains to prove equality of the sublinear envelopes. For $x \in \cone(X)$,
\[
\sconv(f)(x) = \max\bigl\{\langle \alpha, x\rangle \bigm| \alpha \in EP_f\bigr\} = \max \bigl\{\langle \alpha, x\rangle \bigm| \alpha \in EP_{\check{f}} \bigr\} = \sconv(\check{f})(x),
\]
where the first and third equalities follow from Lemma~\ref{lemma:dualsconv}, and the second equality holds due to $\EP_{f} = \EP_{\check{f}}$. 
\end{proof}
We next give a pre-order representation of $\check X$. For each $i\in[n]$, define the slice
$X_i:=\{x\in X\mid x_i=1\}$. Each nonempty $X_i$ is a lattice family. Suppose it is given by the pre-order representation
\begin{equation}\label{eq:intersecting-rep}
X_i=
\bigl\{
x\in\{0,1\}^n
\bigm|
x_j=0 \text{ for } j\in I_i^0,\;
x_j=1 \text{ for } j\in I_i^1,\;
x_j\geq x_\ell \text{ for }(j,\ell)\in D_i
\bigr\},
\end{equation}
where $D_i\subseteq[n]\times[n]$ is reflexive and transitive. For each nonempty $X_i$, let $\bot^i$ denote its least element. A pre-order representation of $\check X$ can be constructed directly from these least elements. Specifically, let
$I^0:=\{i\in[n]\mid X_i=\emptyset\}$ and
$D:=\{(i,j)\in[n]\times[n]\mid X_j\neq\emptyset,\ \bot_i^j=1\}
\cup\{(i,i)\mid X_i=\emptyset\}$. Then,
\begin{equation}
\label{eq:lattice-extension-preorder}
\check X=
\bigl\{
x\in\{0,1\}^n
\bigm|
x_i=0 \text{ for } i\in I^0,\;
x_i\geq x_j \text{ for }(i,j)\in D
\bigr\}.
\end{equation}
Here, $I^0$ collects the coordinates that are identically zero over $X$, and hence remain fixed at zero in $\check X$. The relation $D$ captures the structural dependencies among coordinates through the least elements $\bot^j$, with $\bot_i^j=1$ indicating that coordinate $i$ is forced whenever coordinate $j$ is active.

\begin{proposition}
\label{prop:lattice-extension-preorder}
Let $X\subseteq\{0,1\}^n$ be an intersecting family and define
$\check X:=\{\sum_{v\in V}v\mid V\subseteq X,\ \sum_{v\in V}v\leq\mathbf{1}\}$.
Then, the pre-order representation of $\check X$ is given by~\eqref{eq:lattice-extension-preorder}.
\end{proposition}

\begin{proof}
We first verify the representation. Let $x\in\check X$. By the definition of $\check X$, there exists $V\subseteq X$ such that
$x=\sum_{v\in V}v$ and $\sum_{v\in V}v\leq\mathbf{1}$. If $X_i=\emptyset$, then $v_i=0$ for every $v\in X$, and hence $x_i=0$. Now consider $(i,j)\in D$ with $X_j\neq\emptyset$. If $x_j=1$, then there exists $v\in V$ such that $v_j=1$. Therefore, $v\in X_j$ and $v\geq \bot^j$ which implies $x_i=1$. Therefore, $x_i\geq x_j$ for every $(i,j)\in D$, and thus $x$ satisfies all the constraints in~\eqref{eq:lattice-extension-preorder}.

Conversely, let $x\in\{0,1\}^n$ satisfy the constraints in~\eqref{eq:lattice-extension-preorder}. For any $j$ with $x_j=1$, we have $j\notin I^0$, and hence $X_j\neq\emptyset$. Moreover, if $\bot_i^j=1$, then $(i,j)\in D$, so $x_i\geq x_j=1$. Therefore, $\bot^j\leq x$ for every $j$ with $x_j=1$. Since $\bot_j^j=1$, it follows that
\[
x=\bigvee_{j:x_j=1}\bot^j \in \check{X},
\]
where the inclusion follows from that $\bot^j$ belongs to $X\subseteq\check X$ and $\check X$ is a lattice family.

It remains to verify that $D$ is reflexive and transitive. Reflexivity follows directly from its definition. For transitivity, first observe that, whenever $X_j\neq\emptyset$,
$\bot_i^j=1$ if and only if $X_j\subseteq X_i$. Indeed, if $\bot_i^j=1$, then every $x\in X_j$ satisfies $x_i=1$, and hence $X_j\subseteq X_i$. Conversely, if $X_j\subseteq X_i$, then every element of $X_j$, including $\bot^j$, has $i$th coordinate equal to one, so $\bot_i^j=1$.

Now suppose $(i,j),(j,k)\in D$. If $X_k=\emptyset$, then $(j,k)\in D$ implies $j=k$, and therefore $(i,k)=(i,j)\in D$. Otherwise, $X_k\neq\emptyset$, and $(j,k)\in D$ implies $X_k\subseteq X_j$. In particular, $X_j\neq\emptyset$, so $(i,j)\in D$ implies $X_j\subseteq X_i$. Hence $X_k\subseteq X_i$, which gives $\bot_i^k=1$ and therefore $(i,k)\in D$. Thus, $D$ is transitive, completing the proof.
\end{proof}
We now describe how to evaluate $\check f(x)$ for a fixed $x\in\check X$. Define the support of $x$ by $I_x:=\{j\in[n]:x_j=1\}$ and let $k=|I_x|$. For each $j\in I_x$, the least element $\bot^j$ is obtained directly from the representation of $X_j$. Then, the indices are processed in an order consistent with the inclusion relations among the vectors $\bot^{j}$, with each coefficient increased by the minimum residual slack so that feasibility is preserved and at least one extended polymatroid inequality becomes tight. The resulting tight inequalities define a vector $\alpha^\pi\in\mathbb{R}^n$ satisfying $\langle \alpha^\pi,x\rangle=\check f(x)$, as formalized in Algorithm~\ref{alg:evaluate_lattice_extension}.

Moreover, the algorithm runs in strongly polynomial time. In particular, the optimization problem in Line~9 of Algorithm~\ref{alg:evaluate_lattice_extension} is a submodular minimization problem over a lattice family: any two feasible points have the common nonzero coordinate $\pi(i)$, so the intersecting property of $X$ ensures that their componentwise minimum and maximum remain feasible, while the objective remains submodular since it is obtained from $f$ by adding a linear function. Therefore, this problem can be solved in strongly polynomial time~\cite{lee2015faster,orlin2009faster,schrijver2000combinatorial}. Since the algorithm performs at most $n$ such minimizations, its overall running time is strongly polynomial. The following proposition establishes its correctness.
\begin{proposition}
\label{prop:greedy_evaluation}
Consider an intersecting submodular function $f:X\subseteq\{0, 1\}^n\to\R$ with either $\mathbf{0} \notin X$ or $f(\mathbf{0}) \geq 0$. Let $\check f:\check X\to\R$ be defined in~\eqref{eq:dilworth-truncation}. For any $x\in\check{X}$, let $\alpha^\pi$ be the weight vector constructed by Algorithm~\ref{alg:evaluate_lattice_extension}. Then, 
\begin{equation*} 
\langle \alpha^\pi,x\rangle=\check{f}(x).
\end{equation*} 
\end{proposition}
\begin{proof}
Algorithm~\ref{alg:evaluate_lattice_extension} is the vector-form counterpart of the greedy procedure in Section~49.7 of~\cite{schrijver2003combinatorial}. It iteratively updates the
coefficients until the resulting tight sets certify that the constructed vector attains the Dilworth truncation value at $x$. Therefore, the optimality in Section~49.7 of~\cite{schrijver2003combinatorial} applies directly and yields $\langle \alpha^\pi, x \rangle = \check{f}(x)$.
\end{proof}

\begin{algorithm}[htbp]
\caption{Evaluation of the extended lattice submodular function
$\check{f}(x)$}
\label{alg:evaluate_lattice_extension}
\begin{algorithmic}[1]
\Require
    An intersecting family $X \subseteq \{0,1\}^n$;
    an intersecting submodular function $f:X\to\R$;
    a vector $x\in\check{X}$;
    the least elements $\bot^j$ for each $j\in I_x$.
\Ensure
    Value $\check{f}(x)$.

\If{$x=\mathbf{0}$}
    \State \Return $0$
\EndIf

\State Initialize $\alpha^\pi\gets\mathbf{0}\in\R^n$
\State Initialize $s^\pi_0\gets\mathbf{0}\in\{0,1\}^n$
\State Let $\pi(1),\ldots,\pi(k)$ be an ordering of $I_x$ such that
$\bot^{\pi(j)}<\bot^{\pi(i)}$ implies $j<i$

\For{$i=1,\ldots,k$}
    \State $s^\pi_i\gets s^\pi_{i-1}\vee \bot^{\pi(i)}$

    \State
    $\displaystyle
    \delta_i\gets
    \min_{u}
    \left\{
        f(u)-\langle\alpha^\pi,u\rangle
        \,\middle|\,
        u\in X,\;
        u_{\pi(i)}=1,\;
        u\le s^\pi_i
    \right\}$

    \State
    $\alpha^\pi_{\pi(i)}
    \gets
    \alpha^\pi_{\pi(i)}+\delta_i$
\EndFor

\State \Return $\langle\alpha^\pi,x\rangle$
\end{algorithmic}
\end{algorithm}

\section{Proof of Proposition~\ref{prop:sep_intersecting}}\label{proof:sep_intersecting}
We start with discussing how a value oracle of $\bar{f}_k$ can be built. Here, we assume that the intersecting family domain $X$ admits a representation given as in~\eqref{eq:intersecting-rep}. Then, it follows from Proposition~\ref{prop:lattice-extension-preorder} that the lattice family $\check{X}$ admits a pre-order representation, yielding a value oracle of the retraction $\rho$. Moreover, Proposition~\ref{prop:greedy_evaluation} provides a value oracle of $\check{f}_k$. These together lead to a value oracle of $\bar{f}_k$ . 

\begin{proof}
Let $\bar{f}(x):=\min\{\bar{f}_1(x),\bar{f}_2(x)\}$. Then, we observe that the separation problem of $\sconv(\bar{f})$ can be solved using Proposition~\ref{prop:sep_sublinear}. This is because $\bar{f}_k(\mathbf{0}) = \check{f}_k(\mathbf{0}) = 0$, and, by Theorem~\ref{them:extension-lattice}, the submodularity of $\bar{f}_k$ on $\{0,1\}^n$ follows from that of $\check{f}_k$ on $\check{X}$.  

Next, we note that, for $x \in \cone(X)$, $\sconv(f)(x) = \sconv(\bar{f})(x)$. To see this, for any $\bar x\in\cone(X)$, we have
\begin{align*}
\sconv(f)(\bar x)
&=
\max\left\{
    \alpha^\intercal \bar x
    \,\middle|\,  \alpha\in \EP_{ f_1}\cap \EP_{ f_2}
\right\} \\
&=
\max\left\{
    \alpha^\intercal \bar x
    \,\middle|\,
    \alpha\in \EP_{\check f_1}\cap \EP_{\check f_2}
\right\} \\
&=
\sconv(\check f)(\bar x) =
\sconv(\bar f)(\bar x),
\end{align*}
where the first and third equalities follow from Lemma~\ref{lemma:dualsconv}, the second equality follows from
$\EP_{f_k}=\EP_{\check f_k}$ for $k=1,2$, which is established in Proposition~\ref{prop:lattice-extension-1}, and the last equality follows due to Lemma~\ref{lem:sublinear-envelope-extension} and $\check{f}_k(\mathbf{0}) = 0$.

Now, for a given $(\bar{x},\bar{\mu}) \in \Q^{n+1}$ with $\bar{x} \in \cone(X)$, we call the separation oracle of $\bar{f}$. If $\bar{\mu} \geq \sconv(\bar{f})(\bar{x})$ then we can conclude that $\bar{\mu} \geq \sconv(f)(\bar{x})$. Otherwise, the valid inequality $\mu \geq \langle \alpha, x\rangle $ returned by the oracle is also valid for the epigraph of $\sconv(f)$ as $\sconv(f)(x) = \sconv(\bar{f})(x)$ on $x \in \cone(X)$, and violates $(\bar{x},\bar{\mu})$. 
\end{proof}

\section{Separation for Bilinear Functions}\label{App:sep-cycle}

Consider the bilinear function
\[
f_G(x)=\sum_{(i,j)\in E}a_{ij}x_ix_j \quad \for x\in \{0, 1\}^n
\]
associated with a graph or subgraph \(G=(V,E)\). Given a fractional point
\((\bar{x},\bar{\mu})\in[0,1]^n\times\mathbb{R}\), the separation procedure is described in Algorithm~\ref{alg:separation_graph}. Our procedure greedily selects a collection of edge-disjoint cycles and applies the cycle inequalities derived in \eqref{eq:cycle-env}. All edges that are not assigned to a selected cycle are treated individually using their McCormick envelopes. Since each edge appears in exactly one component of the resulting decomposition, the component-wise affine underestimators can be aggregated into a valid inequality for \(f_G\).

To prioritize the candidate cycles, we assign a score to each edge. For \((i,j)\in E\), define
\begin{equation}
\label{eq:edge_violation}
v_{ij}
=
\begin{cases}
a_{ij}\bigl(\bar{x}_i\bar{x}_j-
\max\{0,\bar{x}_i+\bar{x}_j-1\}\bigr),
& \text{if }a_{ij}>0,\\[1mm]
a_{ij}\bigl(\bar{x}_i\bar{x}_j - \min\{\bar{x}_i,\bar{x}_j\}\bigr),
& \text{if }a_{ij}\leq 0.
\end{cases}
\end{equation}
The quantity \(v_{ij}\) is the gap between the bilinear term and its single-edge McCormick envelope at \(\bar{x}\). For a cycle \(C\) with edges $E(C)$, we define $\operatorname{score}(C) = \sum_{(i,j)\in E(C)}v_{ij}$. The cycles are processed in non-increasing order of their scores. This ordering is used only as a heuristic and does not necessarily coincide with the ordering induced by the actual cycle-cut violations.

\begin{algorithm}[htbp]
\caption{Cycle-based separation for a bilinear function}
\label{alg:separation_graph}
\begin{algorithmic}[1]
\Require Graph \(G=(V,E)\); weights \(a_{ij}\); cycles \(\mathcal{C}\) ordered by non-increasing score; point \((\bar{x},\bar{\mu})\).
\Ensure A violated inequality defined by \((\alpha,\beta)\), or \(\emptyset\).

\State Set \(\alpha\gets 0\), \(\beta\gets 0\), and \(U\gets\emptyset\)

\For{\(C\in\mathcal{C}\)}
    \If{\(E(C)\cap U\neq\emptyset\)}
        \State \textbf{continue}
    \EndIf
    \State \((\alpha^C,\beta^C)\gets\) Algorithm~\ref{alg:separation_cycle}\((C,\bar{x})\)
    \If{\((\alpha^C,\beta^C)\neq\emptyset\)}
        \State \(\alpha_i\gets\alpha_i+\alpha_i^C\) for all \(i\in V(C)\)
        \State \(\beta\gets\beta+\beta^C\)
        \State \(U\gets U\cup E(C)\)
    \EndIf
\EndFor

\For{\((i,j)\in E\setminus U\)}
    \If{\(a_{ij}<0\)}
        \If{\(\bar{x}_i\leq\bar{x}_j\)}
            \State \(\alpha_i\gets\alpha_i+a_{ij}\)
        \Else
            \State \(\alpha_j\gets\alpha_j+a_{ij}\)
        \EndIf
    \ElsIf{\(a_{ij}>0\) and \(\bar{x}_i+\bar{x}_j>1\)}
        \State \(\alpha_i\gets\alpha_i+a_{ij}\), \(\alpha_j\gets\alpha_j+a_{ij}\)
        \State \(\beta\gets\beta-a_{ij}\)
    \EndIf
\EndFor

\If{\(\alpha^\intercal\bar{x}+\beta>\bar{\mu}\)}
    \State \Return \((\alpha,\beta)\)
\EndIf
\State \Return \(\emptyset\)
\end{algorithmic}
\end{algorithm}

For a residual edge with \(a_{ij}<0\), the algorithm selects an active affine piece of
\(a_{ij}\min\{x_i,x_j\}\). For an edge with \(a_{ij}>0\), it selects an active affine piece of
\(a_{ij}\max\{0,x_i+x_j-1\}\). Therefore, every residual-edge inequality is valid. Together with the cycle inequalities, edge-disjointness ensures the validity of the inequality in Algorithm~\ref{alg:separation_graph}.

We next describe the separation procedure for an individual cycle. Consider $f_C(x) = \sum_{i=1}^{n-1}a_ix_ix_{i+1}+a_nx_1x_n$. A switching transformation replaces \(x_i\) by \(1-x_i\) for selected indices. The switching set is constructed sequentially so that the transformed coefficients of the first \(n-1\) edges are negative. If the remaining coefficient is also negative, the switched function is submodular and no cycle-specific inequality is required. Otherwise, the switched function contains exactly one positive quadratic coefficient, and \eqref{eq:cycle-env} can be applied. The implementation details are provided in Algorithm~\ref{alg:separation_cycle}.

\begin{algorithm}[htbp]
\caption{Separation for a cycle bilinear function}
\label{alg:separation_cycle}
\begin{algorithmic}[1]
\Require Cycle function \(f_C(x)=\sum_{i=1}^{n-1}a_ix_ix_{i+1}+a_nx_1x_n\), where \(a_i\neq 0\) for \(i\in[n]\); point \(\bar{x}\).
\Ensure An affine underestimator defined by \((\alpha,\beta)\), or \(\emptyset\).

\State Set \(S\gets\emptyset\)

\For{\(i\in[n-1]\)}
    \If{\((a_i<0\text{ and }i\in S)\) or \((a_i>0\text{ and }i\notin S)\)}
        \State \(S\gets S\cup\{i+1\}\)
    \EndIf
\EndFor

\State Define \(T_i(y)=1-y\) if \(i\in S\), and \(T_i(y)=y\) otherwise
\State Set \(\bar{y}_i\gets T_i(\bar{x}_i)\) for all \(i\in[n]\)
\State Expand \(g(y)\gets f_C(T(y))=\sum_{i=1}^{n-1}\widetilde{a}_iy_iy_{i+1}+\widetilde{a}_ny_1y_n+b^\intercal y+c\)

\If{\(\widetilde{a}_n<0\)}
    \State \Return \(\emptyset\)
\EndIf

\State Set \(\widehat{a}_i\gets\widetilde{a}_i/\widetilde{a}_n\) for all \(i\in[n-1]\)
\State Define \(h(y) = \sum_{i=1}^{n-1}\widehat{a}_iy_iy_{i+1}+y_1y_n\)
\State Obtain an active affine piece \(\widehat{\alpha}^{\intercal}y+\widehat{\beta}\) of \(\conv(h)\) at \(\bar{y}\) using \eqref{eq:cycle-env}
\State Set \(\alpha'\gets\widetilde{a}_n\widehat{\alpha}+b\) and \(\beta'\gets\widetilde{a}_n\widehat{\beta}+c\)

\For{\(i\in[n]\)}
    \State Set \(\alpha_i\gets-\alpha'_i\) if \(i\in S\), and \(\alpha_i\gets\alpha'_i\) otherwise
\EndFor

\State Set \(\beta\gets\beta'+\sum_{i\in S}\alpha'_i\)
\State \Return \((\alpha,\beta)\)

\end{algorithmic}
\end{algorithm}

\end{appendices}

\bibliographystyle{spmpsci}      
\bibliography{reference} 

\end{document}